\documentclass{amsart}

\usepackage{subfigure}
\usepackage{amsmath,amssymb, amsthm,amscd}
\usepackage{enumerate}
\usepackage{fancybox}
\usepackage{graphicx,color}
\usepackage{amsfonts}
\usepackage{multicol}
\usepackage[normalem]{ulem} 
\usepackage{cleveref}
\usepackage[all]{xy}
\usepackage{amscd}

\numberwithin{equation}{section}

\newcommand{\imunit}{\sqrt{-1}}

\newcommand{\V}[1]{\vspace{#1}}
\newcommand{\id}{\operatorname{id}}

\newcommand{\PP}{\mathbb{P}}
\newcommand{\C}{\mathbb{C}}
\newcommand{\R}{\mathbb{R}}
\newcommand{\Z}{\mathbb{Z}}
\newcommand{\Pa}{\partial}

\newcommand{\Int}{\operatorname{Int}}
\newcommand{\Crit}{\operatorname{Crit}}
\newcommand{\Critv}{\operatorname{Critv}}
\newcommand{\Ker}{\operatorname{Ker}}

\renewcommand{\Re}{\operatorname{Re}}
\renewcommand{\Im}{\operatorname{Im}}

\newcommand{\std}{\operatorname{std}}

\newcommand{\Hom}{\operatorname{Hom}}

\newcommand{\Aut}{\operatorname{Aut}}

\newcommand{\Hir}{\mathbb{F}}
\newcommand{\disk}{\mathbb{D}}
\newcommand{\sph}{\mathbb{S}}
\newcommand{\Emb}{\operatorname{Emb}}
\newcommand{\PD}{\operatorname{PD}}
\newcommand{\Coker}{\operatorname{Coker}}

\theoremstyle{plain}
\newtheorem{theorem}{Theorem}[section]
\newtheorem{corollary}[theorem]{Corollary}
\newtheorem{lemma}[theorem]{Lemma}
\newtheorem{proposition}[theorem]{Proposition}
\newtheorem{problem}[theorem]{Problem}

\theoremstyle{definition}
\newtheorem{definition}[theorem]{Definition}

\newtheorem{conjecture}[theorem]{Conjecture}
\newtheorem{remark}[theorem]{Remark}

\title{Combinatorial construction of symplectic 6-manifolds via bifibration structures}
\author{Kenta Hayano}
\email{k-hayano@math.keio.ac.jp}
\address{Department of Mathematics, Faculty of Science and Technology, Keio University,
Yagami Campus, 3-14-1, Hiyoshi, Kohoku-ku, Yokohama, 223-8522, Japan}

\thanks{
The author was supported by Japanese Society for the Promotion of Science KAKENHI Grant Numbers JP22K03320 and JP23K03123.
}

\begin{document}

\begin{abstract}

A bifibration structure on a $6$-manifold is a map to either the complex projective plane $\PP^2$ or a $\PP^1$-bundle over $\PP^1$, such that its composition with the projection to $\PP^1$ is a ($6$-dimensional) Lefschetz fibration/pencil, and its restriction to the preimage of a generic $\PP^1$-fiber is also a ($4$-dimensional) Lefschetz fibration/pencil. 
This object has been studied by Auroux, Katzarkov, Seidel, among others.
From a pair consisting of a monodromy representation of a Lefschetz fibration/pencil on a $4$-manifold and a relation in a braid group, which are mutually compatible in an appropriate sense, we construct a bifibration structure on a closed symplectic $6$-manifold, producing the given compatible pair as its monodromies.
We further establish methods for computing topological invariants of symplectic $6$-manifolds, including Chern numbers, from compatible pairs. 
Additionally, we provide an explicit example of a compatible pair, conjectured to correspond to a bifibration structure derived from the degree-$2$ Veronese embedding of the $3$-dimensional complex projective space. 
This example can be viewed as a higher-dimensional analogue of the lantern relation in the mapping class group of the four-punctured sphere.
Our results not only extend the applicability of combinatorial techniques to higher-dimensional symplectic geometry but also offer a unified framework for systematically exploring symplectic $6$-manifolds.

\end{abstract}

\maketitle

\section{Introduction}

The construction of symplectic structures from Lefschetz fibrations/pencils established by Gompf \cite{GompfTowardTopcharasympmfd,Gompf2005localholyieldsymp} (see also \cite{GS19994-mfd}), together with the correspondence between Lefschetz fibrations/pencils and their monodromies shown in \cite{Kas1980handlebodyLF,Matsumoto1996LFgenus2,BaykurHayanoHuriwitzMultisection}, has enabled the construction of symplectic $4$-manifolds from relations in the mapping class groups of surfaces.
This combinatorial approach has yielded various intriguing examples of symplectic 4-manifolds, such as those obtained in \cite{EMV2011substitution,BH2016multisection,AM2020genus2LF}.
Conversely, for a given symplectic manifold, Donaldson \cite{Donaldson1999LPonsympmfd} constructed a Lefschetz pencil on it by perturbing a map to the complex projective space $\PP^1$ obtained as the ratio of two sections of a line bundle with certain properties, which are later referred to as \textit{approximately holomorphic sections} in \cite{Auroux2000symp4mfdasbranchedcover,AurouxKatzarkovbranchedcoversymp4mfd,AurouxSympplecticMapsInvariant}. 
Auroux and Katzarkov \cite{Auroux2000symp4mfdasbranchedcover,AurouxKatzarkovbranchedcoversymp4mfd,AurouxSympplecticMapsInvariant} further developed this idea: they used three approximately holomorphic sections to construct a map from a higher-dimensional symplectic manifold to $\PP^2$ with well-behaved critical point/value sets, which we call a \textit{Lefschetz bipencil} in this manuscript.
In particular, by examining the monodromies of this map on a symplectic $6$-manifold, one can obtain a pair consisting of a monodromy representation of a Lefschetz pencil on a 4-manifold and a relation in a braid group, which are compatible in an appropriate sense (see \Cref{def:compatibility condition} for further details).

This study aims to lay groundwork for investigating symplectic 6-manifolds from the aforementioned pairs via bifibration structures, specifically as follows: 

\begin{enumerate}[(I)]

\item 
To clarify the correspondence between bifibration structures on symplectic $6$-manifolds and compatible pairs, including both existence and uniqueness (under a suitable equivalence). 

\item 
To establish a method for computing the topological invariants of symplectic 6-manifolds, such as the fundamental group, the (co)homology groups, and the Chern classes/numbers, from compatible pairs.

\item 
To present examples of compatible pairs that potentially serve as building blocks for other compatible pairs, ideally as simple as possible yet non-trivial, such as the lantern, chain, and torus relations in mapping class groups of surfaces.

\end{enumerate}

In order to deal with broader classes of compatible pairs than those associated with Lefschetz bipencils appearing in \cite{AurouxKatzarkovbranchedcoversymp4mfd,AurouxSympplecticMapsInvariant}, we consider a \textit{Lefschetz bifibration} on a $6$-manifold, which is a map to the total space of a $D^2$- or $\PP^1$-bundle $Z$ such that the composition with the bundle projection $\rho:Z\to E$ is a Lefschetz fibration, and its restriction on the preimage of a generic fiber $\rho^{-1}(\ast)$ is also a Lefschetz fibration (see \Cref{def:Lefschetz bifibration} for the precise definition). 
Note that the concept of Lefschetz bifibrations has already been introduced in \cite{SeidelbookFukayaPicard} and relevant references with additional geometric structures. 
As in the case of Lefschetz fibrations/pencils, one can obtain a Lefschetz bifibration from a Lefschetz bipencil by blowing-up the target $\PP^2$ and the source manifold.

We first address the existence part of the purpose (I). 
Specifically, under some mild assumptions, we construct a Lefschetz bifibration $F:X\to \Hir_m$ on a closed symplectic $6$-manifold $X$ (where $\Hir_m$ is the Hirzebruch surface with the $(-m)$-section) from a compatible pair of a monodromy representation $\theta$ of a Lefschetz fibration and a relation $\tau_{\mu_1}^{\epsilon_1}\cdots \tau_{\mu_n}^{\epsilon_n}=t_\Pa^m$ in the braid group. (\Cref{thm:combinatorial construction LbF over Hir_m,thm:almost cpx str LbF}.) 
Suppose that $\theta$ is the monodromy representation of a Lefschetz fibration with $b$ sections.
The corresponding Lefschetz bifibration $F:X\to \Hir_m$ also has $b$ sections, which are naturally identified with $\Hir_m$ via $F$, in particular admit $\PP^1$-bundle structures. 
If some of sections of a Lefschetz fibration associated with $\theta$ have self-intersection $-1$, the corresponding sections of $F$ can be blown-down along fibers (as a $\PP^1$-bundle).
%
%
On the other hand, the preimage of the $(-m)$-section of $\Hir_m$ by $F$ admits a trivial $\PP^1$-bundle over a genus-$g$ surface. 
If $m=1$, this $\PP^1$-bundle can be blown-down along fibers. 
In both cases, the resulting $6$-manifold admits a \textit{pencil-fibration structure}, and a symplectic structure. (\Cref{thm:combinatorial construction pencil-fibration,prop:symplectic structure blowdowns}. See \Cref{def:other pencil-fibration structures} for the definitions of various pencil-fibration structures.)

Part of the theorems presented above (for Lefschetz bipencils on symplectic manifolds) was stated at the second paragraph of \cite[p.~33]{AurouxSympplecticMapsInvariant}, though proofs were not provided. 
In the construction of Lefschetz bifibrations, we use a criterion for an element of a braid group, considered as an isotopy class of a self-diffeomorphism of $D^2$ fixing a finite subset $R\subset \Int(D^2)$, to be lifted to a self-isomorphism of a Lefschetz fibration over the disk whose critical value set is $R$. (\Cref{thm:condition lift a braid}.) 
This criterion was also stated in \cite[p.~35]{AurouxSympplecticMapsInvariant}, again without proofs (and any technical assumptions, cf.~\Cref{rem:assumptions for lifing a braid}).

As for the second purpose (II), it is straightforward to compute the Euler characteristic, the fundamental group, and the (co)homology group from a given compatible pair, by examining the handle decomposition associated with the Lefschetz fibration obtained as the composition of the Lefschetz bifibration and the bundle projection on the target space. (\Cref{prop:Euler LbF,prop:pi1 LbF,prop:H2 LbF}.)
Utilizing an almost complex structure constructed in \Cref{thm:almost cpx str LbF}, we show that the Poincar\'{e} duals of the sets of the critical points and the cusps of a Lefschetz bifibration are respectively represented by the Thom polynomials for folds and cusps. (\Cref{thm:crit pt set Chern class}.)
As a corollary, one can compute the Chern number $c_1c_2$ of a symplectic $6$-manifold from the corresponding compatible pair (by counting the number of cusps, see \Cref{cor:explicit description cusp [crit]}). 
Note that the methods for computing topological invariants explained here are also valid for other pencil-fibration structures. (\Cref{prop:blowdown topinv 1,prop:blowdown topinv 2,prop:blowdown topinv 3}.)

For the last purpose (III), we present an example of a compatible pair of a monodromy representation of a genus-$1$ Lefschetz pencil and a relation (a factorization of $t_{\Pa}$ into a product of half twists) in a braid group (\Cref{thm:relation for deg2 Veronese}). 
Building on the results in the preceding paragraph, we further demonstrate the computation of topological invariants of the corresponding symplectic $6$-manifold.
The compatible pair presented here can potentially be regarded, in a sense, as a generalization of the lantern relation in the mapping class group of the four-punctured sphere. 
Indeed, the lantern relation arises as the monodromy factorization of a Lefschetz pencil obtained as the composition of the degree-$2$ Veronese embedding $\PP^2 \hookrightarrow \PP^5$ and a generic projection $\PP^5 \dashrightarrow \PP^1$. 
By analogy, the given compatible pair is conjectured to correspond to the monodromies of a Lefschetz bipencil derived from the degree 2 Veronese embedding of $\PP^3$. 
Refer to \Cref{rem:example is deg2 Veronese} for further details (and \cite{Torricelli2020generalizedlantern_preprint} for generalized lantern relations).

The paper is organized as follows. 
In \Cref{sec:fibration str symp mfd} explains the definitions of Lefschetz fibrations/pencils, bifibrations, other pencil-fibration structures, and basic properties of their monodromies. 
The theorems for combinatorial constructions of Lefschetz bifibrations and pencil-fibration structures are stated in this section. (\Cref{thm:combinatorial construction LbF over D2timesD2,thm:combinatorial construction LbF over Hir_m,thm:combinatorial construction pencil-fibration}.) 
In \Cref{sec:liftable briad}, we give a criterion for an element in a braid group to be lifted to a self-isomorphism of a Lefschetz fibration. 
\Cref{sec:combinatorial construction pencil-fibration} is devoted to the proofs of the construction theorems, in which the aforementioned criterion for lifting a braid is utilized. 
In \Cref{sec:cpx/symp str LbF}, we give almost complex and symplectic structures of the total spaces of Lefschetz bifibrations and pencil-fibration structures. (\Cref{thm:almost cpx str LbF,prop:symplectic structure blowdowns}.)
In \Cref{sec:top inv LbF}, we explain how to compute topological invariants of symplectic $6$-manifolds obtained via the combinatorial constructions. 
In \Cref{sec:example}, we give an example of a compatible pair of a monodromy representation of a Lefschetz pencil and a relation in a braid group.

\subsection*{Conventions and notations}

Throughout the paper, we assume that manifolds are smooth, compact, oriented and connected unless otherwise noted. 
We denote the $n$-dimensional complex projective space by $\PP^n$, and the unit sphere in $\C^n$ by $\sph^{2n-1}$. 
We put $\disk_r =\{z\in \C~|~|z|\leq r\}$ for $r>0$ and $D^2=\disk_1$. 
For a smooth map $f:X\to Y$ between manifolds, we denote the set of critical points and values of $f$ by $\Crit(f)\subset X$ and $\Critv(f)\subset Y$, respectively. 
We use the dashed arrow $\dashrightarrow$ to represent a map defined on the complement of a closed submanifold with positive codimension (e.g.~a Lefschetz pencil).

Let $X$ be a manifold and $P,S\subset X$.
We denote by $\mathcal{D}_P(X,S)$ the group of orientation-preserving self-diffeomorphisms which preserve $P$ (resp.~$S$) pointwise (resp.~setwise). 
We denote the mapping class group $\pi_0(\mathcal{D}_P(X,S))$ by $\mathcal{M}_P(X,S)$. 
The symbols $P$ and $S$ are omitted when they are the empty set. 
We put $\Pa$ instead of $P$ when $P=\Pa X$ (e.g.~$\mathcal{M}_\Pa(\disk_1,P)$ in page \pageref{page:braid group}). 
We define multiplication of mapping class groups to be opposite to composition of representatives, that is, $[\xi_1]\cdot [\xi_2] =[\xi_2\circ \xi_1]$ for diffeomorphisms $\xi_1,\xi_2$, in order to make group structures of mapping class groups consistent to those of braid groups, and also make monodromy representations homomorphisms.

\section{Fibration structures on symplectic manifolds}\label{sec:fibration str symp mfd}

\subsection{Lefschetz fibrations, pencils and their monodromies}\label{subsec:LF/LP}

In this subsection we briefly review Lefschetz fibrations (over $D^2$ or $\PP^1$), pencils, and their monodromy representations. 

\begin{definition}

Let $X$ be a compact $2n$-manifold ($n\geq 2$) and $E$ be either $D^2$ or $\PP^1$. 
A \emph{Lefschetz fibration} on $X$ is a smooth map $f:X\to E$ satisfying the following conditions.

\begin{enumerate}[(1)]

\item
$\Crit(f)$ is contained in $\Int (X)$.
For any $x\in \Crit(f)$, one can take an orientation-preserving complex chart $\varphi:U\to \C^n$ (resp.~$\psi:V\to \C$) around $x$ (resp.~$f(x)$) satisfying
\[
\psi\circ f\circ \varphi^{-1}(z_1,\ldots, z_n) = z_1^2+\cdots +z_n^2.
\]

\item
Each fiber of $f$ contains at most one critical point. 

\item 
The boundary $\Pa X$ admits a decomposition into two codimension-$0$ submanifolds $\Pa_v X= f^{-1}(\Pa E)$ and $\Pa_h X = \overline{\Pa X \setminus \Pa_v X}$, called the \textit{vertical boundary} and the \textit{horizontal boundary} of $X$ with respect to $f$.

\item 
The restriction $f:\Pa_h X\to E$ is a $(2n-3)$-manifold bundle.

\end{enumerate}

\noindent
In general, a critical point of a smooth map admitting the local model above is called a \emph{Lefschetz singularity}. 
Let $f_i:X_i\to E$ be a Lefschetz fibration ($i=0,1$), and $\Phi:X_0\to X_1$ and $\Psi:E\to E$ be diffeomorphisms.
The pair $(\Phi,\Psi)$ is called an \textit{isomorphism} from $f_0$ to $f_1$ if $\Psi\circ f_0 = f_1\circ \Phi$ holds. 

\end{definition}

%
%
%
%
%
%

Let $f:X\to D^2$ be a Lefschetz fibration over the disk and $q_0\in D^2\setminus \Critv(f)$. 
We denote the regular fiber $f^{-1}(q_0)$ by $W$. 
Take a trivialization $\Pa_h X \cong D^2\times \Pa W$ of a $(2n-3)$-manifold bundle $f:\Pa_hX\to D^2$ and a horizontal distribution $\mathcal{H}\subset TX\setminus \Crit(f)$ for $f$ so that $\mathcal{H}$ is equal to $\{0\}\oplus (TD^2 \oplus\{0\})$ in $\nu \Pa_h X$ under the identification $\nu\Pa_h X \cong (-1,0]\times(D^2\times \Pa W)$.  
One can define the \emph{monodromy representation}
\[
\theta_f:\pi_1(D^2\setminus \Critv(f),q_0) \to \mathcal{M}_\Pa(W)
\]
by taking a parallel transport with respect to $\mathcal{H}$. 

A system of simple paths $\gamma_1,\ldots, \gamma_m\subset D^2$ is called a \emph{Hurwitz system} of $f$ (from $q_0$) if it satisfies the following conditions.

\begin{enumerate}[(A)]\label{list:condition Hurwitz system}

\item
$\gamma_1,\ldots,\gamma_m$ are mutually disjoint except at the common initial point $q_0$.

\item
$\gamma_1,\ldots, \gamma_m$ appear in this order when one goes around $q_0$ counterclockwise.

\item
Each point in $\Critv(f)$ is an end of one of $\gamma_i$'s.

\item
Each $\gamma_i$ is away from $\Critv(f)$ except at one end of it.

\end{enumerate}

\noindent
For each path $\gamma_i$ in a Hurwitz system of $f$, we take an element $\alpha_i\in \pi_1(D^2\setminus\Critv(f),q_0)$ as follows: $\alpha_i$ is represented by a loop going along $\gamma_i$ from $q_0$, going around the end of $\alpha_i$ counterclockwise, and then going back to $q_0$ along $\gamma_i$. 
Such a loop or the corresponding element in $\pi_1(D^2\setminus\Critv(f),q_0)$ is called a \emph{meridian loop} of $\gamma_i$. 
It is known that the monodromy $\theta_f(\alpha_i)$ is the (generalized) \emph{Dehn twist} $t_{C_i}$ along an $(n-1)$-sphere $C_i\subset W$. 
The sphere $C_i$ or its isotopy class is called a \emph{vanishing cycle} of $f$ associated with $\gamma_i$. 
Since the product $\alpha_1\cdots \alpha_m\in \pi_1(D^2\setminus \Critv(f),q_0)$ represents a loop homotopic to $\Pa D^2$, the monodromy along $\Pa D^2$ is equal to $t_{C_1}\cdots t_{C_m}$.
This factorization of the monodromy into Dehn twists is called a \emph{monodromy factorization} of $f$. 
Any sequence of Dehn twists can be realized as a monodromy factorization of a Lefschetz fibration. 
Furthermore, the isomorphism class of a Lefschetz fibration over $D^2$ is uniquely determined from the sequence of its vanishing cycles (\cite{Kas1980handlebodyLF,Matsumoto1996LFgenus2}).
The following indeed holds. (See \cite{Matsumoto1996LFgenus2} for Lefschetz fibrations on $4$-manifolds.) 

\begin{theorem}\label{thm:equivalence LF from van cyc}

Let $W$ be a $(2n-2)$-manifold and $C_1,\ldots, C_m$ be $(n-1)$-spheres in $W$ whose normal bundles are isomorphic to their cotangent bundles. 
There exists a Lefschetz fibration $f:X\to D^2$ on a $2n$-manifold $X$ such that its regular fiber is diffeomorphic to $W$, and its monodromy factorization is $t_{C_1}\cdots t_{C_m}$. 
Let $f_i:X_i\to D^2$ be a Lefschetz fibration ($i=0,1$). 
Suppose that $f_0$ and $f_1$ have the same monodromy factorization $t_{C_1}\cdots t_{C_m}$. 
Let $\varphi_j^i:U_j^i\to \C^n$ and $\psi_j^i:V_j^i\to \C$ be orientation-preserving complex charts of the critical points and values of $f_i$ associated with the vanishing cycle $C_j$ such that $\psi_j^i\circ f_i \circ (\varphi_j^i)^{-1}(z_1,\ldots, z_n)=z_1^2+\cdots +z_n^2$. 
There exists an isomorphism $(\Phi,\Psi)$ from $f_0$ to $f_1$ such that $\Psi$ preserves Hurwitz systems yielding the monodromy factorization up to isotopy, $\Phi|_{U_j^0}=(\varphi_j^1)^{-1}\circ \varphi_j^0$, and $\Psi|_{V_j^0}=(\psi_j^1)^{-1}\circ \psi_j^0$. 

\end{theorem}

Let $f:X\to \PP^1$ be a Lefschetz fibration over $\PP^1$, and $D\subset \PP^1$ be a closed disk containing no critical values of $f$. 
Put $X'=X\setminus f^{-1}(\Int(D))$. 
The restriction $f:X'\to \PP^1\setminus \Int(D)$ is a Lefschetz fibration over a disk, in particular we obtain a monodromy factorization $t_{C_1}\cdots t_{C_m}$ of this fibration, which we call a monodromy factorization of $f$. 
Since the restriction $f:f^{-1}(\overline{D})\to \overline{D}$ is a trivial $(2n-2)$-manifold bundle, the product $t_{C_1}\cdots t_{C_m}\in \mathcal{M}_\Pa(W)$ is contained in the kernel of the forgetting homomorphism $\mathcal{M}_\Pa(W)\to \mathcal{M}(W)$. 
Suppose that $X$ is a closed $4$-manifold. 
Let $\mathcal{S}_i\subset X$~($i=1,\ldots, b$) be a section of $f$ with self-intersection $a_i\in \Z$, and $\nu\mathcal{S}_i$ be its tubular neighborhood. 
The restriction $f:X\setminus (\bigsqcup_i \nu \mathcal{S}_i)\to \PP^1$ is a Lefschetz fibration, whose regular fiber $W$ is diffeomorphic to the genus-$g$ surface with $b$ boundary components for some $g\geq 0$. 
It is known that the product $t_{C_1}\cdots t_{C_m}\in \mathcal{M}_\Pa (W)$ for this fibration is equal to $t_{\delta_1}^{-a_i}\cdots t_{\delta_b}^{-a_b}$, where $\delta_1,\ldots, \delta_b\subset W$ are the simple closed curves parallel to the boundary components.

Let $X$ be a closed $2n$-manifold and $B\subset X$ be a codimension $4$ submanifold. 
A smooth map $f:X\setminus B\to \PP^1$ is called a \emph{Lefschetz pencil} if $B$ admits a neighborhood $V\subset X$ with a complex vector bundle $p:V\to B$ compatible with the orientation of $X$ such that the structure group of $p:V\to  B$ is $S^1$ and $f|_V$ is the projectivization on each fiber of $p$, and $f|_{X\setminus V}$ is a Lefschetz fibration. 
The submanifold $B$ is called the \emph{base locus} of $f$. 
As explained in the introduction, we denote a Lefschetz pencil by $f:X\dashrightarrow \PP^1$ when we do not need to represent its base locus explicitly. 
Let $\widetilde{X}$ be the blow-up of $X$ along $B$ and $\pi:\widetilde{X}\to X$ be the projection. 
It is easy to see that there exists a Lefschetz fibration $\widetilde{f}:\widetilde{X}\to \PP^1$ such that the restriction $\widetilde{f}:{\widetilde{X}\setminus \mathcal{S}}\to \PP^1$ is equal to $f\circ \pi$, where $\mathcal{S}$ is the exceptional divisor.
Let $W$ be a regular fiber of the restriction $\widetilde{f}:\widetilde{X}\setminus \nu \mathcal{S}\to \PP^1$, whose boundary admits an $S^1$-bundle associated with $p:V\to B$.
It is known that the product $t_{C_1}\cdots t_{C_m}\in \mathcal{M}_{\Pa}(W)$ is represented by the $2\pi$ counterclockwise rotation of the fibers of the $S^1$-bundle on $\Pa W$ (cf.~\cite{AurouxSympplecticMapsInvariant}). 
In the $4$-dimensional case (i.e.~$n=2$), this product is equal to $t_{\delta_1}\cdots t_{\delta_m}$.

\subsection{Bifibration structures on $6$-manifolds and their monodromies}\label{subsec:pencil-fibration str}

In this subsection, we introduce several bifibration structures on $6$-manifolds, which induce Lefschetz fibrations/pencils on $6$-manifolds and their $4$-dimensional regular fibers. 

\begin{definition}\label{def:Lefschetz bifibration}

Let $X$ be a $6$-manifold and $\rho:Z\to E$ be either a $D^2$-bundle or a $\PP^1$-bundle over $E=D^2$ or $\PP^1$. 
Let $\Pa_v Z = \rho^{-1}(\Pa E)$ and $\Pa_hZ = \overline{\Pa Z\setminus \rho^{-1}(\Pa E)}$, which are called the vertical and horizontal boundaries of $Z$, respectively. (Note that $\Pa_vZ=\emptyset$ if $E=\PP^1$, and $\Pa_h Z=\emptyset$ if $\rho$ is a $\PP^1$-bundle.)
A smooth map $F:X\to Z$ is called a \emph{Lefschetz bifibration} over $\rho:Z\to E$ if it satisfies the following conditions.

\begin{enumerate}[(1)]

\item 
For any $x\in\Crit(F)\cap \Int (X)$, one can take an orientation-preserving complex chart $\varphi:U\to \C^3$ (resp.~$\psi:V\to \C^2$) around $x$ (resp.~$F(x)$) satisfying either
\begin{itemize}

\item 
$\psi\circ F\circ \varphi^{-1}(z_1,z_2,z_3) = (z_1,z_2^2+z_3^2)$, or

\item 
$\varphi(x)=0$, $\psi\circ F\circ \varphi^{-1}(z_1,z_2,z_3) = (z_1,z_2^3+z_1z_2+z_3^2)$, and there exists an orientation-preserving complex chart $\psi':V'\to \C$ around $\rho(F(x))\in E$ such that $\psi'\circ \rho\circ \psi^{-1}$ is the projection to the first component.

\end{itemize} 

\noindent
The point $x\in\Crit(F)$ satisfying the former (resp.~the latter) condition is called a \emph{fold} (resp.~a \emph{cusp}). 
Let $\mathcal{C}_F$ be the set of cusps and $C_F =F(\mathcal{C}_F)\subset \Critv(F)$. 

\item 
Each fiber of $F$ contains at most two critical points. 
Each double point $y\in \Critv(F)$ is contained in $\Int(Z)$, and for such a $y$, one can take an orientation-preserving complex chart $\psi:V\to \C^2$ (resp.~$\psi':V'\to \C$) around $y$ (resp.~$\rho(y)$) so that $\psi'\circ \rho\circ \psi^{-1}$ is the projection to the first component, and $\psi(V\cup \Critv(F))$ is either 
\begin{equation}\label{eqn:local model +doublept}
\psi(V)\cap \bigl(\{(z,0)\in \C^2~|~z\in \C\}\cup \{(z,z)\in \C^2~|~z\in \C\}\bigr)
\end{equation}
or
\begin{equation}\label{eqn:local model -doublept}
\psi(V)\cap \bigl(\{(z,0)\in \C^2~|~z\in \C\}\cup \{(z,\overline{z})\in \C^2~|~z\in \C\}\bigr).
\end{equation}
Let $I_F^+, I_F^-\subset \Critv(F)$ be the sets of positive and negative double points, $I_F = I_F^+\cup I_F^-$, $\mathcal{I}_F^\pm = F^{-1}(I_F^\pm)\cap \Crit(F)$, and $\mathcal{I}_F = \mathcal{I}_F^+\cup \mathcal{I}_F^-$. 

\item 
The set $B_F=\Crit(\rho|_{\Critv(F)\setminus(C_F\cup I_F)})$ is contained in $\Int(Z)$.
We put $\mathcal{B}_F = F^{-1}(B_F)\cap \Crit(F)$. 
For each $x\in \mathcal{B}_F$, one can take an orientation-preserving complex chart $\varphi:U\to \C^3$ (resp.~$\psi:V\to \C^2$, and $\psi':V'\to \C$) around $x$ (resp.~$F(x)$, and $\rho\circ F(x)$) so that $\psi'\circ \rho\circ \psi^{-1}$ is the projection to the first component, and $\psi\circ F\circ \varphi^{-1}(z_1,z_2,z_3)=(z_1^2+z_2^2+z_3^2,z_1)$. 
A point in $B_F$ and one in $\mathcal{B}_F$ are called a \emph{branch point} of $F$.

\item 
Each fiber of $\rho$ contains at most one point in $\Delta_F = C_F\cup I_F \cup B_F$. 

\item 
The boundary $\Pa X$ admits a decomposition into three codimension-$0$ submanifolds $\Pa_{vv}X = F^{-1}(\Pa_v Z)$, $\Pa_{vh}X = F^{-1}(\Pa_h Z)$, and $\Pa_h X = \overline{\Pa X \setminus F^{-1}(\Pa Z)}$, called the \textit{vertical boundaries over $\Pa_v Z$ and $\Pa_h Z$} and the \textit{horizontal boundary} with respect to $F$. 

\item
The restriction $F:{\Pa_h X}\to Z$ is a disjoint union of $S^1$-bundles. 

\item
For any $x\in \Crit(F)\cap \Pa X$, $F(x)$ is contained in $\Pa_v Z$, and one can take an orientation-preserving complex chart $\varphi:U\to \C^3_{\geq 0}$ (resp.~$\psi:V\to \C^2_{\geq 0}$) around $x$ (resp.~$F(x)$) with $\psi\circ F\circ \varphi^{-1}(z_1,z_2,z_3) = (z_1,z_2^2+z_3^2)$, where $\C^n_{\geq 0} = \{(z_1,\ldots, z_n)\in \C^n~|~ \Im z_1 \geq 0\}$. 

\end{enumerate}

\noindent
For a Lefschetz bifibration $F:X\to Z$, regular fibers of $F$ and $\rho\circ F$ are respectively called a \textit{$2$-dimensional fiber} and a \textit{$4$-dimensional fiber}, and the genus of a $2$-dimensional fiber is called the \textit{genus} of $F$.

\end{definition}

For $m\geq 0$, we define an $(\sph^1)^2$-action to $(\sph^3)^2$ as follows:
\[
(\lambda,\mu)\cdot (s_0,s_1,t_0,t_1) := (\lambda s_0,\lambda s_1,\mu t_0, \mu\lambda^{-m}t_1). 
\]
Let $\Hir_m = (\sph^3)^2/(\sph^1)^2$. 
It is known that the projection $\rho_m: \Hir_m \to \sph^3/\sph^1\cong \PP^1$ is a holomorphic $\PP^1$-bundle and $S_m:= [\sph^3\times \{(1,0)\}]\subset \Hir_m$ is a holomorphic section of $\rho_m$ with self-intersection $-m$, in particular, $\Hir_m$ is the Hirzebruch surface of degree $m$. 
For a technical reason, we always assume the following condition for a Lefschetz bifibration $F:X\to \Hir_m$ over $\rho_m$. 

\begin{enumerate}[(1)]

\setcounter{enumi}{-1}

\item 
$\Critv(F)\cap S_m=\emptyset$.

\end{enumerate}

%
%
%
%
%
%
%
%
%

Let $F:X\to Z$ be a Lefschetz bifibration and $L\subset Z$ be a fiber of $\rho$ away from $\Delta_F$. 
The restriction $F|_{F^{-1}(L)}:F^{-1}(L)\to L$($\cong D^2$ or $\PP^1$) is a Lefschetz fibration by the following lemma. 

\begin{lemma}\label{lem:restriction fold to transverse submfd is Lef}

Let $G:\C^3\to \C^2$ be the smooth map defined by $G(z_1,z_2,z_3)=(z_1,z_2^2+z_3^2)$ and $L\subset \C^2$ be a $2$-dimensional submanifold which intersects $\C\times \{0\}$ at one point transversely. 
Then the restriction $\overline{G}:=G|_{G^{-1}(L)}:G^{-1}(L)\to L$ has a single Lefschetz singularity. 

\end{lemma}

\begin{proof}
Without loss of generality, one can assume that $L$ intersects $\C\times\{0\}$ at the origin. 
Since the tangent space $T_x G^{-1}(L)$ is equal to the preimage $(dG_x)^{-1}(\Im dG_x \cap T_{G(x)}L)$, a critical point of $\overline{G}$ is contained in $G^{-1}(L)\cap \Crit(G)=\{0\}$. 
Thus, the origin of $\C^3$ is the unique critical point of $\overline{G}$. 
Let $p:\C^3\to \C^2$ and $p':\C^2\to \C$ be the projections to the latter components. 
By the assumption, the restrictions $p|_{G^{-1}(L)}$ and $p'|_{L}$ are both locally diffeomorphic at the origins. 
One can further check that $(p'|_{L})\circ \overline{G}\circ (p|_{G^{-1}(L)})^{-1}(z_2,z_3)$ is equal to $z_2^2+z_3^3$, which completes the proof of \Cref{lem:restriction fold to transverse submfd is Lef}. 
\end{proof}

\subsubsection{Monodromy representations}

Let $\rho:D^2\times D^2\to D^2$ be the projection to the first component, and $F:X\to D^2\times D^2$ be a Lefschetz bifibration over $\rho$. 
Take $q_0\in D^2\setminus \rho(\Delta_F)$, and put $L = \rho^{-1}(q_0) = \{q_0\}\times D^2$, which is identified with $D^2$ in the obvious way. 
Let $R = L\cap \Critv(F)=\{p_1,\ldots, p_d\} \subset \Int(D^2)$. 
We take a horizontal distribution $\mathcal{H}$ of $\rho$ on ${(D^2\times D^2)\setminus \rho^{-1}(\rho(\Delta_F))}$ satisfying the following conditions.

\begin{itemize}

\item 
$\mathcal{H}_y = T_y\Critv(F)$ for $y\in \Critv(F) \setminus (\rho^{-1}(\rho(\Delta_F)))$. 

\item 
$\mathcal{H}_y = T_y(\{\ast\}\times D^2)$ for $y$ close to $\Pa_h (D^2\times D^2) = D^2\times \Pa D^2$. 

\end{itemize}

\noindent
Taking parallel transports along $\mathcal{H}$, one can define a \emph{braid monodromy representation} of $F$
\[
\eta_F :\pi_1(D^2\setminus\rho(\Delta_F),q_0) \to B_d, 
\]
where $B_d$ is the braid group with $d$ strands, which is identified with the mapping class group $\mathcal{M}_\Pa(D^2,R)$\label{page:braid group}. 

Let $\gamma\subset D^2$ be a simple path from $q_0$ which intersects $\rho(\Delta_F)$ only at $q\in \rho(\Delta_F)$ on its end. 
There exist two points in $L\cap \Critv(F)$ which are connected by a path $\mu_0$ in $\rho^{-1}(\gamma)\cap \Critv(F)$. 
Furthermore, one can take a path $\mu\subset L$, unique up to homotopy, such that $\mu_0\cup \mu$ bounds a disk in $\rho^{-1}(\gamma)$ which intersects $\Critv(F)$ only along $\mu_0$. 
We call the path $\mu$ and its homotopy class the \emph{vanishing path} of $\gamma$. 
%
%
%
%
As we introduced for Lefschetz fibrations in \Cref{subsec:LF/LP}, we also define a \emph{Hurwitz system} of $F$ (from $q_0$) as a system of simple paths $\gamma_1,\ldots, \gamma_n$ in $D^2$ satisfying the conditions (A) and (B) on page \pageref{list:condition Hurwitz system}, and
(C) and (D) on the same page, with $\Critv(f)$ replaced by $\rho(\Delta_F)$.
Let $\alpha_i$ be the meridian loop of $\gamma_i$ and $\mu_i$ be the vanishing path of $\gamma_i$.
It is easy to see that the monodromy $\eta_F(\alpha_i)$ is equal to $\tau_{\mu_i}^3$ (resp.~$\tau_{\mu_i}^{\pm 2}$, $\tau_{\mu_i}$) if an end point of $\gamma_i$ is contained in $\rho_m(C_F)$ (resp.~$\rho_m(I_F)$, $\rho_m(B_F)$). 
As in the case of Lefschetz fibrations, we obtain the factorization $\tau_{\mu_1}^{\epsilon_1}\cdots \tau_{\mu_n}^{\epsilon_n}$ of the braid monodromy along $\Pa D^2$ into half twists, where $\epsilon_i\in \{1,\pm2,3\}$. 
This factorization is called a \emph{braid monodromy factorization} of $F$. 


Let $W = F^{-1}(L)$.  
By \Cref{lem:restriction fold to transverse submfd is Lef}, the restriction $F|_W:W\to L$ is a Lefschetz fibration whose critical value set is equal to $R=L\cap \Critv(F)$.
Take $r_0\in \Pa D^2$ and put $\Sigma=F^{-1}(r_0)$.
We call the monodromy representation $\theta_{F|_{W}}:\pi_1(D^2\setminus R,r_0)\to \mathcal{M}_\Pa(\Sigma)$ the \emph{geometric monodromy representation} or \emph{fiber monodromy representation} of $F$, which is denoted by $\theta_F$. 

For each $i=1,\ldots, n$, one can isotope a vanishing path $\mu_i$ so that it goes through $r_0$.
In particular one can regard $\mu_i$ as a pair of reference paths and obtain two meridian loops $\beta_i^1, \beta_i^2\in \pi_1(D^2\setminus R,r_0)$. 
By the local coordinate descriptions of $F$ around cusps, double points and branch points, one can deduce that two vanishing cycles $c_1,c_2$ of $F|_{W}$ associated with $\beta_i^1$ and $\beta_i^2$ can be isotoped so that these intersect on one point transversely (resp.~these are disjoint, these coincide) if the exponent $\epsilon_i$ is $3$ (resp.~$\pm 2$, $1$).
This condition is equivalent to that $\theta_{F}(\beta_i^1)$ and $\theta_{F}(\beta_i^2)$ satisfy the braid relation (resp. commute, coincide) if an end point of $\mu_i$ is contained in $\rho_m(C_F)$ (resp.~$\rho_m(I_F)$, $\rho_m(B_F)$). 

\begin{definition}\label{def:compatibility condition}

The condition for $\mu_i$ (with respect to $\theta_F$ and the index $\epsilon_i$) given here is called the \emph{compatibility condition}.
A factorization $\tau_{\mu_1}^{\epsilon_1}\cdots \tau_{\mu_n}^{\epsilon_n}$ and a homomorphism $\theta_F$ are said to be \textit{compatible} if all the $\mu_i$'s satisfy the compatibility condition. 

\end{definition}

\noindent
Note that if an end point of $\mu_i$ is in $\rho_m(B_F)$, the vanishing path of $\mu_i$ is a \emph{matching path} in the sense of \cite[(16g)]{SeidelbookFukayaPicard} and one can obtain a \emph{matching cycle} of the vanishing path, which is a $2$-sphere in $W$. 
Moreover, the monodromy of the Lefschetz fibration $\rho\circ F$ along a meridian loop of $\mu_i$ is equal to the Dehn twist along a matching cycle.

\subsubsection{Construction of bifibration structures from monodromies}

As in the case of Lefschetz fibrations, one can construct a Lefschetz bifibration from a compatible pair of braid and fiber monodromies. 
The following theorem, which is one of the main results of this paper, provides a combinatorial method for constructing Lefschetz bifibrations on $6$-manifolds over $D^2 \times D^2$.

\begin{theorem}\label{thm:combinatorial construction LbF over D2timesD2}

Let $R=\{r_1,\ldots, r_d\}\subset \Int(D^2)$, $\mu_1,\ldots, \mu_n\subset \Int(D^2)$ be simple paths between two points in $R$, and $\epsilon_i \in \{1,\pm 2,3\}$. 
Let $r_0\in \Pa D^2$, $\Sigma_g^b$ be a genus-$g$ surface with $b$ boundary components and 
\[
\theta:\pi_1(D^2\setminus R,r_0) \to \mathcal{M}_\Pa(\Sigma_g^b)
\] 
be a homomorphism such that $\theta(\alpha)$ is a non-trivial Dehn twist for the meridian loop $\alpha$ of a simple path between $r_0$ and a point in $R$.
Suppose that $\mu_i$ satisfies the compatibility condition with respect to $\theta$ and $\epsilon_i$ for $i=1,\ldots, n$.  
There exist a $6$-manifold $X$ and a Lefschetz bifibration $F:X\to D^2\times D^2$ over the projection $\rho:D^2\times D^2\to D^2$ whose braid monodromy factorization is $\tau_{\mu_1}^{\epsilon_1}\cdots \tau_{\mu_n}^{\epsilon_n}$ and fiber monodromy representation is $\theta$. 

\end{theorem}

\noindent
The proof of this theorem is given in \Cref{sec:combinatorial construction pencil-fibration}.

Let $F:X\to \Hir_m$ be a Lefschetz bifibration over $\rho_m:\Hir_m\to \PP^1$. 
Take a tubular neighborhood $\nu S_m$ of the section $S_m$ and a closed disk $D\subset \PP^1$ containing no values in $\rho_m(\Delta_F)$. 
Put $Z' = \Hir_m \setminus (\rho_m^{-1}(\Int(D))\cup \nu S_m)$ and $X' = F^{-1}(Z')$. 
The restriction $F':=F|_{X'}:X'\to Z'$ is a Lefschetz bifibration over the restriction $\rho_m:Z'\to \PP^1\setminus \Int(D)$. 
A \textit{braid monodromy factorization} of $F$ is that of the Lefschetz bifibration $F':X'\to Z'$, and denoted by $\tau_{\mu_1}^{\epsilon_1}\cdots \tau_{\mu_n}^{\epsilon_n}$. 
We also define a \textit{fiber monodromy representation} of $F$, denoted by $\theta_F$, in the same manner. 
Since the self-intersection of $S_m$ is $-m$, the braid monodromy along $\Pa D^2$ is $t_\Pa^m$, where $t_\Pa$ is the Dehn twist along a simple closed curve in $D^2$ parallel to $\Pa D^2$. 
Thus the product $\tau_{\mu_1}^{\epsilon_1}\cdots \tau_{\mu_n}^{\epsilon_n}\in B_d$ is equal to $t_\Pa^m$. 

\begin{theorem}\label{thm:combinatorial construction LbF over Hir_m}

Let $R$, $\mu_i$, $\epsilon_i$, and $\theta$ be the same ones as those in \Cref{thm:combinatorial construction LbF over D2timesD2} (satisfying the same assumptions). 
Suppose that $(g,b)\neq (0,2)$, the product $\tau_{\mu_1}^{\epsilon_1}\cdots \tau_{\mu_n}^{\epsilon_n}$ is equal to $t_\Pa^m$, and the image $\theta(\alpha_\Pa)$ is equal to $t_{\delta_1}^{a_1}\cdots t_{\delta_b}^{a_b}$, where $\alpha_\Pa$ is represented by a loop homotopic to $\Pa D^2$ with counterclockwise orientation, $\delta_1,\ldots, \delta_b\subset \Sigma_g^b$ are the simple closed curves parallel to the boundary components, and $a_1,\ldots,a_b\in \Z$.
There exist a closed $6$-manifold $X$, a Lefschetz bifibration $F:X\to \Hir_m$, and sections $\mathcal{S}_1,\ldots, \mathcal{S}_b\subset X$ of $F$ satisfying the following conditions.

\begin{itemize}

\item 
A braid monodromy factorization of the Lefschetz bifibration $F:X\setminus (\bigsqcup_k \nu \mathcal{S}_k)\to \Hir_m$ is $\tau_{\mu_1}^{\epsilon_1}\cdots \tau_{\mu_n}^{\epsilon_n}$, and its fiber monodromy representation is $\theta$. 

\item 
The normal bundle of $\mathcal{S}_k$ has Euler number $-a_k$ (resp.~$0$) over a fiber of $\rho_m$ (resp.~$S_m$). (Here, we identify $\mathcal{S}_k$ with $\Hir_m$ via $F$.)

\end{itemize}

\end{theorem}

\noindent
The proof of this theorem is also given in \Cref{sec:combinatorial construction pencil-fibration}. 
Note that the assumption $(g,b)\neq (0,2)$ is inessential. 
As discussed in \Cref{rem:assumption (gb)neq(02)}, this condition can be removed by refining the argument concerning the behavior of Lefschetz bifibrations at the horizontal boundaries.
While such a refinement is possible, it is omitted in this paper to maintain overall conciseness, as the resulting examples in the case of $(g,b)=(0,2)$ are not particularly significant.

As in the case of Lefschetz fibrations and pencils, one can obtain a Lefschetz bifibration by blowing up other variants of bifibration structures. 

\begin{definition}\label{def:other pencil-fibration structures}

Let $X$ be a closed $6$-manifold. 

\begin{itemize}

\item 
Let $B\subset X$ be a codimension $4$ ($2$-dimensional) submanifold. 
A smooth map $F:X\setminus B\to \Hir_m$ is called a \emph{Lefschetz pencil-fibration} if it satisfies the conditions (0)--(4) in \Cref{def:Lefschetz bifibration} and the following condition. 

\begin{enumerate}

\setcounter{enumi}{7}

\item 
For any $x\in B$, one can take an orientation-preserving complex chart $\varphi:U\to \C^3$ (resp.~$\psi':V'\to \C$) around $x$ (resp.~$\rho_m(F(q))$) and a local trivialization $\psi:\rho_m^{-1}(V')\to V'\times \PP^1$ of $\rho_m$ satisfying $(\psi'\times \id_{\PP^1})\circ \psi \circ F \circ \varphi^{-1}(z_1,z_2,z_3)=(z_1,[z_2:z_3])$.  

\end{enumerate}

\item 
A smooth map $F:X\to \PP^2$ is called a \emph{Lefschetz pencil-fibration} if $\Critv(F)$ is away from $[0:0:1]$ and $F$ satisfies the conditions (1)--(4) in \Cref{def:Lefschetz bifibration} with $\rho$ in the conditions replaced with the projection $\rho_1':\PP^2\setminus \{[0:0:1]\}\to \PP^1$. 

\item 
Let $B \subset X$ be a finite set. 
A smooth map $F:X\setminus B\to \PP^2$ is called a \emph{Lefschetz bipencil} if $\Critv(F)$ is away from $[0:0:1]$, $F$ satisfies the conditions (1)--(4) in \Cref{def:Lefschetz bifibration} with $\rho$ in the conditions replaced with the projection $\rho_1':\PP^2\setminus \{[0:0:1]\}\to \PP^1$, and the following condition. 

\begin{enumerate}

\setcounter{enumi}{8}

\item 
For any $x\in B$, one can take an orientation-preserving complex chart $\varphi:U\to \C^3$ around $x$ satisfying $F \circ \varphi^{-1}(z_1,z_2,z_3)=[z_1:z_2:z_3]$.  

\end{enumerate}

\end{itemize}

\end{definition}

Let $F:X\setminus B \to \PP^2$ be a Lefschetz bipencil. 
By blowing up $X$ at all points in $B$, one can obtain a Lefschetz pencil-fibration $\widetilde{F}_1:\widetilde{X}_1\to \PP^2$. 
On the other hand, by blowing up $X$ (resp.~$\PP^2$) along $\overline{F^{-1}([0:0:1])}$ (resp.~at $[0:0:1]$), one can obtain a Lefschetz pencil-fibration $\widetilde{F}_2:\widetilde{X}_2\setminus \widetilde{B} \to \Hir_1$, where $\widetilde{B}\subset \widetilde{X}_2$ is the preimage of $B\subset X$ under the blow-down map from $\widetilde{X}_2$ to $X$. 
One can further obtain a Lefschetz bifibration by either blowing up $\widetilde{X}_1$ along $\widetilde{F}_1^{-1}([0:0:1])$, or blowing up $\widetilde{X}_2$ along $\widetilde{B}$. 
Let $F$ be one of those given in \Cref{def:other pencil-fibration structures}. 
A braid/fiber monodromy representation/factorization means that of a Lefschetz bifibration obtained by blowing-up $F$.

\begin{theorem}\label{thm:combinatorial construction pencil-fibration}

Let $R$, $\mu_i$, $\epsilon_i$, $\theta$ and $a_k$ be the same as those in \Cref{thm:combinatorial construction LbF over Hir_m} (satisfying the same assumptions). 

\begin{enumerate}[(1)]

\item 
For $k_1,\ldots, k_l\in {1,\ldots, b}$ with $a_{k_1}=\cdots = a_{k_l}=1$, the $6$-manifold $X'$ obtained by blowing-down $X$ along fibers of $\mathcal{S}_{k_1},\ldots, \mathcal{S}_{k_l}$ admits a Lefschetz pencil-fibration $F':X'\dasharrow \Hir_m$ over $\rho_m$ satisfying the same monodromy condition as that in \Cref{thm:combinatorial construction LbF over Hir_m}. 

\item 
If $m=1$, there exist a $6$-manifold $X^\dagger$, a Lefschetz pencil-fibration $F^\dagger: X^\dagger \to \PP^2$, and sections $\mathcal{S}_1^\dagger,\ldots, \mathcal{S}_b^\dagger$ such that it satisfies the same monodromy condition, and the normal bundle of $\mathcal{S}_i^\dagger$ has Euler number $-a_k$ over a linear $\PP^1\subset \PP^2$. 
Furthermore, for $k_1,\ldots, k_l\in \{1,\ldots, b\}$ with $a_{k_1}=\cdots =a_{k_l}=1$, the $6$-manifold $X^{\dagger\dagger}$ obtained by blowing-down $\mathcal{S}_{k_1}^\dagger,\ldots, \mathcal{S}_{k_l}^\dagger$ admits a Lefschetz bipencil $F^{\dagger\dagger}:X^{\dagger\dagger}\dasharrow \PP^2$ satisfying the same monodromy condition. 

\end{enumerate}

\end{theorem}

\noindent
The proof is given in \Cref{sec:combinatorial construction pencil-fibration}.

\section{Liftable braids}\label{sec:liftable briad}

In this section, we give a criterion for an element of a braid group, considered as an isotopy class of a self-diffeomorphism of $D^2$ fixing a finite subset $R\subset \Int(D^2)$, to be lifted to a self-isomorphism of a Lefschetz fibration over the disk whose critical value set is $R$ (\Cref{thm:condition lift a braid}). 
Note that a similar statement (existence and uniqueness of a symplectomorphic lift by a Lefschetz pencil) was given in \cite{AurouxSympplecticMapsInvariant} without proofs (and any technical assumptions in \Cref{thm:condition lift a braid}, cf.~\Cref{rem:assumptions for lifing a braid}).

Let $W$ be a $4$-manifold, and $f:W\to D^2$ be a Lefschetz fibration whose regular fiber is diffeomorphic to $\Sigma_g^b$ (in particular $\Pa_h W \neq \emptyset$).
Let $(\Phi_i,\Psi_i)$ ($i=0,1$) be a self-isomorphism of $f$. 
A pair of level-preserving diffeomorphisms 
\[
(\mathfrak{H}:[0,1]\times W\to [0,1]\times W, \mathfrak{h}:[0,1]\times D^2\to [0,1]\times D^2)
\]
is called a \textit{fiber-preserving isotopy pair} from $(\Phi_0,\Psi_0)$ to $(\Phi_1,\Psi_1)$ if it satisfies the following conditions. 

\begin{itemize}

\item 
$\mathfrak{h}_t\circ f = f\circ \mathfrak{H}_t$ for any $t\in [0,1]$, where $\mathfrak{H}_t(x):=p_2\circ \mathfrak{H}(t,x)$ and $\mathfrak{h}_t(x):=p_2\circ \mathfrak{h}(t,x)$. 

\item 
$\mathfrak{H}_0=\id_W$, $\Phi_0\circ \mathfrak{H}_1 = \Phi_1$, $\mathfrak{h}_0=\id_{D^2}$, and $\Psi_0\circ \mathfrak{h}_1 = \Psi_1$. 


\end{itemize}

\noindent
Two self-isomorphisms $(\Phi_0,\Psi_0)$ and $(\Phi_1,\Psi_1)$ of $f$ are said to be \textit{isotopic} if there exists a fiber-preserving isotopy pair between them. 
Let $(\mathfrak{H},\mathfrak{h})$ be a fiber-preserving isotopy pair from $(\Phi_0,\Psi_0)$ to $(\Phi_1,\Psi_1)$. 
By taking a monotone-increasing smooth function $\rho:[0,1]\to [0,1]$ with $\rho\equiv 0$ (resp.~$\rho\equiv 1$) on a neighborhood of $0\in [0,1]$ (resp.~$1\in [0,1]$), one can define a level-preserving diffeomorphism $\widetilde{\mathfrak{H}}:\R\times W\to \R\times W$ by $\widetilde{\mathfrak{H}}(t+n,x) = (t+n,\Phi_0^{-n}\circ \mathfrak{H}_{\rho(t)}\circ \Phi_1^n(x))$ for any $t\in [0,1]$ and $n\in \Z$, and $\widetilde{\mathfrak{h}}:\R\times D^2\to \R\times D^2$ in the same way. 
These maps satisfy the following conditions.

\begin{itemize}

\item 
$\widetilde{\mathfrak{h}}_t\circ f = f\circ \widetilde{\mathfrak{H}}_t$ for any $t\in \R$. 

\item 
$\widetilde{\mathfrak{H}}_0=\id_W$, $\Phi_0^n\circ \widetilde{\mathfrak{H}}_{t+n} = \widetilde{\mathfrak{H}}_{t}\circ \Phi_1^n$ for $t\in \R$ and $n\in \Z$, and the same conditions for $\widetilde{\mathfrak{h}}_t$, $\Psi_0$, and $\Psi_1$.


\end{itemize}

\noindent
We also call $(\widetilde{\mathfrak{H}},\widetilde{\mathfrak{h}})$ satisfying the conditions above a fiber-preserving isotopy pair. 
Note that the map $\widetilde{\mathfrak{H}}$ induces an isomorphisms (as $W$-bundles) between the mapping tori $(\R\times W)/((t+n,x)\sim (t,\Phi_1^n(x)))$ and $(\R\times W)/((t+n,x)\sim (t,\Phi_0^n(x)))$, and so does $\widetilde{\mathfrak{h}}$ (for $\Psi_0,\Psi_1$).  

\begin{theorem}\label{thm:condition lift a braid}

Let $W$ and $f$ be the same as above. 
We take $r_0\in \Pa D^2$, and put $\Sigma=f^{-1}(r_0)$ and $R=\Critv(f)=\{r_1,\ldots, r_d\}$.
Let $\theta_{f,r_0}:\pi_1(D^2\setminus R,r_0)\to \mathcal{M}_\Pa(\Sigma)$ be the monodromy representation of $f$. 

\begin{itemize}

\item 
For mapping classes $\tau\in \mathcal{M}_\Pa(D^2,R)$ and $\xi\in \mathcal{M}_\Pa (\Sigma)$, the following conditions are equivalent. 

\begin{enumerate}

\item 
$\theta \circ \tau_\ast = \xi^{-1} \cdot \theta\cdot \xi$, where $\tau_\ast$ is the automorphism of $\pi_1(D^2\setminus R,r_0)$ induced by $\tau$. 

\item 
For any $\overline{\xi}\in \xi$, there exist $T\in \tau$ and $\widetilde{T}\in \mathcal{D}(W)$ satisfying the following conditions.

\begin{enumerate}[(A)]

\item 
$(\widetilde{T},T)$ is a self-isomorphism of $f$. 

\item 
$\widetilde{T}|_{\Sigma}=\overline{\xi}$. 

\end{enumerate}

\end{enumerate}

\item 
For any $j\in \{1,\ldots,d\}$, take a sufficiently small orientation-preserving complex chart $\varphi_j:U_j\to \C^2$ (resp.~$\psi_j:V_j\to \C$) around $x_j\in \Crit(f)\cap f^{-1}(r_j)$ (resp.~$r_j$) so that $\psi_j\circ f\circ \varphi_j^{-1}(z,w)=zw$. 
For $\tau$ and $\xi$ satisfying the condition (1) above, we can take $\tilde{T}$ and $T$ satisfying the conditions (A), (B) and

\begin{enumerate}[(A)]

\setcounter{enumi}{2}

\item 
$\varphi_{\sigma_\tau(j)}\circ \widetilde{T}\circ \varphi_{j}^{-1}(z,w)$ is equal to either $(z,w)$ or $(w,z)$ for $(z,w)\in \C^2$ in a neighborhood of the origin and $j=1,\ldots, d$, where $\sigma_\tau\in \mathfrak{S}_d$ (which is the symmetric group) satisfying $T(r_j)=r_{\sigma_\tau(j)}$. 

\end{enumerate}

\item 
Suppose that $f$ is relatively minimal, that is, no fiber of $f$ contains a $(-1)$-sphere, and $(g,b)$ is not equal to $(1,0)$.
The isomorphism $(\widetilde{T},T)$ satisfying the conditions (A)--(C) is uniquely determined from $\tau$ and $\xi$ up to fiber-preserving isotopy. 

\end{itemize}

\end{theorem}

\begin{remark}\label{rem:cpx coordinates in lifting a braid}

Instead of the complex charts $\varphi_j,\psi_j$, one can also consider other charts $\varphi_j',\psi_j'$ satisfying $\psi_j'\circ f \circ \varphi_j'^{-1}(z,w)=z^2+w^2$.  
The condition (C) is then equivalent to the condition that $\varphi_{\sigma_\tau(j)}'\circ \widetilde{T}\circ \varphi_{j}'^{-1}(z,w)$ is equal to either $(z,w)$ or $(\pm z,\mp w)$.

\end{remark}

\begin{remark}\label{rem:assumptions for lifing a braid}

The last statement (uniqueness of a self-isomorphism) does not necessarily hold when $f$ is not relatively minimal. 
Indeed, let $f:W\to D^2$ be a Lefschetz fibration whose regular fiber is $\Sigma_g$ ($g\geq 2$). 
Suppose that $f$ has only one critical point with null-homologous vanishing cycle. 
One can take an embedding $\iota:D^2\times \Sigma_g^1\to W$ so that the composition $f\circ \iota:D^2\times \Sigma_g^1\to D^2$ is the projection, and a diffeomorphism $\widetilde{T}:W\to W$ so that $(\widetilde{T},\id_{D^2})$ is a self-isomorphism of $f$, the restriction $\widetilde{T}|_{W\setminus \Im (\iota)}$ is the identity map, and $\iota^{-1}\circ \widetilde{T}\circ \iota = \id_{D^2}\times \Xi$, where $\Xi:\Sigma_g^1\to \Sigma_g^1$ represents a non-trivial pushing map (i.e.~an element in the kernel of the capping homomorphism $\mathcal{M}(\Sigma_g^1)\to \mathcal{M}(\Sigma_g)$). 
Then, the restriction of $\widetilde{T}$ on a regular fiber is isotopic to the identity map, and thus $(\widetilde{T},\id_{D^2})$ satisfies the conditions (A)--(C) for $\tau =1$ and $\xi=1$. 
However, any isotopy $(\mathfrak{H},\mathfrak{h})$ from $(\widetilde{T},\id_{D^2})$ preserves the genus-$g$ component $\Sigma$ in $f^{-1}(y)\setminus \Crit(f)$ for $y\in \Critv(f)$, in particular $\widetilde{T}\circ \mathfrak{H}_1|_{\Sigma}:\Sigma\to \Sigma$ cannot be the identity map. 

The author does not know whether or not uniqueness of $(\widetilde{T},T)$ holds without the condition (C). 
In order to remove the condition (C) from the last statement, we have to examine the structure of the stabilizer subgroup $\mathcal{A}_f$ of $\mathcal{A}$, the group of diffeomorphism-germs of the source and the target, with respect to the Lefschetz singularity germ $f:(z,w)\mapsto zw$. 
Such stabilizers have been studied in e.g.~\cite{Wall1980secondnotesym,duPW1991rightsym,FduP2011str_multigerm_eq_preprint} for generic map-germs.
Since the germ $f$ is not generic (more precisely, it is neither finitely $\mathcal{A}$-determined nor a critical simplification in the sense of \cite{duPW1991rightsym}), one cannot apply the results in \cite{Wall1980secondnotesym,duPW1991rightsym,FduP2011str_multigerm_eq_preprint} to $f$. 
Indeed, one can easily check that, for the Lefschetz singularity germ $f$, the subgroup $\operatorname{Inv}\widetilde{~}(D)\subset \mathcal{L}$ (defined in \cite{duPW1991rightsym}) is the set of self-diffeomorphism germs of $(\C,0)$, and the projection $\mathcal{A}_f\to \operatorname{Inv}\widetilde{~}(D)$ is not surjective.
However, this projection is surjective for a generic map-germ (\cite[Lemma 37]{FduP2011str_multigerm_eq_preprint}).

While the author expects that the last statement remains valid without the assumption $(g,b)\neq (1,0)$, this case is not addressed in the paper as it is not required for the proofs of \Cref{thm:combinatorial construction LbF over D2timesD2,thm:combinatorial construction LbF over Hir_m,thm:combinatorial construction pencil-fibration}.
For more details, refer to \Cref{rem:assumption not torus in lifting braid}. 

\end{remark}

\begin{proof}[Proof of \Cref{thm:condition lift a braid}]
It is easy to check that all the statements in the theorem hold if $d=0$ (i.e.~$f$ has no critical points). 
In what follows, we assume $d$ is greater than $0$. 

Suppose that the condition (2) holds. 
For a curve $c:[0,1]\to D^2\setminus R$ with $c(0)=c(1)=r_0$, put $E(c^\ast f)=\{(t,x)\in [0,1]\times W~|~c(t)=f(x)\}$. 
Take a trivialization $\Phi:E(c^\ast f) \to [0,1]\times \Sigma$ so that $\Phi_0(x) =x$ for $x\in \Sigma$, where $\Phi_t(x)=\Phi(t,x)$. 
It is then easy to check that $\theta([c])$ is represented by $\Phi_1^{-1}$. 
Since $E((T\circ c)^\ast f)$ is equal to $\{(t,x)\in [0,1]\times W~|~(t,\widetilde{T}^{-1}(x))\in c^\ast f\}$, one can take the following trivialization of $E((T\circ c)^\ast f)$.
\[
E((T\circ c)^\ast f) \xrightarrow{\id_{[0,1]}\times \widetilde{T}^{-1}} E(c^\ast f) \xrightarrow{\Phi} [0,1]\times \Sigma \xrightarrow{\id_{[0,1]}\times \widetilde{T}|_\Sigma} [0,1]\times \Sigma. 
\]
One can further check that the restriction of the trivialization above on $\{0\}\times \Sigma$ is the identity map. 
Thus, the mapping class $\theta \circ \tau_\ast ([c])$ is represented by the inverse map of the restriction of the trivialization above on $\{1\}\times \Sigma$, which is equal to
$\widetilde{T}|_\Sigma\circ \Phi_1^{-1} \circ (\widetilde{T}|_{\Sigma})^{-1}$. 
This map represents $\xi^{-1}\cdot \theta([c])\cdot \xi$, completing the proof of the implication (2)$\Rightarrow$(1). (Note that multiplication of the mapping class group is defined to be opposite to composition.)

Suppose that the condition (1) holds. 
We take a trivialization $\kappa_h: \nu \Pa_h W \to D^2\times (-1,0]\times(\bigsqcup_k S^1)$ of $f$ (as a $((-1,0]\times(\bigsqcup_k S^1))$-bundle). 
In what follows, we construct a self-isomorphism of $f$ satisfying the conditions (A)--(C). 

We take an open disk neighborhood $B\subset D^2$ of $r_0$, a diffeomorphism $\psi_0:\overline{B}\to \mathbb{H}_1$, where $\mathbb{H}_1 = \{z\in \disk_1~|~\Im(z)\geq 0\}$, and a representative $T \in \tau$ so that $T|_B=\id_B$ and $\psi_{\sigma_\tau(j)}\circ T\circ \psi_j^{-1}|_{\disk_2}=\id_{\disk_2}$ for $j=1,\ldots, d$. (Recall that $\disk_r = \{z\in \C~|~|z|\leq r\}$ for $r>0$, as explained in the introduction.) 
We also take a Hurwitz path system $\gamma_1,\ldots, \gamma_d \subset D^2$ from $r_0$ and a tubular neighborhood $\iota_j:\gamma_j\times [-1,1]\to D^2$ for each $j=1,\ldots, d$ so that they satisfy the following conditions for a sufficiently small $\delta>0$. (In what follows, we identify $\gamma_j$ with $[0,1]$ so that $r_0\in \gamma_j$ corresponds to $0\in [0,1]$.)

\begin{itemize}

\item 
$\iota_j([\delta,1]\times [-1,1])\cap \iota_k([\delta,1]\times [-1,1])=\emptyset$ for $j\neq k$. 

\item 
$\iota_j(\gamma_j\times [-1,1])\cap \overline{B} = \iota_j([0,2\delta]\times [-1,1])$ and $d(\psi_0\circ \iota_j)_{(s,t)}\left(\left(\frac{\Pa}{\Pa s}\right)\right)$ has the radial direction centered at $0\in \mathbb{H}_1$ for any $(s,t)\in [\delta,2\delta]\times [-1,1]$. 

\item 
$\iota_j(\gamma_j\times [-1,1])\cap \psi_j^{-1}(\mathbb{I})=\iota_j([1-\delta,1]\times [-1,1])$, where $\mathbb{I}=\{z\in \C~|~|\Re(z)|, |\Im(z)|\leq 1\}$, and $\psi_j\circ \iota_j(s,t) = (s-1)/\delta + t \imunit$ for each $j=1,\ldots, d$ and $(s,t)\in [1-\delta,1]\times [-1,1]$.

\end{itemize}

\noindent
Put $\widetilde{B} = B \cup \left(\bigsqcup_{j=1}^d \iota_j([\delta,1-\delta]\times [-1,1])\right)$. 
We take a nowhere zero vector field $\chi$ on $\widetilde{B}\setminus \{r_0\}$ so that $(\psi_0)_\ast \chi$ has the radial direction on $\mathbb{H}_1\setminus \{0\}$ and $\chi = d\iota_j\left(\left(\frac{\Pa}{\Pa s}\right)\right)$ on $\iota_j([\delta,1-\delta]\times [-1,1])$.

We define $q:\C^2\to \C$ by $q(z,w)=zw$.
Let $A_i = \{\zeta\in \C~||\zeta| \geq 4\}$, $A=A_1\sqcup A_2$, and $\mathbb{E}_r:= \{(z,w)\in \C^2~|~zw\in \mathbb{D}_2,\hspace{.3em}\|z\|\leq r,\hspace{.3em}\|w\|\leq r\}$.
We can define a local trivialization $\kappa:q^{-1}(\mathbb{D}_2)\setminus \Int(\mathbb{E}_4) \to \mathbb{D}_2\times A$ of $q$ as follows:
\[
\kappa(z,w) = \begin{cases}
(zw, z)\in \mathbb{D}_2\times A_1 & (|z|\geq |w|) \\
(zw,w)\in \mathbb{D}_2\times A_2 & (|w|\geq |z|). 
\end{cases}
\]
We take a Riemannian metric $g$ on $\C^2$ so that the restriction of it on $q^{-1}(\mathbb{D}_2)\setminus \mathbb{E}_4$ coincides with the pull-back of the product metric on $\mathbb{D}_2\times A$ by $\kappa$. (Note that $\mathbb{D}_2$ and $A$ have the standard metric as subsets of $\C$.) 
We next take a Riemannian metric $g_0$ of $W$ so that $g_0$ is equal to $(\varphi_j)^\ast g$ on $\varphi_j^{-1}(\mathbb{C}^2)$ and the pull-back by $\kappa_h$ of the product metric of $D^2\times (0,1]\times (\bigsqcup_k S^1)$ on $\nu\Pa_h W$. 
Let $\widetilde{\chi}$ be the lift of $\chi$ by $df$ via the horizontal distribution $(\Ker df)^\perp$ with respect to $g_0$, which is defined on $f^{-1}(\widetilde{B})$. 
Using the flow of $\widetilde{\chi}$ starting from $\Sigma=f^{-1}(r_0)$, we take a trivialization $\Phi_0:f^{-1}(\widetilde{B}) \to \widetilde{B}\times \Sigma$. 
We also take a trivialization $\Phi_0^T:f^{-1}(T(\widetilde{B})) \to T(\widetilde{B})\times \Sigma$ using the vector field $T_\ast \chi$ on $T(\widetilde{B})$ (and possibly another metric $g_0^T$ satisfying the same conditions as those for $g_0$) in the same manner. 
Let $\Xi\in \xi$ and 
\[
\Psi_0=(\Phi_0^T)^{-1}\circ (T\times \Xi)\circ \Phi_0:f^{-1}(\widetilde{B})\to f^{-1}(T(\widetilde{B})). 
\]
One can easily check that $T\circ f=f\circ \Psi_0$ on $f^{-1}(\widetilde{B})$ and $\Psi_0|_\Sigma = \Xi$. 
For $t\in [-1,1]$ and $j=1,\ldots, d$, put $\beta_t=-1+t\sqrt{-1}$ and define $\mathcal{C}_{j,t},\mathcal{A}_{j,t}\subset \Sigma$ as follows:
{\allowdisplaybreaks
\begin{align*}
\mathcal{C}_{j,t}& = p_2\circ \Phi_0\circ \varphi_j^{-1}\left(\left\{\left.(z,w)\in q^{-1}(\beta_t)~\right|~|z|=|w|\right\}\right),\\
\mathcal{A}_{j,t}&= p_2\circ \Phi_0\circ \varphi_j^{-1}\left(\left\{(z,w)\in q^{-1}(\beta_t)~\left|~\sqrt{1+t^2}/4\leq|z|\leq 4\right.\right\}\right),
\end{align*}
}%
where $p_2:\widetilde{B}\times \Sigma \to \Sigma$ is the projection. 
We give the orientation of $\mathcal{C}_{j,t}$ so that the curve 
\[
t \mapsto p_2\circ \Phi_0\circ \varphi_j^{-1}(z\exp(\sqrt{-1}t),w\exp(-\sqrt{-1}t))
\]
on $\mathcal{C}_{j,t}$ is positively oriented for any $(z,w)\in \mathcal{C}_{j,t}$. 
Using $\Phi_0^T$ instead of $\Phi_0$, we also define $\mathcal{C}_{j,t}^T,\mathcal{A}_{j,t}^T\subset \Sigma$, and the orientation of $\mathcal{C}_{j,t}^T$ in the same way.  
Since $\mathcal{C}_{j,t}$ and $\mathcal{C}_{\sigma_\tau(j),t}^T$ are vanishing cycles of $f$ associated with $\gamma_j$ and $T(\gamma_j)$, respectively, $\Xi(\mathcal{C}_{j,t})$ is isotopic to $\mathcal{C}_{\sigma_\tau(j),t}^T$ by the assumption. 
As one can regard $\mathcal{A}_{j,t}$ (resp.~$\mathcal{A}_{j,t}^T$) as a tubular neighborhood of $\mathcal{C}_{j,t}$ (resp.~$\mathcal{C}_{j,t}^T$), by uniqueness of a tubular neighborhood, one can change $\Phi_0$ on $\bigsqcup_j f^{-1}(\iota_j([2\delta,1-\delta]\times [-1,1]))$ with $\Phi_0^T$ fixed (by modifying the metric $g_0$ with $g_0^T$ fixed) so that $\varphi_{\sigma_\tau(j)}\circ\Psi_0|_{\Phi_0^{-1}(\{\beta_t\}\times \mathcal{A}_{j,t})}\circ \varphi_j^{-1}(z,w)$ is equal to $(z,w)$ (resp.~$(w,z)$) for any $t\in [-1,1]$ and $(z,w)\in q^{-1}(\beta_t)\cap \mathbb{E}_4$ if the orientation of $\Xi(\mathcal{C}_{j,t})$ is equal to (resp.~opposite to) that of $\mathcal{C}_{\sigma_\tau(j),t}^T$. 
By the assumptions on metrics, $\Psi_0$ can be extended to
\[
\widetilde{\Psi}_0:f^{-1}\left(\widetilde{B}\cup \left(\bigsqcup_j \psi_j^{-1}(\mathbb{I})\right)\right)\to f^{-1}\left(T(\widetilde{B})\cup \left(\bigsqcup_j \psi_j^{-1}(\mathbb{I})\right)\right)
\]
so that $\varphi_{\sigma_\tau(j)}\circ \widetilde{\Psi}_0\circ \varphi_j^{-1}(z,w)$ is equal to either $(z,w)$ or $(w,z)$ for $(z,w)\in q^{-1}(\mathbb{I})\cap \mathbb{E}_4$ and $\widetilde{\Psi}_0$ is defined on the other part of $f^{-1}\left(\widetilde{B}\cup \left(\bigsqcup_j \psi_j^{-1}(\mathbb{I})\right)\right)$ by using the extensions of $\Phi_0$ and $\Phi_0^T$ (obtained by lifting the extensions of $\chi$ and $T_\ast \chi$). 

We take a simple smooth path $\alpha\subset \widetilde{B}\cup \left(\bigsqcup_j \psi_j^{-1}(\mathbb{I})\right)$ satisfying the following conditions. 


\begin{itemize}

\item 
$\alpha$ intersects the boundary of $\widetilde{B}\cup \left(\bigsqcup_j \psi_j^{-1}(\mathbb{I})\right)$ only at its ends transversely. 
Furthermore, $\Pa \alpha$ is contained in $\Pa D^2$.

\item 
$\alpha$ is away from $\{r_0\}\cup R$.

\item 
$\{r_0\}\cup R$ is contained in a connected component $R_1\subset D^2$ of $D^2\setminus \alpha$. 

\end{itemize}

\noindent
We also take a tubular neighborhood $\lambda:\alpha\times (-1,1)\to D^2$ so that $\lambda (\Pa \alpha\times (-1,1))\subset \Pa D^2$, $\lambda(\alpha\times (-1,0))$ is contained in $R_1$ and the image of $\lambda$ is away from $\{r_0\}\cup R$. 
Let $R_2 = (D^2\setminus R_1)\cup \lambda(\alpha\times (-1,1))$. 
Since $R_2$ and $T(R_2)$ are contractible and do not contain any critical value of $f$, one can take trivializations $\Phi_1:f^{-1}(R_2) \to R_2 \times \Sigma$ and $\Phi_1^T:f^{-1}(T(R_2))\to T(R_2)\times \Sigma$ of $f$ (as a $\Sigma$-bundle) so that $\Phi_1\circ \kappa_h^{-1}(x,t,s)=(x,\kappa_h(r_0,t,s))$ for $(x,t,s)\in R_2\times (-1,0]\times (\bigsqcup_k S^1)$ and $\Phi_1^T$ satisfies the same condition. 
We identify $\alpha$ and $T(\alpha)$ with $[-1,1]$ so that $T|_\alpha$ becomes the identity map on $[-1,1]$.  
For $s\in [-1,1]$($\cong \alpha$) and $t\in (-1,0]$, we take a diffeomorphism $\Xi_{s,t}:\Sigma \to \Sigma$ satisfying the following equality for any $x\in \Sigma$. 
\[
((T\circ \lambda)^{-1}\times \id_\Sigma)\circ \Phi_1^T\circ \widetilde{\Psi}_0\circ \Phi_1^{-1}\circ (\lambda \times \id_\Sigma)((s,t),x)= ((s,t), \Xi_{s,t}(x)). 
\]
Take a monotone-increasing smooth function $\varrho:[0,1]\to [0,1]$ so that $\varrho\equiv 0$ on $[0,1/4]$ and $\varrho\equiv 1$ on $[3/4,1]$. 
We can eventually define $\widetilde{T}:W\to W$ as follows. 

\begin{itemize}

\item 
$\widetilde{T} = \widetilde{\Psi}_0$ on $f^{-1}(R_1\setminus \lambda (\alpha\times (-1,0]))$.

\item 
$((T\circ \lambda)^{-1}\times \id_\Sigma)\circ \Phi_1^T\circ \widetilde{T}\circ \Phi_1^{-1}\circ (\lambda \times \id_\Sigma)((s,t),x)= ((s,t), \Xi_{\varrho(-t)s,\varrho(-t)t}(x)). 
$ for $s\in [-1,1]$ and $t\in (-1,0]$. 

\item 
$\Phi_1^T\circ \widetilde{T}\circ \Phi_1^{-1}(w,x)= (T(w),\Xi_{0,0}(x))$ for $w\in R_2\setminus \lambda(\alpha\times (-1,0])$ and $x\in \Sigma$. 

\end{itemize}

\noindent
The pair $(\widetilde{T},T)$ satisfies the conditions (A)--(C) in \Cref{thm:condition lift a braid}. 

In order to show the last statement of \Cref{thm:condition lift a braid}, suppose that $\tau$ and $\xi$ satisfy the condition (1) and both $(\widetilde{T}_0,T_0)$ and $(\widetilde{T}_1,T_1)$ satisfy the condition (A)--(C) for $\tau$ and $\xi$. 
By considering the pair $(\widetilde{T}_1^{-1}\circ \widetilde{T}_0,(T_1)^{-1}\circ T_0)$, one can assume $\tau =1$, $\xi =1$ and $(\widetilde{T}_0,T_0)=(\id,\id)$ without loss of generality.
In what follows, we put $(\widetilde{T},T)=(\widetilde{T}_1,T_1)$.
Furthermore, by modifying $\varphi_j$ and $\psi_j$ if necessary, one can assume that $\widetilde{T}$ (resp.~$T$) is the identity map on $\varphi_j^{-1}(\mathbb{E}_4)$ (resp.~$\psi_j^{-1}(\mathbb{D}_2)$). 
Since $T$ represents the unit element in $\mathcal{M}_\Pa(D^2,R)$, one can take an ambient isotopy $H:[0,1]\times D^2\to [0,1]\times D^2$ satisfying the following conditions.

\begin{itemize}

\item 
$H$ preserves $\Pa D^2$ and $\bigsqcup_j\psi_j^{-1}(\disk_1)$ pointwise. 

\item 
$H_0=\id_{D^2}$, where $H_t(x)=p_2\circ H(t,x)$ and $p_2:[0,1]\times D^2\to D^2$ is the projection. 

\item 
$H_1(T(x))=x$ for any $x\in D^2\setminus \bigsqcup_j \psi_j^{-1}(\mathbb{D}_2)$. 

\item 
There exists $n_j\in \Z$ such that $H_1(T(\psi_j^{-1}(z)))=\psi_j^{-1}(\exp(2\pi \sqrt{-1}n_j\varrho'(|z|))z)$ for any $z\in \mathbb{D}_2$, where $\varrho':\R\to \R$ is a monotone-decreasing function with $\varrho(r)=1$ for $r\leq 1$ and $\varrho(r)=0$ for $r\geq 2$. 

\end{itemize}
Let $\upsilon=H_\ast \left(\frac{\Pa}{\Pa t}\right)$ and $\widetilde{\upsilon}$ be the vector field on $[0,1]\times (W\setminus \Crit(f))$ satisfying $d(\id_{[0,1]}\times f)(\widetilde{\upsilon})=\upsilon$ and $dp_2(\widetilde{\upsilon})$ is contained in $(\Ker df)^\perp$ with respect to the metric $g_0$ of $W$ used in the proof of (1)$\Rightarrow$(2), where $p_2:[0,1]\times W\to W$ is the projection. 
As $H$ preserves $\bigsqcup_j\psi_j^{-1}(\disk_1)$ pointwise, $\upsilon$ is equal to $\left(\frac{\Pa }{\Pa t}\right)$ on $\bigsqcup_j\psi_j^{-1}(\disk_1)$. 
Thus, $\widetilde{\upsilon}$ can be extended on $[0,1]\times W$. 
We denote the extended vector field by the same symbol $\widetilde{\upsilon}$. 
Taking flows of $\widetilde{\upsilon}$, we obtain an ambient isotopy $\widetilde{H}:[0,1]\times W\to [0,1]\times W$ satisfying $d\widetilde{H}_{(t,x)}\left(\left(\frac{\Pa}{\Pa t}\right)\right) = \widetilde{\upsilon}_{\widetilde{H}(t,x)}$.
It is easy to check that $(\id_{[0,1]}\times f)\circ \widetilde{H}$ is equal to $H\circ (\id_{[0,1]}\times f)$, in particular $H$ preserves fibers of $f$.
We can further deduce that $\widetilde{H}_0\circ\widetilde{T}=\widetilde{T}$, $\widetilde{H}_1|_{\bigsqcup_jf^{-1}(\psi_j^{-1}(\disk_1))}=\id$ and 
\[
f\circ \widetilde{H}_1\circ \widetilde{T}(x)=H_1\circ f\circ \widetilde{T}(x)=H_1\circ T\circ f(x). 
\]
Thus, by replacing $\widetilde{T}$ and $T$ with $\widetilde{H}_1\circ \widetilde{T}$ and $H_1\circ T$, respectively, we can assume that $T$ is the identity map on $D^2\setminus \bigsqcup_j \psi_j^{-1}(\mathbb{D}_2)$ and $T(\psi_j^{-1}(z))=\psi_j^{-1}(\exp(2\pi \sqrt{-1}n_j\varrho'(|z|))z)$ for $z\in \mathbb{D}_2$. 

We define an ambient isotopy $H':[0,1]\times D^2\to [0,1]\times D^2$ from $\id_{D^2}$ to $T^{-1}$ as follows. 
\[
H'(t,x) = \begin{cases}
(t,x) & (x\in D^2\setminus \bigsqcup_i \psi_i^{-1}(\mathbb{D}_2)) \\
(t, \psi_j^{-1}(\exp(-2\pi \sqrt{-1}n_jt\varrho(|z|))z)) & (x= \psi_j^{-1}(z)\mbox{ for }z\in \disk_2).
\end{cases}
\]
We take an isotopy $\widetilde{H}':[0,1]\times (W\setminus \bigsqcup_j\varphi_j^{-1}(\mathbb{E}_4))\to [0,1]\times (W\setminus \bigsqcup_j\varphi_j^{-1}(\mathbb{E}_4))$ from $H'$ in the same way as construction of $\widetilde{H}$ from $H$. 
The following equality holds for $(z,w)\in q^{-1}(\disk_2)\setminus \mathbb{E}_4$. 
\[
\varphi_j\circ\widetilde{H}'_t\circ \varphi_j^{-1}(z,w)=\begin{cases}
(z,\exp(-2\pi \sqrt{-1}n_jt\varrho'(|zw|)w) & (|z|\geq 4)\\
(\exp(-2\pi \sqrt{-1}n_jt\varrho'(|zw|)z,w) & (|w|\geq 4). 
\end{cases}
\]
Let $\varrho'':\R\to \R$ be a smooth function with $\varrho''(r)=0$ for $r\leq 4$ and $\varrho''(r)=1$ for $r\geq 5$. 
We extend the restriction of $\widetilde{H}'$ on $W\setminus \bigsqcup_j\varphi_j^{-1}(q^{-1}(\mathbb{D}_2)\setminus \mathbb{E}_5)$ to $W$ so that it satisfies the following equality for $(z,w)\in \mathbb{E}_5$ and $j=1,\ldots, d$.
\begin{align*}
&\varphi_j\circ\widetilde{H}'_t\circ \varphi_j^{-1}(z,w)\\
=&\left(\exp(-2\pi \sqrt{-1}n_jt(1-\varrho''(|z|))\varrho'(|zw|))z,\exp(-2\pi \sqrt{-1}n_jt\varrho''(|z|)\varrho'(|zw|))w\right).
\end{align*}%
One can check that $\widetilde{H}'_0=\id_W$, $(\id_{[0,1]}\times f)\circ \widetilde{H}' = H'\circ (\id_{[0,1]}\times f)$, and $\widetilde{H}'_1$ is the identity map on $\bigsqcup_j\varphi_j^{-1}(\mathbb{E}_4)$. 
By replacing $\widetilde{T}$ and $T$ with $\widetilde{H}_1\circ \widetilde{T}$ and $H_1\circ T$, respectively, we can assume $T=\id_{D^2}$. 

We take $B, \gamma_j, \iota_j, \widetilde{B}$ and $\Phi_0$ as in the proof of (1)$\Rightarrow$(2). 
By the assumption, $\widetilde{T}|_{f^{-1}(r_0)}$is isotopic to the identity map.
Thus, one can isotope $\widetilde{T}$ so that it is the identity on $f^{-1}(B)$. 
Identifying $f^{-1}(\iota_j([\delta,1-\delta]\times [-1,1]))$ with $[\delta,1-\delta]\times [-1,1]\times \Sigma$ via $\Phi_0$ and $\iota_j\times \id_\Sigma$, we can regard the restriction of $\widetilde{T}$ on $[\delta,1-\delta]\times \{0\}\times \mathcal{A}_{j,0}$ as a one-parameter family of embeddings of the annulus $\mathcal{A}_{j,0}$ into $\Int(\Sigma)$. 
Since the restrictions of $\widetilde{T}$ on $\{\delta\}\times \{0\}\times \mathcal{A}_{j,0}$ and $\{1-\delta\}\times \{0\}\times \mathcal{A}_{j,0}$ are both the identity map, we obtain an element in $\pi_1(\Emb(\mathcal{A}_{j,0},\Sigma))$, where $\Emb(\mathcal{A}_{j,0},\Sigma)$ is the space of embeddings defined in \cite{Ivanov1998MCG}, whose base point is the inclusion. 

In what follows, we obtain a generator of $\pi_1(\Emb(\mathcal{A}_{j,0},\Sigma))$. 
By \cite[Theorem 2.6.A]{Ivanov1998MCG}, we obtain the following exact sequence. 
\begin{equation}\label{eqn:exact seq Emb}
\pi_1(\mathcal{D}_\Pa(\Sigma))\to \pi_1(\Emb(\mathcal{A}_{j,0},\Sigma))
\xrightarrow{\Pa} \mathcal{M}_\Pa(\Sigma\setminus \Int(\mathcal{A}_{j,0}))\xrightarrow{\iota}\mathcal{M}_\Pa(\Sigma). 
\end{equation}%
\noindent
Since $f$ is relatively minimal, $d>0$ and $(g,b)\neq (1,0)$, either $g\geq 2$ or $b\geq 1$ holds. (Recall that $g$ and $b$ are the genus and the number of boundary components of $\Sigma$, respectively.)
In particular, the connected component of $\mathcal
D_\Pa(\Sigma)$ containing the identity map is contractible (\cite{EE1969homotopytypediffeogrp,ES1970homotopytypediffeogrp}). 
One can thus deduce from \eqref{eqn:exact seq Emb} that the map $\Pa:\pi_1(\Emb(\mathcal{A}_{j,0},\Sigma))\to \Im \Pa =\Ker \iota$ is an isomorphism. 
The complement $\Sigma\setminus \Int(\mathcal{A}_{j,0})$ does not have any disk components since $f$ is relatively minimal. 
By \cite[Theorem 3.18]{FarbMargalitBookMCG}, $\Ker \iota$ is a cyclic group generated by $t_{c_1}t_{c_2}^{-1}$, where $c_1,c_2\subset \Sigma\setminus \Int(\mathcal{A}_{j,0})$ are simple closed curves parallel to the two boundary components of $\mathcal{A}_{j,0}$. 
It is easy to see that the one-parameter family $\Psi:[0,1]\times \mathcal{A}_{i,0}\to \mathcal{A}_{j,0}\subset \Sigma_g^b$ defined below represents the generator of $\pi_1(\Emb(\mathcal{A}_{j,0},\Sigma))$ corresponding to $t_{c_1}t_{c_2}^{-1}$ or its inverse. (Here, we identify $\mathcal{A}_{j,0}$ with $\left\{(z,w)\in \C^2~|~zw=1, 1/4\leq\|z\|\leq 4\right\}$ in the obvious way.)
\[
\Psi(t,(z,w)) = (\exp(2\pi\sqrt{-1}t)z,\exp(-2\pi\sqrt{-1}t)w). 
\]

By the observation above, we can change $\widetilde{T}$ by an ambient isotopy of $W$ preserving fibers of $f$ so that it satisfies the following conditions. 

\begin{itemize}

\item 
$\widetilde{T}$ is the identity map on $f^{-1}(B)\cup \bigsqcup_{j,t}([\delta,1-\delta]\times \{t\}\times \mathcal{A}_{j,t})$.

\item 
There exist $m_j\in \Z$ and a smooth function $\varrho_j:\R\to \R$ with $\varrho_j(r)=0$ (resp.~$\varrho_j(r)=m_j$) if $r\leq -1$ (resp.~$r\geq -1/2$) such that the following equality holds for $(z,w)\in \mathbb{E}_4$.
\begin{align*}
&\varphi_j\circ \widetilde{T}\circ \varphi_j^{-1}(z,w) \\
=&(\exp(2\pi\sqrt{-1}\varrho_j(\Re(zw)))z,\exp(-2\pi\sqrt{-1}\varrho_j(\Re(zw)))w).
\end{align*}

\end{itemize}

\noindent
Let $E_j = \varphi_j^{-1}(\mathbb{E}_4)\cap (\psi_j\circ f)^{-1}(\mathbb{I})\subset W$. 
We define an ambient isotopy $\Theta_j$ of $E_j$ by putting 
\begin{align*}
&\varphi_j\circ(\Theta_j)_t\circ \varphi_j^{-1}(z,w)\\
=& (\exp(2\pi\sqrt{-1}(1-t)\varrho_j(\Re(zw)))z,\exp(-2\pi\sqrt{-1}(1-t)\varrho_j(\Re(zw)))w)
\end{align*}%
for $(z,w)\in E_j$ and $t\in [0,1]$. 
It is easy to see that this isotopy can be extended to one on $(\psi_j\circ f)^{-1}(\mathbb{I})$. 
This isotopy changes $\widetilde{T}$ so that it is the identity map on 
\[
f^{-1}(B)\cup \bigsqcup_{j,t}\left([\delta,1-\delta]\times \{t\}\times \mathcal{A}_{j,t}\right)\cup \bigsqcup_j E_j. 
\]
Since the restriction of $f$ on 
\[
f^{-1}\left(\bigsqcup_j (\iota_j([\delta,1]\times [-1,1])\cup\psi_j^{-1}(\mathbb{I}))\right)\setminus\left( \bigsqcup_{j,t}\left([\delta,1-\delta]\times \{t\}\times \mathcal{A}_{j,t}\right)\cup \bigsqcup_j E_j\right)
\]
is a trivial bundle (with fiber $\Sigma\setminus \Int(\mathcal{A}_{j,0})$), one can further change $\widetilde{T}$ by a fiber-preserving ambient isotopy of $W$ so that it is the identity map on $f^{-1}(\widetilde{B}\cup \bigsqcup_j \psi_j^{-1}(\mathbb{I}))$.
Since the complement of $\widetilde{B}\cup \bigsqcup_j \psi_j^{-1}(\mathbb{I})$ in $D^2$ is contractible and does not contain any critical value of $f$, the map $f$ is a trivial $\Sigma$-bundle on this complement, and thus $\widetilde{T}$ can be changed to the identity map on the whole of $W$. 
This completes the proof of \Cref{thm:condition lift a braid}. 
\end{proof}

\begin{remark}\label{rem:assumption not torus in lifting braid}

The assumption $(g,b)\neq (1,0)$ in \Cref{thm:condition lift a braid} is used to show that $\pi_1(\Emb(\mathcal{A}_{j,0},\Sigma))$ is a cyclic group. 
This does not hold if $(g,b)=(1,0)$. 
In this case, the map $\pi_1(\mathcal{D}_\Pa(\Sigma))\to \pi_1(\Emb(\mathcal{A}_{j,0},\Sigma))$ in \eqref{eqn:exact seq Emb} is an isomorphism, and the map $\Z\oplus \Z\cong\pi_1(\Sigma)\to \pi_1(\mathcal{D}_\Pa(\Sigma))$ induced by the group action of $\Sigma\cong T^2$ is also an isomorphism. 
The author believes that the uniqueness of $(\widetilde{T},T)$ in the case $(g,b)=(1,0)$ can be shown in a similar manner, that is, by making a given self-isomorphism of $f$ trivial on $f^{-1}(\widetilde{B}\cup \bigsqcup_j \psi_j^{-1}(\mathbb{I}))$ via a suitable fiber-preserving isotopy on $(\psi_j\circ f)^{-1}(\mathbb{I})$ (corresponding to $\Theta_j$).
However, for the sake of simplicity, the case $(g,b)=(1,0)$ is not pursued further, as this case is not required for the proofs of \Cref{thm:combinatorial construction LbF over D2timesD2,thm:combinatorial construction LbF over Hir_m,thm:combinatorial construction pencil-fibration} (cf.~the beginning of \Cref{sec:combinatorial construction pencil-fibration}).

\end{remark}

Let $f:W\to D^2$ be a Lefschetz fibration satisfying the last statement in \Cref{thm:condition lift a braid} and $r_0\in \Pa D^2$.
We define 
{\allowdisplaybreaks
\begin{align*}
\mathcal{D}_f(W)&=\{\widetilde{T}\in \mathcal{D}(W)~|~\exists T\in \mathcal{D}_\Pa(D^2,\Critv(f))\mbox{ s.t. } f\circ \widetilde{T}=T\circ f\}, \\
\mathcal{M}_f(W)&= \pi_0(\mathcal{D}_f(W)), \\
\mathcal{L}(f,r_0)&=\{(\tau,\xi) \in \mathcal{M}_\Pa(D^2,\Critv(f))\times \mathcal{M}(\Sigma)~|~\theta\circ \tau_\ast = \xi^{-1}\cdot \theta\cdot \xi\}.
\end{align*}
}%
By \Cref{thm:condition lift a braid}, we obtain a well-defined homomorphism $\ell_{f,r_0}:\mathcal{L}(f,r_0)\to \mathcal{M}_f(W)$ by putting $\ell_{f,r_0}(\tau,\xi)=[\widetilde{T}]$ for $(\tau,\xi)\in \mathcal{L}(f,r_0)$ and $\widetilde{T}\in \mathcal{D}_f(W)$ obtained by applying \Cref{thm:condition lift a braid} to $(\tau,\xi)$. 
Furthermore, it is easy to see that if $(\tau,1)\in \mathcal{L}(f,r_0)$ for some $r_0\in \Pa D^2$, then $(\tau,1)\in \mathcal{L}(f,r)$ for any $r\in \Pa D^2$. 
One can thus define a subgroup $\mathcal{L}_0(f)$ of $\mathcal{M}_\Pa(D^2,\Critv(f))$ by
\[
\mathcal{L}_0(f)=\{\tau\in \mathcal{M}_\Pa(D^2,\Critv(f))~|~(\tau,1)\in \mathcal{L}(f,r)\mbox{ for some }r\in \Pa D^2\}. 
\]
We call an element in $\mathcal{L}_0(f)$ a \emph{liftable braid} with respect to $f$, and define $\ell_f:\mathcal{L}_0(f)\to \mathcal{M}_f(W)$ as the restriction of $\ell_{f,r_0}$ for some $r_0$, which is independent of the choice of $r_0$. 
It is easy to check that $\tau_{\mu}^{\epsilon}$ is a liftable braid with respect to $f$ if $\mu$ satisfy the compatibility condition with respect to $\theta_f$ and $\epsilon$ (cf.~\cite{AurouxSympplecticMapsInvariant}).

\section{Combinatorial construction of pencil-fibration structures}\label{sec:combinatorial construction pencil-fibration}

This section is devoted to the proofs of \Cref{thm:combinatorial construction LbF over D2timesD2,thm:combinatorial construction LbF over Hir_m,thm:combinatorial construction pencil-fibration}, that is, constructions of various pencil-fibration structures from braid and fiber monodromies. 
Since the theorems become obvious when $d$ is equal to $0$ (the trivial pencil-fibration structure satisfies the desired conditions), we assume $d>0$ in the proofs below. 
In this case, $(g,b)\neq (0,0), (0,1)$ since $\theta(\alpha)$ is assumed to be a \textit{non-trivial} Dehn twist for a meridian loop $\alpha$.

If $(g,b)=(1,0)$, we can take a lift $\overline{\theta}:\pi_1(D^2\setminus R,r_0)\to \mathcal{M}_\Pa(\Sigma_1^1)$ of $\theta$ with respect to the capping homomorphism $\mathcal{M}_\Pa(\Sigma_1^1)\to \mathcal{M}(\Sigma_1)$, so that $\tau_{\mu_1}^{\epsilon_1}\cdots \tau_{\mu_n}^{\epsilon_n}$ and $\overline{\theta}$ are compatible. 
We can obtain a desired Lefschetz bifibration in \Cref{thm:combinatorial construction LbF over D2timesD2} by applying the theorem to $\overline{\theta}$, yielding a Lefschetz bifibration $\overline{F}:\overline{X}\to D^2\times D^2$ over $\rho$, and then gluing $D^2\times D^2\times D^2$ with $\overline{X}$ along $\Pa_h \overline{X}$.
As for \Cref{thm:combinatorial construction LbF over Hir_m,thm:combinatorial construction pencil-fibration}, we can also obtain desired bifibration structures by applying the theorems to $\overline{\theta}$, and then forgetting a section. 
It is therefore enough to show \Cref{thm:combinatorial construction LbF over D2timesD2} (resp.~\Cref{thm:combinatorial construction LbF over Hir_m,thm:combinatorial construction pencil-fibration}) under the assumption that $(g,b)\neq (1,0)$ (resp.~$(g,b)\neq (1,0),(0,2)$). 

\begin{proof}[Proof of \Cref{thm:combinatorial construction LbF over D2timesD2}]
By Theorem~\ref{thm:equivalence LF from van cyc}, there exists a Lefschetz fibration $h:W\to D^2$ such that $\Sigma := h^{-1}(r_0)$ is diffeomorphic to $\Sigma_g^b$, and $\theta_h$ is equal to $\theta$ under this diffeomorphism. 
For each critical point $x_j\in \Crit(h)$, we take orientation-preserving complex charts $\varphi_j:U_j\to \C^2$ and $\psi_j:V_j\to \C$ around $x_j$ and $h(x_j)$, respectively, so that $\psi_j\circ h\circ \varphi_j^{-1}(z_1,z_2)$ is equal to $z_1^2+z_2^2$. 
By the assumption, a braid $\tau_{\mu_i}^{\epsilon_i}$ is liftable with respect to $h$ for each $i=1,\ldots,n$.  
Let $(\widetilde{T}_i,T_i)$ be a self-isomorphism (with respect to $\kappa_h$) obtained by applying \Cref{thm:condition lift a braid} to $\tau = \tau_{\mu_i}^{\epsilon_i}$, $\overline{\xi}=\id_{\Sigma}$, and complex charts taken above.
We put $T_0=T_n\circ \cdots \circ T_1$ and $\widetilde{T}_0 = \widetilde{T}_n\circ \cdots \circ \widetilde{T}_1$. 


Let $E'\subset \R^2$ be the region given in \Cref{fig:original disk1} and take edges $\alpha_i,\beta_i\subset \Pa E'$ as shown in \Cref{fig:original disk1}. 
Let $E_0$ be a planar surface obtained by gluing $\alpha_i$ and $\beta_i$ by an orientation-reversing diffeomorphism $\lambda_i:\alpha_i\to \beta_i$ (see \Cref{fig:original disk2}).
For each $i=1,\ldots,n$, we glue $\alpha_i\times D^2\subset \Pa E'\times D^2$ with $\beta_i\times D^2$ by $\Lambda_i=\lambda_i\times T_i:\alpha_i\times D^2\to \beta_i\times D^2$, and denote the resulting $4$-manifold by $Z_0$. 
We can also obtain a $6$-manifold $X_0$ by gluing $\alpha_i\times W\subset \Pa E' \times W$ with $\beta_i\times W$ by $\widetilde{\Lambda}_i=\lambda_i\times \widetilde{T}_i$ for each $i$. 
Let $F_0:X_0\to Z_0$ and $\rho_0:Z_0\to E_0$ be the maps defined by $F_0([(w,x)])=[(w,h(x))]$ and $\rho_0([(w,y)])=[w]$, respectively. (Note that $X_0,Z_0,E_0$ are quotient spaces obtained from $E'\times W,E'\times D^2$, and $E'$, respectively.) 
We denote by $S_0,\ldots, S_n\subset \Pa E_0$ each component of $\Pa E_0$ as shown in \Cref{fig:original disk2}. 
Let $\nu S_i$ be a collar neighborhood of $S_i\subset \Pa E_0$. 
One can easily take diffeomorphisms $\widetilde{\iota}_i,\iota_i,\iota'_i$ making the following diagram commute:
\begin{equation}\label{eq:commutative diagram S_i}
\begin{CD}
(-1,0]\times (\R\times W) /\sim_{\widetilde{T}_i} @>\widetilde{\iota}_i >> (\rho_0\circ F_0)^{-1}(\nu S_i)\\
@V\id_{(-1,0]}\times \id_{\R}\times hVV @VV F_0 V\\
(-1,0]\times (\R\times D^2) /\sim_{T_i}  @>\iota_i >>\rho_0^{-1}(\nu S_i)\\
@V\id_{(-1,0]}\times \mbox{proj.}VV @VV \rho_0V\\
(-1,0]\times (\R/\Z)@>\iota'_i >>\nu S_i, 
\end{CD}
\end{equation}
where we denote by $(\R\times W) /\sim_{\widetilde{T}_i}$ the quotient space $(\R\times W)/((t+n,x)\sim(t,\widetilde{T}_i^n(x)))$, which also appears in the beginning of \Cref{sec:liftable briad}.
The same notation applies to $T_i$ and other self-diffeomorphisms introduced hereafter.
\begin{figure}[htbp]
\subfigure[The region $E'$.]{\includegraphics{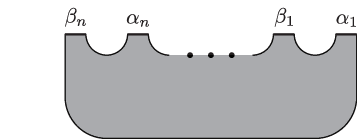}\label{fig:original disk1}}
\subfigure[The region $E$.]{\includegraphics{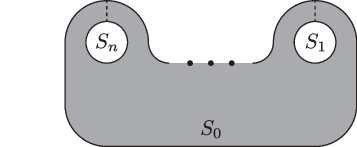}\label{fig:original disk2}}
\caption{Base spaces of bifibrations.}
\label{fig:original disk}
\end{figure}%
In what follows, for each $i=1,\ldots, n$, we construct a Lefschetz bifibration $F_i:X_i\to \disk_1\times D^2$ over the projection $\rho:\disk_1\times D^2\to \disk_1$ so that $F_i(\Delta_{F_i})=\{0\}$ and its braid monodromy along $\Pa \disk_1$ is $\tau_{\mu_i}^{\epsilon_i}$.
We then glue $F_i$ with $F_0$ along $S_i$. 
Let 
$c_i^1,c_i^2\subset \Sigma$ be the two vanishing cycles of the Lefschetz fibration $h:W\to D^2$ associated with a path $\mu_i$.

We first deal with the case $\epsilon_i=3$. 
We define $c:\C^3\to \C^2$ by $c(z_1,z_2,z_3)=(z_1,z_2^3+z_1z_2+z_3^2)$. 
Let $W_C=c^{-1}(\mathbb{D}_1 \times \mathbb{D}_{1})\cap \{|z_2|\geq 4\}$ and $A_C = \C \setminus \Int(\mathbb{D}_{2})$. 
We define a holomorphic function $\mathcal{R}:\C\setminus \{r\in \R~|~r\leq 0\} \to \C$ by $\mathcal{R}(z)=\sqrt{z}$ and $\mathcal{R}(1)=1$. 
Since $1>|w_1/q^4+w_2/q^6|$ for any $q\in A_C$ and $(w_1,w_2)\in \disk_{1}\times \disk_{1}$, we can define a local trivialization $\Phi_C:(\disk_{1}\times \disk_{1})\times A_C \to W_C$ of $c$ as follows.
\[
\Phi_C((w_1,w_2),q) = \left(w_1,-q^2,q^3\mathcal{R}\left(1+\frac{w_1}{q^4}+\frac{w_2}{q^6}\right)\right)
\]
for $(w_1,w_2)\in \disk_{1}\times \disk_{1}$ and $q\in A_C$. 

Since $c_i^1$ and $c_i^2$ intersect at a single point transversely, a regular neighborhood $\nu (c_i^1\cup c_i^2)$ of $c_i^1\cup c_i^2$ in $\Sigma$ has one boundary component.
Identifying a collar neighborhood of $\Pa(\nu (c_i^1\cup c_i^2))$ in $\Sigma'=\overline{\Sigma\setminus \nu (c_i^1\cup c_i^2)}$ with $A_C$, we can obtain a $6$-manifold $Y_i$ by gluing $(\mathbb{D}_{1}\times \mathbb{D}_1)\times \Sigma'$ with $c^{-1}(\mathbb{D}_{1}\times \mathbb{D}_1)$ using $\Phi_C$. 
We define $G_i:Y_i\to \mathbb{D}_{1}\times \mathbb{D}_1$ as follows. 
\[
G_i(x) = \begin{cases}
c(x) & (x\in c^{-1}(\mathbb{D}_{1}\times \mathbb{D}_1))\\
(w_1,w_2) & (x=((w_1,w_2),v)\in (\mathbb{D}_{1}\times \mathbb{D}_1)\times \Sigma'). 
\end{cases}
\]
The intersection $h^{-1}(\Pa\nu\mu_i)\cap \nu\Pa_hW$ can be identified with $\Pa (\nu\mu_i)\times \nu \Pa \Sigma$ via $\kappa_h$. 
We take diffeomorphisms $\xi:\Pa(\nu \mu_i)\to \Pa\disk_1$ and $\widetilde{\xi}:\disk_1\times h^{-1}(\Pa \nu\mu_i)\to G_i^{-1}(\disk_1\times \Pa \disk_1)$ so that 
\[
\widetilde{\xi} (w,y,z)=(w,\xi(y),z)\in (\disk_1\times \Pa \disk_1)\times \nu \Pa \Sigma \subset (\disk_1\times \Pa\disk_1)\times \Sigma'\subset G_i^{-1}(\disk_1\times \Pa \disk_1)
\]
for $(w,y,z)\in \disk_1\times (\Pa(\nu\mu_i)\times \nu\Pa\Sigma)\cong h^{-1}(\Pa\nu\mu_i)\cap \nu\Pa_hW$, and the following diagram commutes:
\[
\begin{CD}
\disk_1\times h^{-1}(\Pa \nu\mu_i) @>\widetilde{\xi}>> G_i^{-1}(\disk_1\times \Pa \disk_1)\\
@V\id_{\disk_1}\times hVV @VVG_iV\\
\disk_1 \times \Pa(\nu \mu_i) @>\id_{\disk_1}\times \xi>> \disk_1\times \Pa \disk_1\\
@V\mathrm{proj.}VV @VV\mathrm{proj.}V\\
\disk_1 @>\id_{\disk_1}>> \disk_1.
\end{CD}
\]
Using these maps, one can obtain a desired Lefschetz bifibration $F_i:X_i\to Z_i$ over $\rho_i:Z_i\to \disk_1$, where $X_i = (\disk_1\times h^{-1}(D^2\setminus\nu\mu_i))\cup_{\widetilde{\xi}}Y_i$, $Z_i = (\disk_1\times (D^2\setminus \nu\mu_i))\cup_{\id_{\disk_1}\times \xi}(\disk_1\times \disk_1)$, and $\rho_i$ is the projection to the first component. 

In what follows, we take coordinate neighborhoods around critical points and values of $F_i$ in which $F_i$ becomes standard, and examine behavior of the monodromy of $\rho_i\circ F_i$ along $\Pa D^2$ in the neighborhoods. 
(This process is necessary for applying \Cref{thm:condition lift a braid} to the monodromy of $F_i$, especially checking the condition (C) in the theorem.)
We define maps $\Upsilon_C^1:\R\times \C^2\ \to \Pa \disk_1\times \C^2$ and $\upsilon_C^+:\R\times \C\to \Pa \disk_1\times \C$ as follows. 
\begin{align*}
\Upsilon_C^1(\zeta,(z_2,z_3)) =& \left(\exp(2\pi \sqrt{-1}\zeta),z_2+\exp(\pi\sqrt{-1}\zeta)\frac{\sqrt{-1}}{\sqrt{3}},z_3\right),\\
\upsilon_C^+(\zeta,w_2) =& \left(\exp(2\pi\sqrt{-1}\zeta),w_2 +\exp(3\pi\sqrt{-1}\zeta)\frac{2\sqrt{-1}}{3\sqrt{3}}\right).
\end{align*}%
The following equality then holds for $(\zeta,z_2,z_3)\in \R\times \C^2$. 
\[
c\circ \Upsilon_C^1(\zeta,z_2,z_3)\\
=\upsilon_C^+\left(\zeta,z_2^2\left(z_2+\exp(\pi\sqrt{-1}\zeta)\sqrt{3}\sqrt{-1}\right) +z_3^2 \right).
\]
We can define $\Upsilon_C^2 :\R\times \disk_{\delta}\times \C\to \R\times \C^2$ for $\delta\ll 1$ as follows. 
{\allowdisplaybreaks
\begin{align*}
\Upsilon_C^2(\zeta,z_2,z_3) &=\left(\zeta, \sqrt{-1}z_2\mathcal{R}\left(-z_2-\exp(\pi\sqrt{-1}\zeta)\sqrt{3}\sqrt{-1}\right),z_3\right).
\end{align*}
}%
We can deduce from the inverse function theorem that $\Upsilon_C^2$ is a diffeomorphism on its image for a sufficiently small $\delta$. 
For such a $\delta$ and $(\zeta,z_2,z_3)$ in the image of $\Upsilon_C^2$, the following equality holds:
\[
(\upsilon_C^+)^{-1}\circ c \circ \Upsilon_C^1\circ (\Upsilon_C^2)^{-1}(\zeta,z_2,z_3)=(\zeta,z_2^2+z_3^2).
\] 
We put $\Upsilon_C^+ = \Upsilon_C^1\circ (\Upsilon_C^2)^{-1}$. 
Let $\tau_C:\R\times \C\to \R\times \C$ and $\widetilde{\tau}_C:\R\times \C^2\to \R\times \C^2$ be diffeomorphisms defined by $\tau_C(\zeta,w_2)=(\zeta,-w_2)$ and $\widetilde{\tau}_C(\zeta,z_2,z_3)=(\zeta,-z_2,\sqrt{-1}z_3)$, respectively.
We define $\Upsilon_C^-$ and $\upsilon_C^-$ as follows:
\[
\Upsilon_C^-(\zeta,z_2,z_3)=\widetilde{\tau}_C\circ \Upsilon_C^+(\zeta,-\sqrt{-1}z_2,-\sqrt{-1}z_3), \hspace{.3em}\upsilon_C^-(\zeta,w)=\tau_C\circ \upsilon_C^+(\zeta,-w). 
\]
These also satisfy the equality $(\upsilon_C^-)^{-1}\circ c \circ  \Upsilon_C^-(\zeta,z_2,z_3)=(\zeta,z_2^2+z_3^2)$.

It is easy to check that the restriction $h_i:=F_i|_{(\rho_i\circ F_i)^{-1}(1)}:(\rho_i\circ F_i)^{-1}(1)\to \rho_i^{-1}(1) \cong (D^2\setminus \nu\mu_i)\cup_\xi \disk_1$ is a Lefschetz fibration whose monodromy representation is isomorphic to $\theta$. 
In the total space $(\rho_i\circ F_i)^{-1}(1)=(\{1\}\times h^{-1}(D^2\setminus \nu\mu_i))\cup_{\widetilde{\xi}}G_i^{-1}(\{1\}\times \disk_1)$, there are $n-2$ critical points $(1,x_{j_1}),\ldots, (1,x_{j_{n-2}})$ in $\{1\}\times h^{-1}(D^2\setminus \nu\mu_i)$ (for some $j_1,\ldots, j_{n-2}\in \{1,\ldots, n\}$) and two critical points $x_C^+:=\Upsilon_C^+(0,0,0)$ and $x_C^-:=\Upsilon_C^-(0,0,0)$ in $G^{-1}(\{1\}\times \disk_1)$. 
We can take the complex charts $\varphi_{j_i}\circ p_2$ and $\psi_{j_i}\circ p_2$ around $(1,x_{j_i})$ and $h_i(1,x_{j_i})$, respectively, making $h_i$ the standard model $(z_1,z_2)\mapsto z_1^2+z_2^2$ (where $p_2$ is the projection). 
Furthermore, suitable restrictions of the maps $(\Upsilon_C^\pm)^{-1}$ and $(\upsilon_C^\pm)^{-1}$ (around $(0,0,0)$ and $(0,0)$, respectively) induce complex charts around $x_C^\pm$ and $h_i(x_C^\pm)$ with the same property, which we denote by $(\overline{\Upsilon_C^\pm})^{-1}$ and $(\overline{\upsilon_C^\pm})^{-1}$.
By Theorem~\ref{thm:equivalence LF from van cyc}, one can take an isomorphism $(\Phi_i,\Psi_i)$ from $h$ to $h_i$ so that their restrictions on neighborhoods of critical points and values are compositions of the given complex charts. 

Let $\gamma:\R\to \Pa \disk_1$ be the loop defined by $\gamma(\zeta)=\exp(2\pi\sqrt{-1}\zeta)$. 
It is easy to check that $\Upsilon_C^\pm$ (resp.~$\upsilon_C^\pm$) factors through the natural map $E(\gamma^\ast(\rho_i\circ F_i))\to (\rho_i\circ F_i)^{-1}(\Pa \disk_1)$ (resp.~$E(\gamma^\ast \rho_i)\to \rho_i^{-1}(\Pa \disk_1)$), where $E(\gamma^\ast(\rho_i\circ F_i))$ and $E(\gamma^\ast \rho_i)$ are the total spaces of the pull-backs. 
We also denote the induced maps to $E(\gamma^\ast (\rho_i\circ F_i))$ and $E(\gamma^\ast \rho_i)$ by $\Upsilon_C^\pm$ and $\upsilon_C^\pm$, respectively. 
We take vector fields $\chi$ and $\widetilde{\chi}$ on $E(\gamma^\ast \rho_i)$ and $E(\gamma^\ast(\rho_i\circ F_i))$ so that they satisfy the following conditions:

\begin{itemize}

\item 
$dF_i(\widetilde{\chi}) = \chi$, 

\item 
$d\rho_i(\chi) = \frac{d\gamma}{d\zeta}$, 

\item 
$\widetilde{\chi}=\frac{d}{d\zeta}$ on $\R\times h^{-1}(D^2\setminus \nu\mu_i)$, where we identify 
\[
\{(\zeta,\exp(2\pi\sqrt{-1}\zeta),x)\in \R\times \Pa \disk_1\times h^{-1}(D^2\setminus \nu\mu_i)\}\subset E(\gamma^\ast (\rho_i\circ F_i))
\]
with $\R\times h^{-1}(D^2\setminus \nu\mu_i)$ via the projection, 

\item 
$\chi = \frac{d}{d\zeta}$ on $\R\times (D^2\setminus \nu\mu_i)$, where we regard $\R\times (D^2\setminus \nu\mu_i)$ as a subset of $E(\gamma^\ast \rho_i)$ in the same way as above,

\item 
$\widetilde{\chi} = d\Upsilon_C^\pm\left(\frac{d}{d\zeta}\right)$ on a neighborhood of the image $\Upsilon_C^\pm(\R\times \{0\})$, 

\item 
$\chi = d\upsilon_C^\pm\left(\frac{d}{d\zeta}\right)$ on a neighborhood of the image $\upsilon_C^\pm(\R\times \{0\})$, 


\end{itemize}

\noindent
Taking time-$1$ flows of $\widetilde{\chi}$ and $\chi$, one can take a self-isomorphism $(\widetilde{T}_i',T_i')$ of $h_i$. 
The diffeomorphism $T_i'$ represents a braid monodromy of $F_i$ along $\gamma$. 
In particular $\widetilde{T}_i'$ preserves $(1,x_{j_1}),\ldots, (1,x_{j_{n-2}})$ and interchanges the two critical points $x_C^+$ and $x_C^-$.
For $(z_1,z_2)\in \C^2$ sufficiently close to the origin, the image $(\overline{\Upsilon_C^\pm})^{-1}\circ \widetilde{T}_i' \circ \overline{\Upsilon_C^\mp}(z_1,z_2)$ can be calculated as follows: 
{\allowdisplaybreaks
\begin{align*}
(\overline{\Upsilon_C^-})^{-1}\circ \widetilde{T}_i' \circ \overline{\Upsilon_C^+}(z_1,z_2)=& (\overline{\Upsilon_C^-})^{-1}\circ \Upsilon_C^+(1,z_1,z_2)\\
=& (\overline{\Upsilon_C^-})^{-1}\circ \Upsilon_C^1\circ (\Upsilon_C^2)^{-1}(1,z_1,z_2)\\
=& (\overline{\Upsilon_C^-})^{-1}\circ \Upsilon_C^1(1,w_1,z_2)\\
&\hspace{3em}(\mbox{where }\sqrt{-1}w_1\mathcal{R}(\sqrt{-1}\sqrt{3}-w_1)=z_1)\\
=& (\overline{\Upsilon_C^-})^{-1}\left(1,w_1-\frac{\sqrt{-1}}{\sqrt{3}},z_2\right)\\
=& (\overline{\Upsilon_C^-})^{-1}\circ \widetilde{\tau}_C\circ \Upsilon_C^1\left(0,-w_1,-\sqrt{-1}z_2\right)\\
=& \left(-\sqrt{-1}w_1r_{\sqrt{-1}}(-w_1+\sqrt{-1}\sqrt{3}),z_2\right)=\left(z_1,z_2\right), \\
(\overline{\Upsilon_C^+})^{-1}\circ \widetilde{T}_i' \circ \overline{\Upsilon_C^-}(z_1,z_2)=& (\overline{\Upsilon_C^+})^{-1}\circ \Upsilon_C^-(1,z_1,z_2)\\
=& (\overline{\Upsilon_C^+})^{-1}\circ \widetilde{\tau}_C\circ\Upsilon_C^1\circ (\Upsilon_C^2)^{-1}(1,-\sqrt{-1}z_1,-\sqrt{-1}z_2)\\
=& (\overline{\Upsilon_C^+})^{-1}\left(1,-w_1'+\frac{\sqrt{-1}}{\sqrt{3}},z_2\right)\\
&\hspace{3em}(\mbox{where }\sqrt{-1}w_1'\mathcal{R}(\sqrt{-1}\sqrt{3}-w_1')=-\sqrt{-1}z_1)\\
=& (\overline{\Upsilon_C^+})^{-1}\circ \Upsilon_C^1\left(0,-w_1',z_2\right)=(-z_1,z_2).
\end{align*}
}%
Since the restriction $T_i'|_{D^2\setminus \nu\mu_i}$ is the identity map, one can show that the composition $\Psi_i^{-1}\circ T_i'\circ \Psi_i$ is isotopic to $T_i$ (both of them represent $\tau_{\mu_i}^3$). 
One can further show that both of the pairs $(\widetilde{T}_i,T_i)$ and $(\Phi_i^{-1}\circ \widetilde{T}_i'\circ \Phi_i,\Psi_i^{-1}\circ T_i'\circ \Psi_i)$ satisfy the conditions (A)--(C) in Theorem~\ref{thm:condition lift a braid} for $f=h$, complex charts $\{(\varphi_i,\psi_i)\}_i$, $\tau = \tau_{\mu_i}^3$ and $\xi=1$ (cf.~\Cref{rem:cpx coordinates in lifting a braid}).
By Theorem~\ref{thm:condition lift a braid}, one can take a fiber-preserving isotopy pair $(\mathfrak{H}:\R\times W\to \R\times W,\mathfrak{h}:\R\times D^2\to \R\times D^2)$ from $(\widetilde{T}_i^{-1},T_i^{-1})$ to $(\Phi_i^{-1}\circ (\widetilde{T}_i')^{-1}\circ \Phi_i,\Psi_i^{-1}\circ (T_i')^{-1}\circ \Psi_i)$.

The vertical boundaries $\Pa_v X_i$ and $\Pa_v Z_i$ are diffeomorphic to $(\R\times (\rho_i\circ F_i)^{-1}(1))/\sim_{\widetilde{T}_i'}$ and $(\R\times \rho_i^{-1}(1))/\sim_{T_i'}$, respectively. 
Furthermore, under the identifications by these diffeomorphisms, $F_i:\Pa_v X_i\to \Pa_v Z_i$ and $\rho_i:\Pa_v Z_i\to \Pa \disk_1$ coincide with $\id_{\R}\times h_i$ and $[(\zeta,y)]\mapsto \exp(2\pi \sqrt{-1}\zeta)$, respectively.
Thus, the following diagram commutes (see \eqref{eq:commutative diagram S_i} for the diffeomorphisms in the left side): 
\[
\begin{CD}
(\rho_0\circ F_0)^{-1}(S_i) \cong(\R\times W) /\sim_{\widetilde{T}_i'} @> (\id_{\R}\times \Phi_i)\circ \mathfrak{H}>> (\R\times (\rho_i\circ F_i)^{-1}(1)) /\sim_{\widetilde{T}_i}\cong \Pa_v X_i\\
@VF_0VV @VVF_iV\\
\rho_0^{-1}(S_i) \cong(\R\times D^2) /\sim_{T_i'} @>(\id_{\R}\times \Phi_i)\circ \mathfrak{h}>> (\R\times \rho_i^{-1}(1)) /\sim_{T_i} \cong \Pa_v Z_i\\
@V\rho_0VV @VV\rho_iV\\
S_i\cong\R/\Z @>\exp(2\pi\sqrt{-1}~\cdot~)>> \Pa \disk_1. 
\end{CD}
\]
We can eventually glue $F_i$ with $F_0$ along $S_i$ by the horizontal diffeomorphisms above.

We next deal with the case $\epsilon_i=2$. 
We define $f:\C^3\to \C^2$ by $f(z_1,z_2,z_3)=(z_1,z_2z_3)$.
Let $W_I=f^{-1}(\disk_1\times \disk_1)\cap \{(z_1,z_2,z_3)\in \C^3~|~|z_2|\geq 2\mbox{ or }|z_3|\geq 2\}$
and $A_I^2=A_I^3=\C \setminus \Int(\mathbb{D}_{2})$. 
We can define a local trivialization $\Phi_I:(\disk_{1}\times \disk_{1})\times (A_I^2\sqcup A_I^3)\to W_I$ of $f$ as follows: 
\[
\Phi_I((w_1,w_2),z) = \begin{cases}
(w_1,z,w_2/z) & (z\in A_I^2) \\
(w_1,w_2/z,z) & (z\in A_I^3). 
\end{cases}
\]
We define a linear map $a: \C^2\to \C^2$ by $a(z_1,z_2)=(z_1,z_2+z_1/2)$ and let
\[
\widetilde{W}_I = (a\circ f)^{-1}(\disk_{1}\times \disk_{1})\cap \{(z_1,z_2,z_3)\in \C^3~|~|z_2|\geq 2\mbox{ or }|z_3|\geq 2\}.
\] 
One can easily check that $\Critv(a\circ f)=\{(z,z/2)\in \C^2~|~z\in \C\}$ and the map 
\[
\widetilde{\Phi}_I:(\disk_{1}\times \disk_{1})\times (A_I^2\sqcup A_I^3)\to \widetilde{W}_I
\]
defined by 
\[
\widetilde{\Phi}_I((w_1,w_2),z) =\begin{cases}
(w_1,z,(w_2-w_1/2)/z) &(z\in A_I^2)\\
(w_1,(w_2-w_1/2)/z,z) & (z\in A_I^3)
\end{cases} 
\]
is a local trivialization of $a \circ f$.
In the same way as the construction for a cusp (i.e.~the case $\epsilon_i=3$), we can obtain a $6$-manifold $Y_i$ by gluing $(\disk_{1}\times \disk_{1})\times (\Sigma\setminus (\nu (c_i^1\sqcup c_i^2)))$ with $f^{-1}(\disk_{1}\times \disk_{1})\sqcup (a \circ f)^{-1}(\disk_{1}\times \disk_{1})$ using $\Phi_I\sqcup \widetilde{\Phi}_I$, define a map $G_i:Y_i\to \disk_{1}\times \disk_{1}$, and obtain a desired Lefschetz bifibration $F_i:X_i\to Z_i$. 
%
%
%
%
Furthermore, one can easily take an embedding $\Upsilon_I^\pm:\Pa\disk_1\times U\to G^{-1}(\Pa \disk_1\times \disk_1)\subset X_i$ (resp.~$\upsilon_I^\pm:\Pa\disk_1\times V\to \Pa \disk_1\times \disk_1\subset Z_i$) to a neighborhood of a component of the critical point set (resp.~critical value set) of $F_i:X_i\to Z_i$ (where $U\subset \C^2$ and $V\subset \C$ are neighborhoods of the origins) so that $(\upsilon_I^\pm)^{-1}\circ F_i \circ \Upsilon_I^\pm(s,z,w)=(s,zw)$.  
Using these embeddings (instead of $\Upsilon_C^\pm$ and $\upsilon_C^\pm$), one can glue $F_i$ with $F_0$ along $S_i$ in the same way as that for the case $\epsilon_i=3$. 
We can also construct a desired $F_i$ and glue it with $F_0$ for $i\in \{1,\ldots, n\}$ with $\epsilon_i=-2$ in the same manner.

We next deal with the case $\epsilon_i=1$. 
We define $b:\C^3\to \C^2$ by $b(z_1,z_2,z_3)=(z_1^2+z_2z_3, 2z_1)$. Let
\[
W_B=b^{-1}(\disk_{1}\times \disk_{1})\cap \{(z_1,z_2,z_3)\in \C^3~|~|z_2|\geq 2\mbox{ or }|z_3|\geq 2\}, 
\]
and $A_B^2=A_B^3=\C\setminus \Int(\disk_{2})$.
We can define a local trivialization $\Phi_B:(\disk_{1}\times \disk_{1})\times (A_B^2\sqcup A_B^3)\to W_B$ of $b$ as follows: 
\[
\Phi_B((w_1,w_2),z) = \begin{cases}
(w_2/2,z,(w_1-w_2^2/4)/z) & (z\in A_B^2) \\
(w_2/2,(w_1-w_2^2/4)/z,z) & (z\in A_B^3). 
\end{cases}
\]
As we did for a cusp and a double point, we can obtain a $6$-manifold $Y_i$ by gluing $(\disk_1\times \disk_1)\times (\Sigma\setminus \nu c_i^1)$ with $b^{-1}(\disk_1\times \disk_1)$ using $\Phi_B$ (note that $c_i^1$ and $c_i^2$ coincide in this case), define a map $G_i:Y_i\to \disk_1\times \disk_1$, and obtain a desired Lefschetz bifibration $F_i:X_i\to Z_i$. 

For sufficiently small neighborhoods $U\subset \C^2$ and $V\subset \C$ of the origins, we define maps $\Upsilon_B^+:\R\times U\to b^{-1}(\Pa \disk_1\times \C)$ and $\upsilon_B^+:\R\times V\to \Pa \disk_1\times \C$ as follows: 
{\allowdisplaybreaks
\begin{align*}
\Upsilon_B^+(s,z_2,z_3) =& \left(\exp(\pi\sqrt{-1}s)\mathcal{R}\left(1-z_2z_3\exp(-2\pi\sqrt{-1}s)\right),z_2,z_3\right), \\
\upsilon_B^+(s,w)=&\left(\exp(2\pi\sqrt{-1}s),\frac{\exp(\pi\sqrt{-1}s)}{2}\mathcal{R}\left(1-w\exp(-2\pi\sqrt{-1}s)\right)\right). 
\end{align*}
}%
One can easily check that these maps satisfy the equality $b \circ \Upsilon_B^+(s,z_2,z_3)=\upsilon_B^+(s,z_2z_3)$. 
Let $\tau_B:\R\times \C\to \R\times \C$ and $\widetilde{\tau}_B:\R\to \C^2\to \R\times \C^2$ be diffeomorphisms defined by $\widetilde{\tau}_B(s,z_2,z_3)=(s,-z_2,z_3)$ and $\tau_B(s,w)=(s,-w)$, respectively.
We put $\Upsilon_B^-=\widetilde{\tau}_B\circ \Upsilon_B^+$ and $\upsilon_B^-=\tau_B\circ \Upsilon_B^+$. 
Using the embeddings $\Upsilon_B^\pm$ and $\upsilon_B^\pm$, one can glue $F_i$ with $F_0$ along $S_i$ in the same way as that for the case $\epsilon_i=3$.

We have obtained the $6$-manifold $\widetilde{X}_0 := X_0 \cup (\bigsqcup_i X_i)$, the $4$-manifold $\widetilde{Z}_0 := Z_0\cup (\bigsqcup_i Z_i)$, the Lefschetz bifibration $\widetilde{F}_0 := F_0\cup (\bigsqcup_i F_i):\widetilde{X}_0\to \widetilde{Z}_0$ over $\widetilde{\rho}_0:=\rho_0\cup (\bigsqcup_i \rho_i):\widetilde{Z}_0\to \widetilde{E}_0:=E_0\cup (\bigsqcup_i \disk_1)\cong D^2$. 
Since $\widetilde{E}_0$ is diffeomorphic to $D^2$ and $\widetilde{\rho}_0$ is a disk-bundle over $\widetilde{E}_0$, $\widetilde{Z}_0$ is diffeomorphic to $D^2\times D^2$ and $\widetilde{\rho}_0$ is trivial. 
We thus complete the proof of \Cref{thm:combinatorial construction LbF over D2timesD2} ($\widetilde{F}_0$ is a desired Lefschetz bifibration). 
\end{proof}

\begin{proof}[Proof of \Cref{thm:combinatorial construction LbF over Hir_m}]
We denote by $A_k$ the component of $\Pa \Sigma$ corresponding to that of $\Pa \Sigma_g^b$ parallel to $\delta_k$. 
We identify a collar neighborhood $\nu A_k$ of $A_k$ with $(-1,0]\times (\R/\Z)$, and denote by $(s,[\zeta])_k$ by the element of it. 
Let $\rho:[-1,0]\to \R$ be a monotone-increasing smooth function such that $\rho \equiv 0$ (resp.~$\rho\equiv 1$) on a neighborhood of $-1\in [-1,0]$ (resp.~$0\in [-1,0]$).
We define a diffeomorphism $\phi \in \mathcal{D}_\Pa(\Sigma)$ as follows: 
\[
\phi(z) = \begin{cases}
z & (z\not\in \nu \Pa \Sigma = \bigsqcup_k \nu A_k) \\
(u,[\zeta-\rho(u)a_k])_k & (z = (u,[\zeta])_k \in \nu A_k). 
\end{cases}
\]
Note that $\phi$ represents the product $t_{\delta_1}^{a_1}\cdots t_{\delta_b}^{a_b}$.
Construction of $h$ (i.e.~the proof of Theorem~\ref{thm:equivalence LF from van cyc}) implies that one can take neighborhoods of boundaries and diffeomorphisms to them as follows:

\begin{itemize}

\item 
a collar neighborhood $\nu\Pa D^2$ of $\Pa D^2$ and a diffeomorphism $\kappa':(-1,0]\times (\R/\Z) \to \nu \Pa D^2$ so that $\kappa'(0,[0])=r_0$, 

\item 
a neighborhood $\nu \Pa_h W$ of $\Pa_hW$, and a diffeomorphism 
\[
\kappa_h:D^2\times \nu\Pa \Sigma \cong D^2\times \left(\bigsqcup_{k=1}^b(-1,0]\times (\R/\Z)\right) \to \nu \Pa_h W
\]
so that $h\circ \kappa_h$ is the projection to the first $D^2$-component, 

\item 
the neighborhood $\nu \Pa_v W=h^{-1}(\nu \Pa D^2)$ of $\Pa_vW$, and a diffeomorphism
\[
\kappa_v : (0,1]\times (\R\times \Sigma)/\sim_{\phi}\to \nu\Pa_v W,
\]
so that $h\circ \kappa_v(s,[\zeta,y]) = \kappa'(s,[\zeta])$ and
\[
\kappa_h (\kappa'(s,[\zeta_1]),(u,[\zeta_2])_k) = \kappa_v(s,[\zeta_1,(u,[\zeta_2+(\rho(u)-1)\zeta_1a_k])_k]).
\] 

\end{itemize} 

\noindent

Let $\widetilde{F}_0:\widetilde{X}_0\to \widetilde{Z}_0$ be the Lefschetz bifibration constructed in the proof of \Cref{thm:combinatorial construction LbF over D2timesD2} from the given monodromies. 
We first glue the trivial Lefschetz bifibration $\disk_1\times W\to \disk_1\times D^2\to \disk_1$ to  $\widetilde{F}_0$ along $S_0 = \Pa \widetilde{E}_0$. 
We define a diffeomorphism $T_0':D^2\to D^2$ as follows: 
\[
T_0'(y) = \begin{cases}
y & (y\not \in \nu \Pa D^2) \\
\kappa'(s,[\zeta - \rho(s)m]) & (y = \kappa'(s,[\zeta]) \in\nu \Pa D^2 \mbox{ for }(s,[\zeta])\in (-1,0]\times(\R/\Z)). 
\end{cases}
\]
This map represents the mapping class $t_{\Pa}^m\in B_d$. 
We also define diffeomorphisms $\widetilde{T}_0^h,\widetilde{T}_0^v:W\to W$ by the following conditions:

\begin{itemize}

\item 
The supports of $\widetilde{T}_0^h$ and $\widetilde{T}_0^v$ are respectively contained in $\nu\Pa_h W$ and $\nu \Pa_v W$, 

\item 
$\widetilde{T}_0^h(\kappa_h(y,z)) = \kappa_h(y,\phi^m(z))$ for $(y,z)\in D^2\times \nu \Pa \Sigma$ (note that the support of $\phi$ is contained in $\nu \Pa \Sigma$), 

\item 
$\widetilde{T}_0^v(\kappa_v(s,[\zeta,z])) = \kappa_v(s,[\zeta-\rho(s)m,z])$ for $(s,[\zeta,z])\in (-1,0]\times (\R\times \Sigma)/\sim_{\id}$. 

\end{itemize}

\noindent
It is easy to check that $(\widetilde{T}_0^h,\id_{D^2})$ and $(\widetilde{T}_0^v,T_0')$ are both self-isomorphisms of $h:W\to D^2$, and thus so is the composition $(\widetilde{T}_0^v\circ \widetilde{T}_0^h,T_0')$. 
Furthermore, the following holds for $(s,[\zeta_1])\in (-1,0]\times (\R/\Z)$ and $(u,[\zeta_2])_k\in A_k\subset \nu\Pa \Sigma$.
{\allowdisplaybreaks
\begin{align*}
&\widetilde{T}_0^v\circ \widetilde{T}_0^h\bigl(\kappa_h \bigl(\kappa'(s,[\zeta_1]),(u,[\zeta_2])_k\bigr)\bigr) \\
=& \widetilde{T}_0^v\bigl(\kappa_h \bigl(\kappa'(s,[\zeta_1]),\phi^m((u,[\zeta_2])_k)\bigr)\bigr)\\
=&\widetilde{T}_0^v\bigl(\kappa_h \bigl(\kappa'(s,[\zeta_1]),(u,[\zeta_2-m\rho(u)a_i])_k\bigr)\bigr)\\
=&\widetilde{T}_0^v\bigl(\kappa_v\bigl(s,[\zeta_1,(u,[\zeta_2+(\rho(u)-1)\zeta_1a_k-m\rho(u)a_k])_k]\bigr)\bigr)\\
=&\kappa_v\bigl(s,[\zeta_1-\rho(s)m,(u,[\zeta_2+(\rho(u)-1)\zeta_1a_k-m\rho(u)a_k])_k]\bigr)\\
=&\kappa_h\bigl(\kappa'(s,[\zeta_1-\rho(s)m]),(u,[\zeta_2+(\rho(u)-1)(\rho(s)-1)ma_k])_k\bigr).
\end{align*}
}%
Hence, $\widetilde{T}_0^v\circ \widetilde{T}_0^h$ is the identity map on $\Pa_v W$, especially on the reference fiber $\Sigma$. 
By Theorem~\ref{thm:condition lift a braid} (in particular uniqueness of $\widetilde{T}$), there exists a fiber-preserving isotopy pair $(\mathfrak{H},\mathfrak{h})$ from $(\widetilde{T}_0,T_0)$ to $(\widetilde{T}_0^v\circ \widetilde{T}_0^h,T_0')$ making the following diagram commute:
\begin{equation*}
\begin{CD}
(\R\times W) /\sim_{\widetilde{T}_0^v\circ \widetilde{T}_0^h} @>\widetilde{\iota}_0\circ \mathfrak{H}>>  \Pa_{vv}\widetilde{X}_0 \\
@VV\id_{\R}\times hV @V\widetilde{F}_0VV \\
(\R\times D^2) /\sim_{T_0'} @>\iota_0\circ \mathfrak{h} >> \Pa_v\widetilde{Z}_0\\
@VVp_1V @V\widetilde{\rho}_0VV \\
\R/\Z @>\iota_0' >> \Pa \widetilde{E}_0, 
\end{CD}
\end{equation*}
where $\widetilde{\iota}_0,\iota_0,\iota_0'$ are (the restrictions of) diffeomorphisms given in \eqref{eq:commutative diagram S_i}. 
(Note that $\Pa \widetilde{E}_0 = S_0$, $\widetilde{F}_0 = F_0$, and $\widetilde{\rho}_0=\rho_0$ on the preimages of $\Pa \widetilde{E}_0$.)
We take a monotone non-decreasing smooth odd function $\widetilde{\varrho}:\R\to \R$ so that $\widetilde{\varrho}(\zeta+1)=\widetilde{\varrho}(\zeta)+1$ for any $\zeta\in \R$ and $\widetilde{\varrho}\equiv 0$ on a neighborhood of $0\in \R$. 
Using this function, we define a diffeomorphism $\mathfrak{g}_0:\R\times D^2\to \R\times D^2$ so that its support is contained in $\R\times\nu\Pa D^2$ and $\mathfrak{g}_0(\zeta_1,\kappa'(s,[\zeta_2])) = (\zeta_1,\kappa'(s,[\zeta_2+m\widetilde{\varrho}(\zeta_1)\rho(s)]))$ for $\zeta_i\in \R$ and $s\in (-1,0]$. 
We further define diffeomorphisms $\mathfrak{G}_0^v:\R\times W\to \R\times W$ and $\mathfrak{G}_0^h:\R\times W\to \R\times W$ so that

\begin{itemize}

\item 
the support of $\mathfrak{G}_0^h$ is contained in $\R\times \nu\Pa_h W$, while that of $\mathfrak{G}_0^v$ is contained in $\R\times \nu\Pa_vW$. 

\item 
$\mathfrak{G}_0^h(\zeta_1,\kappa_h(y,(u,[\zeta_2])_k)) = (\zeta_1,\kappa_h(y,(u,[\zeta_2+m\rho(u)\widetilde{\varrho}(\zeta_1)a_k])_k))$ for $\zeta_i\in \R$, $y\in D^2$ and $u\in (-1,0]$. 

\item 
$\mathfrak{G}_0^v(\zeta_1,\kappa_v(s,[\zeta_2,z])) = (\zeta_1,\kappa_v(s,[\zeta_2+m\rho(s)\widetilde{\varrho}(\zeta_1),z]))$ for $\zeta_i\in \R$, $z\in \Sigma$ and $s\in (-1,0]$. 

\end{itemize}


\noindent
It is easy to check that the maps $\mathfrak{g}_0$ and $\mathfrak{G}_0^v\circ \mathfrak{G}_0^h$ induce the horizontal diffeomorphisms (which are denoted by the same symbols) in the following commutative diagram:
\[
\begin{CD}
(\R\times W)/\sim_{\id_{W}} @> \mathfrak{G}_0^v\circ \mathfrak{G}_0^h>> (\R\times W)/\sim_{\widetilde{T}_0^v\circ \widetilde{T}_0^h} \\
@V\id_\R \times hVV @VV\id_\R\times hV \\
(\R\times D^2)/\sim_{\id_{D^2}} @>\mathfrak{g}_0>>(\R\times D^2)/\sim_{T_0'}\\
@Vp_1VV @VVp_1V \\
\R/\Z @>\id_{\R/\Z}>> \R/\Z. 
\end{CD}
\]
Under the identifications $\Pa\disk_1\times W \cong (\R\times W) /\sim_{\id_W}$, $\Pa\disk_1\times D^2\cong (\R\times D^2)/\sim_{\id_{D^2}}$, and $\Pa \disk_1 \cong \R/\Z$, one can define diffeomorphisms $\widetilde{\Phi}_0:\Pa \disk_1\times W\to \Pa_{vv}\widetilde{X}_0$, $\Phi_0:\Pa \disk_1\times D^2\to \Pa_v\widetilde{Z}$, and $\phi_0:\Pa \disk_1\to \Pa\widetilde{E}_0$ as follows: 
{\allowdisplaybreaks
\begin{align*}
\widetilde{\Phi}_0([\zeta,w])&=\widetilde{\iota}_{0}\circ \mathfrak{H}\circ \mathfrak{G}_0^v\circ \mathfrak{G}_0^h([-\zeta,w])\\
\Phi_0([\zeta,y])&=\iota_0\circ \mathfrak{h}\circ \mathfrak{g}_0([-\zeta,y])\\
\Phi_0'([\zeta])&=\iota_0'([-\zeta]). 
\end{align*}
}%
\noindent
Using these diffeomorphisms, we can obtain the following Lefschetz bifibration:
\[
\widetilde{X}_1:=\widetilde{X}_0\cup_{\widetilde{\Phi}_0}(\disk_1\times W) \xrightarrow{\widetilde{F}_1} \widetilde{Z}_1:=\widetilde{Z}_0\cup_{\Phi_0}(\disk_1\times D^2) \xrightarrow{\widetilde{\rho}_1} \widetilde{E}_1:=\widetilde{E}_0\cup_{\phi_0} \disk_1 \cong \PP^1. 
\]

By the construction, $\widetilde{\rho}_1:\widetilde{Z}_1\to \widetilde{E}_1$ is a disk bundle over $\widetilde{E}_1\cong \PP^1$ with the Euler number $m$. 
Let $\widetilde{\rho}_1':\widetilde{Z}_1'\to \widetilde{E}_1$ be a disk-bundle with the Euler number $-m$, and $\Phi_1:\Pa\widetilde{Z}_1'\to \Pa \widetilde{Z}_1$ be an orientation-reversing isomorphism (as an $S^1$-bundle). 
%
%
By the assumption, $2-2g-b <0$ and thus a connected component of $\mathcal{D}(\Sigma)$ is contractible (\cite{EE1969homotopytypediffeogrp,ES1970homotopytypediffeogrp}).
Thus, the isomorphism class of $\widetilde{F}_1:\Pa_{vh} \widetilde{X}_1\to \Pa \widetilde{Z}_1$ (as a $\Sigma$-bundle) is uniquely determined from its monodromy along a fiber in $\Pa\widetilde{Z}_1$ (as an $S^1$-bundle), which is equal to $t_{\delta_1}^{a_1}\cdots t_{\delta_b}^{a_b}$. 
Since this product is in the kernel of the forgetting homomorphism $\mathcal{M}_\Pa(\Sigma)\to \mathcal{M}(\Sigma)$, $\widetilde{F}_1:\Pa_{vh} \widetilde{X}_1\to \Pa \widetilde{Z}_1$ is a trivial $\Sigma$-bundle. 
Let $\widetilde{X}_1'= \widetilde{Z}_1'\times \Sigma$.
We can take an orientation-reversing diffeomorphism $\widetilde{\Phi}_1$ making the following diagram commute: 
\[
\begin{CD}
\Pa_{vh}\widetilde{X}_1'= \Pa \widetilde{Z}_1'\times \Sigma @>\widetilde{\Phi}_1>> \Pa_{vh}\widetilde{X}_1 \\
@V\mathrm{proj.}VV @VV\widetilde{F}_1V \\
\Pa \widetilde{Z}_1' @>\Phi_1>> \Pa \widetilde{Z}_1. 
\end{CD}
\] 
We can eventually obtain the following Lefschetz bifibration:
\[
\widetilde{X}_2:= \widetilde{X}_1\cup_{\widetilde{\Phi}_1}\widetilde{X}_1' \xrightarrow{\widetilde{F}_2:=\widetilde{F}_1\cup \mathrm{proj.}} \widetilde{Z}_2 := \widetilde{Z}_1\cup_{\Phi_1}\widetilde{Z}_1' \xrightarrow{\widetilde{\rho}_2:=\widetilde{\rho}_1\cup \widetilde{\rho}_1'} \widetilde{E}_1. 
\]
By the construction, $\widetilde{\rho}_2:\widetilde{Z}_2\to \widetilde{E}_1$ is a $\PP^1$-bundle which has a $(-m)$-section in $\widetilde{Z}_1'\subset \widetilde{Z}_2$, in particular it is identified with $\rho_m:\Hir_m\to \PP^1$. 
Since $\Critv(\widetilde{F}_2)$ is away from $\widetilde{Z}_1'$, it is also away from $S_m\subset \Hir_m$.

The restriction $\widetilde{F}_2|_{\Pa \widetilde{X}_2}:\Pa \widetilde{X}_2\to \Hir_m$ is a disjoint union of an $S^1$-bundle, and each component of it corresponds to a connected component of $\Pa \Sigma= \bigsqcup_k A_k$. 
Let $Y_k\subset \Pa \widetilde{X}_2$ be the component of $\Pa \widetilde{X}_2$ corresponding to $A_k$ (whose element is denoted by $(u,[\zeta])_k$). 
The isomorphism class of $Y_k$ as an $S^1$-bundle is determined from those of its restrictions over $S_m$ and a fiber of $\rho_m$. 
Since $Y_k\cap \widetilde{F}_2^{-1}(S_m)$ is contained in $\widetilde{X}_1'$, which is a trivial $\Sigma$-bundle, the restriction of $\widetilde{F}_2|_{Y_k}$ over $S_m$ is a trivial $S^1$-bundle. 
The restriction of $\widetilde{F}_2$ over a fiber of $\rho_m$ is a Lefschetz fibration with a monodromy representation $\theta$. 
Since $\theta(\alpha_\Pa)$ is equal to $t_{\delta_1}^{a_1}\cdots t_{\delta_b}^{a_b}$, the restriction of $\widetilde{F}_2:Y_k\to \Hir_m$ over a fiber of $\rho_m$ has the Euler number $a_k$.

We define an $(\sph^1)^2$-action to $(\sph^3)^2\times \disk_1$ as follows:
\[
(\lambda,\mu)\cdot (s_0,s_1,t_0,t_1,z) := (\lambda s_0,\lambda s_1,\mu t_0, \mu\lambda^{-m}t_1,\mu^{-a_k} z). 
\]
Let $\widetilde{X}_{2,k}' = ((\sph^3)^2\times \disk_1)/(\sph^1)^2$. 
It is easy to see that the projection $\widetilde{F}_{2,k}':\widetilde{X}_{2,k}'\to \Hir_m$ to the $(\sph^3)^2$-components is a $\disk_1$-bundle whose restrictions over $S_m$ and a fiber of $\rho_m$ are $0$ and $-a_k$, respectively. 
One can thus glue $\widetilde{X}_{2,k}'$ to $\widetilde{X}_2$ by a orientation-reversing, fiber-preserving diffeomorphism $\widetilde{\Phi}_{2,k}$, and eventually obtain a desired Lefschetz bifibration 
\[
\widetilde{X}_3=\widetilde{X}_2\cup_{\bigsqcup_k\widetilde{\Phi}_{2,k}} \left(\bigsqcup_k\widetilde{X}_{2,k}'\right) \xrightarrow{\widetilde{F}_3=\widetilde{F}_2\cup \left(\bigsqcup_k\widetilde{F}_{2,k}'\right)}\Hir_m \xrightarrow{\rho_m}\PP^1. 
\]
Note that the $0$-section $\mathcal{S}_i$ of the $\disk_1$-bundle $\widetilde{F}_{2,i}'$ is a section of $\widetilde{F}_3$ with the desired property. 
\end{proof}

\begin{remark}\label{rem:assumption (gb)neq(02)}

The assumption $(g,b)\neq (0,2)$ in \Cref{thm:combinatorial construction LbF over Hir_m} is needed only in the construction of the Lefschetz bifibration $\widetilde{F}_2:\widetilde{X}_2\to \widetilde{Z}_2$ (when we show that $\widetilde{F}_1:\Pa_{vh}\widetilde{X}_1\to \Pa \widetilde{Z}_1$ is a trivial $\Sigma$-bundle). 
In fact, this assumption can be removed by carefully considering the behavior of Lefschetz bifibrations at the horizontal boundaries in the proofs of \Cref{thm:combinatorial construction LbF over D2timesD2,thm:combinatorial construction LbF over Hir_m}. 
More concretely, we glue Lefschetz bifibrations using lifts obtained from \Cref{thm:condition lift a braid} in the proofs. 
It suffices to show that these lifts can be made trivial in a certain sense at the horizontal boundary. 
Once this is achieved, it is shown that the structure group of $\widetilde{F}_1:\Pa_{vh}\widetilde{X}_1\to \Pa \widetilde{Z}_1$ is contained in $\mathcal{D}_\Pa(\Sigma)$. 
Since a connected component of $\mathcal{D}_\Pa(\Sigma)$ is contractible even in the case $(g,b)=(0,2)$, it follows that $\widetilde{F}_1:\Pa_{vh}\widetilde{X}_1\to \Pa \widetilde{Z}_1$ is trivial, as shown in the original proof.

To carry out the above argument, it is necessary to refine the discussion in the proof of \Cref{thm:condition lift a braid}. 
Such a refinement is indeed possible; however, in the case of $(g,b)=(0,2)$, the resulting examples are not particularly interesting. 
Therefore, for the sake of overall conciseness, the detailed explanation required to cover this case is omitted.

\end{remark}

\begin{proof}[Proof of (1) of \Cref{thm:combinatorial construction pencil-fibration}]
We define an $\sph^1$-action to $\sph^3\times D^4$~($D^4\subset \C^2$ is the unit ball) by $\lambda\cdot (s_0,s_1,z_0,z_1)=(\lambda s_0,\lambda s_1,z_0,\lambda^{-m}z_1)$. 
Let $\widetilde{X}_{2,k_i}''=(\sph^3\times D^4)/\sph^1$ ($i=1,\ldots, l$). 
It is easy to see that $\widetilde{X}_{2,k_i}''$ is a $6$-manifold with boundary, and it can be obtained by blowing-down $\widetilde{X}_{2,k_i}'$ along fibers of $\mathcal{S}_{k_i}$. 
Indeed, the map $\Psi_{k_i}:\widetilde{X}_{2,k_i}'\to \widetilde{X}_{2,k_i}''$ defined by $\Psi_{k_i}([s_0,s_1,t_0,t_1,z]) = [s_0,s_1,t_0z,t_1z]$ is a blow-down map.
Let $\mathcal{S}_{k_i}' = [\sph^3\times \{0\}]\subset \widetilde{X}_{2,k_i}''$, which is the image of $\mathcal{S}_{k_i}$ by $\Psi_{k_i}$. 
We can define a map $\widetilde{F}_{2,k_i}'':\widetilde{X}_{2,k_i}''\setminus \mathcal{S}_{k_i}'\to \Hir_m$ by $\widetilde{F}_{2,k_i}''([s_0,s_1,z_0,z_1]) = [s_0,s_1,z_0/|z|,z_1/|z|]$, which makes the following diagram commute:
\[
\xymatrix{
\widetilde{X}_{2,k_i}'\setminus \mathcal{S}_{k_i} \ar[rd]_{\widetilde{F}_{2,k_i}'}\ar[r]^{\Psi_{k_i}} & \widetilde{X}_{2,k_i}''\setminus \mathcal{S}_{k_i}' \ar[d]^{\widetilde{F}_{2,i}''}\\
&\Hir_m.
}
\]  
One can thus obtain a Lefschetz pencil-fibration $\widetilde{F}_3':\widetilde{X}_3'\dasharrow \Hir_m$ over $\rho_m$ by gluing $\widetilde{X}_{2,k_i}''$ instead of $\widetilde{X}_{2,k_i}'$ in the construction of $\widetilde{F}_3$. 
\end{proof}

\begin{proof}[Proof of (2) of \Cref{thm:combinatorial construction pencil-fibration}]
The restriction $\widetilde{\rho}_1':\Pa \widetilde{Z}_1'\to \widetilde{E}_1$ is an $S^1$-bundle with Euler number $m=1$. 
Let $\widetilde{Z}_1^\dagger=D^4$ with the opposite orientation to the standard one. 
The boundary $\Pa \widetilde{Z}_1^\dagger\cong \sph^3$ also admits an $S^1$-bundle with Euler number $-1$, in particular it can be identified with $\Pa \widetilde{Z}_1'$ as an $S^1$-bundle. 
One can thus glue $\widetilde{Z}_1^\dagger$ with $\widetilde{Z}_1$ instead of $\widetilde{Z}_1'$ using the same gluing map.
We denote the resulting $4$-manifold by $\widetilde{Z}_2^\dagger$, which is diffeomorphic to $\PP^2$, in particular admits the projection $\widetilde{\rho}_2^\dagger:\widetilde{Z}_2^\dagger\setminus \{0\}\to \widetilde{E}_1$ (where $0\in \widetilde{Z}_1^\dagger\cong D^4$ is the origin).
Since $\widetilde{F}_1':\widetilde{X}_1'\to \widetilde{Z}_1'$ is a trivial $\Sigma$-bundle, one can further glue $\widetilde{X}_1^\dagger = \widetilde{Z}_1^\dagger\times \Sigma$ with $\widetilde{X}_1$ instead of $\widetilde{X}_1'$ by the same gluing map. 
We denote the resulting $6$-manifold and map to $\widetilde{Z}_2^\dagger$ by $\widetilde{X}_2^\dagger$ and $\widetilde{F}_2^\dagger:\widetilde{X}_2^\dagger\to \widetilde{Z}_2^\dagger$, respectively. 

The boundary $\Pa\widetilde{X}_2^\dagger$ is a disjoint union of an $S^1$-bundle over $\widetilde{Z}_2^\dagger$, and each component of it corresponds to a connected component of $\Pa \Sigma =\bigsqcup_k A_k$. 
Let $Y_k^\dagger\subset \Pa \widetilde{X}_2^\dagger$ be the component of $\Pa \widetilde{X}_2^\dagger$ corresponding to $A_k$. 
The isomorphism class of $Y_k^\dagger$ as an $S^1$-bundle is determined from that of its restriction over the closure of a fiber of $\widetilde{\rho}_2^\dagger$, whose Euler number is $a_k$. 
We define an $\sph^1$-action to $\sph^5\times \disk_1$ by $\lambda\cdot (s_0,s_1,s_2,z) = (\lambda s_0,\lambda s_1,\lambda s_2,\lambda^{-a_k}z)$. 
Let $\widetilde{X}_{2,k}^{\dagger\prime} = (\sph^5\times \disk_1)/\sph^1$ and $\widetilde{F}_{2,k}^{\dagger\prime}: \widetilde{X}_{2,k}^{\dagger\prime} \to \PP^2$ be the projection to the former components. 
One can easily check that $\widetilde{F}_{2,k}^{\dagger\prime}$ is a $\disk_1$-bundle over $\PP^2$ with Euler number $-a_k$, and thus one can glue $\widetilde{X}_{2,k}^{\dagger\prime}$ with $\widetilde{X}_2^\dagger$. 
The resulting $6$-manifold $\widetilde{X}_3^\dagger$ admits a Lefschetz pencil-fibration $\widetilde{F}_3^\dagger:\widetilde{X}_3^\dagger\to \widetilde{Z}_2^\dagger \cong \PP^2$ over $\widetilde{\rho}_2^\dagger :\widetilde{Z}_2^\dagger \dasharrow\widetilde{E}_1$.

For $i=1,\ldots, l$, we can obtain $\widetilde{X}_{2,k_i}^{\dagger\prime}$ by blowing-up $D^6\subset \C^3$ at the origin. 
Indeed, a map $\Psi_{k_i}^\dagger: \widetilde{X}_{2,k_i}^{\dagger\prime} \to D^6$ defined by $\Psi_{k_i}^\dagger([s_0,s_1,s_2,z]) = (s_0z,s_1z,s_2z)$ is a blow-down map (the $0$-section $\mathcal{S}_i^\dagger$ in $\widetilde{X}_2^{\dagger\prime}$ is the exceptional divisor). 
Let $\pi:D^6\dasharrow \PP^2$ be the natural projection. 
The following diagram then commutes: 
\[
\xymatrix{
\widetilde{X}_{2,k_i}^{\dagger\prime}\setminus \mathcal{S}_{k_i}^\dagger \ar[rd]_{\widetilde{F}_{2,k_i}^{\dagger\prime}}\ar[r]^{\Psi_{k_i}^\dagger} & D^6\setminus \{0\} \ar[d]^{\pi}\\
&\PP^2.
}
\]  
One can thus obtain a Lefschetz bipencil $\widetilde{F}_3^{\dagger\dagger}:\widetilde{X}_3^{\dagger\dagger} \dasharrow \widetilde{Z}_2^\dagger$ over $\widetilde{\rho}_2^\dagger$ by gluing $D^6$ instead of $\widetilde{X}_{2,k_i}^{\dagger\prime}$. 
This completes the proof of (2) of \Cref{thm:combinatorial construction pencil-fibration}.
\end{proof}

\section{Almost complex/symplectic structures on Lefschetz bifibrations}\label{sec:cpx/symp str LbF}

In this section, we show that the total space of a pencil-fibration structure obtained in \Cref{thm:combinatorial construction LbF over Hir_m} or \Cref{thm:combinatorial construction pencil-fibration} admits an almost complex structure compatible with the pencil-fibration in a suitable sense, and further admits a symplectic structure under some mild assumption.  
Note that we need an almost complex structure constructed here for not only obtaining a symplectic structure, but also showing that the Poincar\'{e} duals of the critical point set and the set of cusps are represented by the Thom polynomials of complex folds and cusps (\Cref{thm:crit pt set Chern class}).

Let $F:X\to \Hir_m$ be a Lefschetz bifibration. 
For $y\in \Hir_m$, we call the closure of a connected component of $F^{-1}(y)\setminus \Crit(F)$ an \textit{irreducible component} of $F$. 
A fiber $F^{-1}(y)$ itself is an irreducible component if $y$ is a regular value, or contained in $C_F$, and the number of irreducible components in a fiber $F^{-1}(y)$ is at most three, depending on configuration of critical points in $F^{-1}(y)$.   
Each irreducible component has a natural orientation (as a fiber of $F$) and represents a class in $H_2(X)$.

\begin{theorem}\label{thm:almost cpx str LbF}

Let $F:X\to \Hir_m$ be a Lefschetz bifibration constructed in the proof of \Cref{thm:combinatorial construction LbF over Hir_m} (from $\mu_i$'s, $\epsilon_i$'s and $\theta$ satisfying the conditions in \Cref{thm:combinatorial construction LbF over D2timesD2,thm:combinatorial construction LbF over Hir_m}), $\mathcal{S}_1,\ldots, \mathcal{S}_b$ be its sections, and $\widetilde{S}_m\subset X$ be a section of the (trivial) surface bundle $F|_{F^{-1}(S_m)}:F^{-1}(S_m)\to S_m$.
There exist almost complex structures $J, J'$, and $J''$ of $X,\Hir_m$, $\PP^1$, respectively, satisfying the following conditions.

\begin{itemize}

\item 
$\rho_m$ is $(J',J'')$-holomorphic and $\rho_m\circ F$ is $(J,J'')$-holomorphic. 

\item 
$F$ is $(J,J')$-holomorphic outside a (arbitrarily small) neighborhood of $\mathcal{I}_F^-$. 

\item 
The restriction of $F$ on a fiber of $\rho_m\circ F$ is $(J,J')$-holomorphic (even if the fiber contains a point in $\mathcal{I}_F^-$).

\item 
$S_m$ is a $J'$-holomorphic submanifold of $\Hir_m$ and $\widetilde{S}_m, \mathcal{S}_1,\ldots, \mathcal{S}_b$ and $\Crit(F)$ are $J$-holomorphic submanifolds of $X$. 


\end{itemize}

\noindent
Furthermore, if there exists $c\in H^2_{\mathrm{dR}}(X)$ with $c([\Sigma])>0$ for any irreducible component of $F$, $X$ admits a symplectic structure taming $J$. 

\end{theorem}

\noindent
Note that the last assumption (existence of $c$) in the theorem is necessary for existence of $\omega$ taming $J$ since $\Ker dF_x$ is $J$-complex, especially symplectic with respect to $\omega$, and thus any irreducible component of $F$ is a symplectic curve. 
The following proposition gives a sufficient condition for existence of $c$ concerning monodromies of $F$. 

\begin{proposition}\label{prop:existence cohomology for symp form}

The Lefschetz bifibration $F:X\to \Hir_m$ constructed in the proof of \Cref{thm:combinatorial construction LbF over Hir_m} satisfies the last assumption in \Cref{thm:almost cpx str LbF} (i.e.~existence of $c$) if there does not exist a path $\mu_i$ such that $\epsilon_i=\pm 2$ and the associated vanishing cycles $\beta_i^1$ and $\beta_i^2$ are isotopic in $\Sigma_g^b$. 

%
%
%

\end{proposition}

\noindent
As shown in \Cref{rem:example null-hom irre comp}, there might exist a null-homologous irreducible component (and thus a class $c$ in \Cref{thm:almost cpx str LbF} does not exist) if there exists a pair $(\mu_i,\mu_j)$ satisfying the conditions in the proposition. 

\begin{proof}[Proof of \Cref{thm:almost cpx str LbF}]
As explained in the beginning of \Cref{sec:combinatorial construction pencil-fibration}, we can assume $2-2g-b<0$ without loss of generality. 
Let $\nu S_m$, $\nu \widetilde{S}_m$ and $\nu \mathcal{S}_k$ be (sufficiently small) tubular neighborhood of $S_m$, $\widetilde{S}_m$ and $\mathcal{S}_k$, respectively. 
We can take a finite open cover $\mathcal{U}$ of $\PP^1$, a finite open cover of $\mathcal{V}_{U}$ of $\rho_m^{-1}(U)$ for each $U\in \mathcal{U}$, and a finite open cover $\mathcal{W}_{U,V}$ of $F^{-1}(V)$ for each $U\in \mathcal{U}$ and $V\in \mathcal{V}_U$ satisfying the following conditions: 

\begin{itemize}

\item 
Each $z\in \rho(\Delta_F)$ has a neighborhood such that only one $U\in \mathcal{U}$ intersects with it. 
For such a $U$, each $y\in \rho^{-1}(z)\cap \Delta_F$ has a neighborhood such that only one $V\in \mathcal{V}_U$ intersects with it. 
For such a $V$, each $x\in \Crit(F)\cap F^{-1}(y)$ has a neighborhood such that only one $W\in \mathcal{W}_{U,V}$ intersects with it. 

\item 
Assume that $U\in \mathcal{U}$ contains a point in $\rho_m(B_F\cup C_F)$ and take $V\in \mathcal{V}_U$ and $W\in \mathcal{W}_{U,V}$ so that $W$ contains the corresponding point in $\mathcal{B}_F\cup \mathcal{C}_F$. 
There exist complex charts $\varphi_{W}:W\to \C^3$, $\varphi'_{V}:V\to \C^2$, and $\varphi''_{U}:U\to \C$ making $F$ and $\rho_m$ the local model in \Cref{def:Lefschetz bifibration}. 

\item 
Assume that $U\in \mathcal{U}$ contains a point in $\rho_m(I_F)$ and take $V\in \mathcal{V}_U$ and $W_1,W_2\in \mathcal{W}_{U,V}$ so that $W_1$ and $W_2$ contain the corresponding points in $\mathcal{I}_F$.  
There exist complex charts $\varphi_{W_i}:W_i\to \C^3$, $\varphi'_V:V\to \C^2$, and $\varphi''_U:U\to \C$ such that $\varphi'_V\circ F\circ {\varphi_{W_1}}^{-1}(z_1,z_2,z_3)=(z_1,z_2^2+z_3^2)$, $\varphi'_V\circ F\circ {\varphi_{W_2}}^{-1}(z_1,z_2,z_3)$ is equal to either $(z_1,z_1+z_2^2+z_3^2)$ or $(z_1,\overline{z_1}+z_2^2+z_3^2)$ (depending on the sign of the intersection), and $\varphi''_U\circ \rho_m\circ {\varphi'_V}^{-1}(z_1,z_2)=z_1$.

\item 
For $U\in \mathcal{U}$, there exist exactly $(d-2)$ (resp.~$d$) $V$'s in $\mathcal{V}_U$ satisfying $V\cap \Critv(F)\neq \emptyset$ and $V\cap \Delta_F=\emptyset$ if $U\cap \rho(\Delta_F)\neq \emptyset$ (resp.~$U\cap \rho(\Delta_F)= \emptyset$). 
For each such $V$, there exists a unique $W\in \mathcal{W}_{U,V}$ satisfying $W\cap \Crit(F)\neq \emptyset$. 
Moreover, there exist complex charts $\varphi_{W}:W\to \C^3$, $\varphi'_V:V\to \C^2$, and $\varphi''_U:U\to \C$ such that $\varphi'_V\circ F\circ \varphi^{-1}_W(z_1,z_2,z_3)=(z_1,z_2^2+z_3^2)$ and $\varphi''_U\circ \rho\circ {\varphi'_V}^{-1}(z_1,z_2)=z_1$.


\item 
For each $U\in \mathcal{U}$, there exists a unique $V\in \mathcal{V}_U$ with $V\cap \nu S_m\neq \emptyset$. 
Furthermore, for such $U$ and $V$, there exists a unique $W\in \mathcal{W}_{U,V}$ with $W\cap \nu \widetilde{S}_m\neq \emptyset$.

\item 
For each $U\in \mathcal{U}$, $V\in \mathcal{V}_U$ and $k \in \{1,\ldots, b\}$, there exists a unique $W\in \mathcal{W}_{U,V}$ with $W\cap \nu \mathcal{S}_k\neq \emptyset$. 

\end{itemize}

Let $\mathcal{J}_{X}\subset \Gamma(\Aut(TX))$ be the set of almost complex structures of $X$ and $\mathcal{L}_{X}\subset \Gamma(\Aut(TX))$ be the set of sections of $\Aut(TX)$ whose value of any point in $X$ has no real eigenvalues.%
\footnote{$\mathcal{L}_X$ is denoted by $\mathcal{B}_X$ in \cite{GompfTowardTopcharasympmfd}. 
We avoid using this notation as it is confusing here.}
We take a retraction $j_{X}:\mathcal{L}_{X}\to \mathcal{J}_{X}$ as in \cite{GompfTowardTopcharasympmfd}. 
We also take $j_{\Hir_m}$ and $j_{\PP^1}$ in the same manner.
Let $\{\eta''_U\}_{U\in \mathcal{U}}$ be a partition of unity subordinate to the open cover $\mathcal{U}$. 
We put $J''_U = {\varphi''_U}^\ast J_{\std}$ for each $U\in \mathcal{U}$ and $A'':=\sum_{U\in \mathcal{U}}\eta''_UJ''_U$.
One can easily check that $A''$ is contained in $\mathcal{L}_{\PP^1}$. 
Let $J'' = j_{\PP^1}(A'')$. 
Since $j_{\PP^1}$ is a retraction, $J''$ is equal to ${\varphi''_U}^\ast J_{\std}$ on a neighborhood of each $z\in \rho(\Delta_F)$. 

For $U\in \mathcal{U}$ and $V\in \mathcal{V}_U$, we take an almost complex structure $J'_V$ of $V$ satisfying the following conditions: 

\begin{itemize}

\item 
$J'_V=\varphi_V'^\ast J_{\std}$ if $V\cap \Critv(F)\neq \emptyset$. 

\item 
$V\cap \nu S_m$ is $J_V'$-holomorphic (if it is not empty). 

\item 
$\rho_m$ is $(J'_V,J''_U)$-holomorphic. 

\end{itemize}

\noindent
For each $U\in \mathcal{U}$, we also take a partition of unity $\{\eta'_{U,V}\}_{V\in \mathcal{V}_U}$ on $\rho_m^{-1}(U)$ subordinate to the open cover $\mathcal{V}_U$.
Let $A'=\sum_{U\in \mathcal{U},V\in \mathcal{V}_U}\eta'_{U,V}\cdot(\eta_U''\circ \rho_m)\cdot J'_V$, which is contained in $\mathcal{L}_{\Hir_m}$.  
Indeed, suppose that $A'v=\lambda v$ for $v\in T_yZ$ and $\lambda\in \R$. 
Then, $\omega_{\mathrm{std}}(d(\rho_m)_y(v),d(\rho_m)_y(A'v))$ is equal to $\lambda\omega_{\mathrm{std}}(d(\rho_m)_y(v),d(\rho_m)_y(v))=0$.
On the other hand, the following holds:
\[
(0=)~\omega_{\mathrm{std}}(d(\rho_m)_y(v),d(\rho_m)_y(A'v))=\sum_{\rho_m(y)\in U\in \mathcal{U}}\eta''_U(\rho_m(y))\omega_{\mathrm{std}}(d(\rho_m)_y(v),J''_U(d(\rho_m)_y(v))).
\]
Since $J''_U$ is $\omega_{\mathrm{std}}$-tame, the right hand side of the equality above is $0$ only if $d(\rho_m)_y(v)=0$. 
As in the proof of Lemma 3.2 and Addendum 3.3 in \cite{GompfTowardTopcharasympmfd}, one can take a non-degenerate $2$-form $\zeta$ on $\Ker d\rho_y$ so that each $J'_{V}$ defined at $y$ is $\zeta$-tame on $\Ker d\rho_y$. 
Since $0= \zeta(v,A'v) = \sum_{U\in \mathcal{U},V\in \mathcal{V}_U}\eta'_{U,V}(y)\zeta(v,J'_{V}v)$, $v$ is equal to $0$. 

Let $J'=j_{\Hir_m}(A')$. 
By the construction, $J'$ is equal to $\varphi_V'^\ast J_{\std}$ on a neighborhood of each point in $\Delta_F$. 
Since $A''\circ d(\rho_m) = d(\rho_m)\circ A'$, one can also deduce from \cite[Corollary 4.2]{GompfTowardTopcharasympmfd} that $\rho_m$ is $(J',J'')$-holomorphic. 
Furthermore, since $A'$ preserves $TS_m$, $J'$ also preserves it by the following lemma.

\begin{lemma}\label{lem:Gompf almostcpxstr preserve subsp}

Let $Q$ be an even dimensional vector space, and $\mathcal{J},\mathcal{L}\subset \Aut(Q)$ and $j:\mathcal{L}\to \mathcal{J}$ be the subsets and the retraction defined in \cite{GompfTowardTopcharasympmfd}. 
If $B\in \mathcal{L}$ preserves a subspace $R\subset Q$, so does $j(B)$.

\end{lemma}

\begin{proof}[Proof of \Cref{lem:Gompf almostcpxstr preserve subsp}]
Let $\mathcal{J}_R,\mathcal{L}_R\subset \Aut(R)$ and $j_R:\mathcal{L}_R\to \mathcal{J}_R$ be the subsets and the retraction defined in the same way as $\mathcal{J}, \mathcal{L}$ and $j$. 
By the assumption, the restriction $B|_R$ is contained in $\mathcal{L}_R$. 
By the definition of $j$, one can easily check that $j(B)|_R$ is equal to $j_R(B|_{R})$, in particular its image is contained in $R$. 
\end{proof}

Let $c'\in H^2_{\mathrm{dR}}(\Hir_m)$ be the Poincar\'{e} dual of $[S_m]\in H_2(\Hir_m)$. 
For each $z\in \PP^1$, we take an area form $\xi_z$ so that $\int_{\rho_m^{-1}(z)}\xi_z$ is equal to $1$. (Note that $\rho_m^{-1}(z)$ is $J'$-holomorphic, especially has the natural orientation.)
For $U\in \mathcal{U}$, take $z\in U$ and a local trivialization $\Psi_U:\rho_m^{-1}(U) \to \rho_m^{-1}(z)\times U$ of $\rho_m$.
Let $\eta_U = (p_1\circ \Psi_U)^\ast \xi_z$. 
The $2$-form $\eta_U$ is closed and $[\eta_U]=c'|_U$ in $H^2_{\mathrm{dR}}(\rho_m^{-1}(U))$ since $[\rho_m^{-1}(z)]$ generates $H_2(\rho_m^{-1}(U))$ and $\left<[\eta_U],[\rho_m^{-1}(z)]\right> = 1 = \left<c',[\rho_m^{-1}(z)]\right>$. 
Moreover, since the restriction $d(p_1\circ \Psi_U)_{y}|_{\Ker d(\rho_m)_y}$ is an isomorphism for any $y\in \rho_m^{-1}(U)$, $J'$ is $\eta_U$-tame on $\Ker d(\rho_m)_y$. 
Applying \cite[Theorem 3.1]{GompfTowardTopcharasympmfd}, we obtain a closed $2$-form $\eta$ on $\Hir_m$ such that $[\eta]=c' \in H^2_{\mathrm{dR}}(\Hir_m)$, $J|_{\Ker d(\rho_m)}$ is $\eta$-tame, and $\omega'=t'\eta + \rho_m^\ast \omega_{\std}$ is a symplectic form taming $J'$ for a sufficiently small $t'>0$. 
Since $\mathcal{U}$ and $\mathcal{V}_U$ are finite covers and $\rho_m$ is $(J_V',J_U'')$-holomorphic for any $U\in \mathcal{U}$ and $V\in \mathcal{V}_U$, in the same way as that in the proof of \cite[Theorem 3.1]{GompfTowardTopcharasympmfd}, one can deduce that $\omega'$ also tames $J_V'$ for a sufficiently small $t'>0$. 

For $U\in \mathcal{U}$, $V\in \mathcal{V}_U$, and $W\in \mathcal{W}_{U,V}$, we take an almost complex structure $J_W$ on $W$ satisfying the following conditions: 

\begin{itemize}

\item 
$J_W=\varphi_W^\ast J_{\std}$ if $W\cap \Crit(F)\neq \emptyset$. 

\item 
$W \cap \mathcal{S}_k$ and $W\cap \widetilde{S}_m$ are $J_W$-holomorphic (if they are not empty).



\item 
$F$ is $(J_W,J'_V)$-holomorphic except for the case $W$ contains a point in $\mathcal{I}_F^-$ and $\varphi_V'\circ F \circ \varphi_W^{-1}(z_1,z_2,z_3)=(z_1,\overline{z_1}+z_2^2+z_3^2)$. 

\end{itemize}

\noindent
Since $\omega'$ tames $J_V'$, the almost complex structure $J_W$ is $(\omega',F)$-tame in the sense of \cite{GompfTowardTopcharasympmfd} when $F$ is $(J_W,J_V')$-holomorphic. 
Take $V\in \mathcal{V}_U$ and $W\in \mathcal{W}_{U,V}$ so that $W$contains a point in $\mathcal{I}_F^-$ and $\varphi_V'\circ F \circ \varphi_W^{-1}(z_1,z_2,z_3)=(z_1,\overline{z_1}+z_2^2+z_3^2)$. 
We identify $W$ (resp.~$V$) with $\C^3$ (resp.~$\C^2$) via $\varphi_W$ (resp.~$\varphi_V'$). 
Let $F_1(z_1,z_2,z_3)=(z_1,z_2^2+z_3^2)$ and $F_2(z_1,z_2,z_3)=(0,\overline{z_1})$. 
The following (in)equalities hold for $v\in TW$: 
{\allowdisplaybreaks
\begin{align*}
&\omega'(dF(v),dF(J_W v)) \\
=&\rho_m^\ast \omega_{\std}(d(F_1)(v),J_{\std}(d(F_1)(v)))+t'\eta(d(F_1)(v),J_{\std}(d(F_1)(v)))-t'\eta(d(F_2)(v),J_{\std}(d(F_2)(v)))\\
\geq & \rho_m^\ast \omega_{\std}(d(F_1)(v),J_{\std}(d(F_1)(v)))-t'\eta(d(F_2)(v),J_{\std}(d(F_2)(v))).
\end{align*}
}%
Moreover, one can check that there exist constants $C_1,C_2>0$ such that 
\begin{align*}
&\rho_m^\ast \omega_{\std}(d(F_1)(v),J_{\std}(d(F_1)(v)))\geq C_1 (a_1^2+b_1^2), \mbox{ and}\\
&\eta(d(F_2)(v),J_{\std}(d(F_2)(v))) \leq C_2(a_1^2+b_1^2)
\end{align*}%
for any $v = \sum_{i=1}^3 a_i\Pa_{x_i}+b_i\Pa_{y_i}\in TW$ (where $z_i = x_i+\sqrt{-1}y_i$).
Thus, $J_W$ is $(\omega',F)$-tame for a sufficiently small $t'>0$ even if $F$ is not $(J_W,J_V')$-holomorphic.

For each $U\in \mathcal{U}$, $V\in \mathcal{V}_U$, we take a partition of unity $\{\eta_{U,V,W}\}_{W\in \mathcal{W}_{U,V}}$ on $F^{-1}(V)$ subordinate to the open cover $\mathcal{W}_{U,V}$. 
Let $A=\sum_{U\in \mathcal{U},V\in \mathcal{V}_U,W\in \mathcal{W}_{U,V}}\eta_{U,V,W}(\eta'_{U,V}\circ F)(\eta_U''\circ \rho_m\circ F)J_W$. 
The following then holds for $v\in T_xX$:
\[
\omega'(dF_x(v),dF_x(Av))=\sum_{x\in W \in \mathcal{W}_{U,V}}\eta''_U(\rho_m(F(x)))\eta_{U,V}'(F(x))\eta_{U,V,W}(x)\omega'(dF_x(v),dF_x(J_W(v))).
\]
Since $J_W$ is $(\omega',F)$-tame, the right-hand side of the equality above is $0$ only if $v\in \Ker dF_x$.
Thus, one can show that $A$ is contained in $\mathcal{L}_X$ in the same way as before. 
Let $J=j_X(A)$. 
One can check the following conditions. 

\begin{itemize}

\item 
$J=\varphi_{W}^\ast J_{\std}$ in a neighborhood of a point in $\mathcal{B}_F\cup \mathcal{I}_F\cup \mathcal{C}_F$.

\item 
$F$ is $(J,J')$-holomorphic outside a neighborhood of $\mathcal{I}_F^-$. 

\item 
$\widetilde{S}_m,\mathcal{S}_1,\ldots, \mathcal{S}_b$ and $\Crit(F)$ are $J$-holomorphic. 

\end{itemize}

\noindent
Since $\rho_m\circ F|_W$ is $(J_W,J''_U)$-holomorphic and the restriction of $F|_W$ on a fiber of $\rho_m\circ F$ is $(J_W,J_V')$-holomorphic for any $U\in \mathcal{U},V\in \mathcal{V}_U$ and $W\in \mathcal{W}_{U,V}$ (even if $W\cap \mathcal{I}_F^-\neq \emptyset$), one can deduce from \cite[Corollary 4.2]{GompfTowardTopcharasympmfd} that $\rho_m\circ F$ is $(J,J'')$-holomorphic and the restriction of $F$ on a fiber of $\rho_m\circ F$ is $(J,J')$-holomorphic. 

Assume that $c\in H^2_{\mathrm{dR}}(X)$ satisfies the last assumption in \Cref{thm:almost cpx str LbF}. 
For any $y\in \Hir_m$, let $U_y$ be a neighborhood of $F^{-1}(y)\cap \Crit(F)$ contained in complex coordinate neighborhood(s) making $F$ the standard model. 
One can easily take a $2$-form $\xi_y$ on $F^{-1}(y)\cap U_y$ so that the restriction of $\xi_y$ on $\Sigma\setminus \Crit(F)$ is an area form with $\int_{\Sigma\setminus \Crit(F)} \xi_y = c([\Sigma])$ for any irreducible component $\Sigma$ in $F^{-1}(y)$, and $\xi_y$ coincides with $\omega_{\std}$ via the complex charts. 
One can further obtain a neighborhood $W_y$ of $F^{-1}(y)$ and a closed $2$-form $\eta_y$ on $W_y$ taming $J|_{\Ker dF}$ with $[\eta_y] = c|_{W_y}$ in the same way as the construction in the proof of \cite[Theorem 2.11]{GompfTowardTopcharasympmfd} (i.e.~taking a splicing map $\pi:W_y\to W_y$, and considering the pull-back of $\xi_y$ by $\pi$). 
Applying \cite[Theorem 3.1]{GompfTowardTopcharasympmfd}, we eventually obtain a symplectic form $\omega$ of $X$ taming $J$. 
\end{proof}

\begin{proof}[Proof of \Cref{prop:existence cohomology for symp form}]
Let $e'\in H^2_{\mathrm{dR}}(X\setminus \Crit(F))$ be the Euler class of the complex line bundle $\Ker dF$ on $X\setminus \Crit(F)$, $e \in H^2_{\mathrm{dR}}(X)$ be the image of $-e'$ by the isomorphism $H^2_{\mathrm{dR}}(X\setminus \Crit(F)) \xrightarrow{(i^\ast)^{-1}} H^2_{\mathrm{dR}}(X)$, where $i:X\setminus \Crit(F)\to X$ is the inclusion, and $s\in H^2_{\mathrm{dR}}(X)$ be the Poincar\'{e} dual of the union $\mathcal{S}_1\sqcup \cdots \sqcup \mathcal{S}_b$.  
In what follows, we show that $(e+s)([\Sigma])$ is positive for any irreducible component $\Sigma$ of $F$ (i.e.~$e+s$ satisfies the desired condition). 

If $\Sigma$ is a fiber of $F$, the value $(e+s)([\Sigma])$ is then equal to $2g-2+b$, which is positive by the assumption. 
In what follows, we assume that $\Sigma$ is not a fiber of $F$. 
In this case, $\Sigma$ does not contain cusps of $F$. 
Let $\widetilde{\Sigma}$ is a closed surface with a continuous immersion $\widetilde{\Sigma}\to \Sigma$ which is one-to-one except at its double points, $g(\Sigma)$ be the genus of $\widetilde{\Sigma}$, $f(\Sigma)$ be the number of folds of $F$ in $\Sigma$, and $d(\Sigma)$ be the number of double points in $\Sigma$.
It is easy to see that the value $e([\Sigma])$ is equal to $2g(\Sigma)-2+f(\Sigma)+d(\Sigma)$, and $s([\Sigma])$ is equal to the number of points in $\Sigma\cap (\mathcal{S}_1\sqcup \cdots \sqcup \mathcal{S}_b)$. 
If $F^{-1}(F(\Sigma))$ contains one fold of $F$, the corresponding vanishing cycle is separating, $d(\Sigma)=0$, and $f(\Sigma)=1$. 
Since the Lefschetz fibration on a $4$-dimensional fiber of $F$ is relatively minimal, $g(\Sigma)$ is positive and so is the value $(e+s)([\Sigma])$.
Suppose that $F^{-1}(F(\Sigma))$ contains two folds of $F$, that is, $F(\Sigma)$ is in $I_F$. 
Let $\mu_i$ be the corresponding vanishing path and $\beta_i^1, \beta_i^2$ be the pair of vanishing cycles associated with $\mu_i$. 
If either $\beta_i^1$ or $\beta_i^2$ is separating and the other one is non-separating, then $\Sigma$ is homologous to another irreducible component $\Sigma'$ such that $F(\Sigma')$ is not in $I_F$. 
Since $(s+e)([\Sigma'])$ is positive, $(s+e)([\Sigma])$ is also positive. 
If both $\beta_i^1$ and $\beta_i^2$ are separating, $d(\Sigma)=0$ and $f(\Sigma)$ is equal to $1$ or $2$. 
If $f(\Sigma)=1$, $g(\Sigma)$ is positive since the Lefschetz fibration on a $4$-dimensional fiber of $F$ is relatively minimal, and thus $(s+e)([\Sigma])$ is also positive. 
If $f(\Sigma)=2$, $e([\Sigma])\geq 0$ and $\Sigma\cap (\mathcal{S}_1\sqcup \cdots \sqcup \mathcal{S}_b)$ is not empty by the assumption in \Cref{prop:existence cohomology for symp form}. 
Thus, $s([\Sigma])$ is positive and so is $(s+e)([\Sigma])$. 
If both $\beta_i^1$ and $\beta_i^2$ are non-separating, $d(\Sigma)=0$ and $f(\Sigma)=2$. 
Thus, $(e+s)([\Sigma])$ is positive if $g(\Sigma)\geq 1$. 
If $g(\Sigma)=0$, $e([\Sigma])=0$ and $s([\Sigma])>0$ by the assumption, and thus $(e+s)([\Sigma])$ is positive. 
\end{proof}

\begin{proposition}\label{prop:symplectic structure blowdowns}

Let $F:X\to \Hir_m$ the same one as in \Cref{thm:almost cpx str LbF}. 
Suppose that there exists $c\in H^2_{\mathrm{dR}}(X)$ with $c([\Sigma])>0$ for any irreducible component of $F$, and thus $X$ admits a symplectic structure. 
Suppose further that $m$ and/or some $a_i$'s are equal to $1$, and let $F':X' \dasharrow \Hir_m$, $F^\dagger:X^\dagger \to \PP^2$, and $F^{\dagger\dagger}:X^{\dagger\dagger}\dasharrow \PP^2$ be the corresponding pencil-fibrations obtained from $F$ by the blowing-down procedures in the proof of \Cref{thm:combinatorial construction pencil-fibration} (if exist).  
The $6$-manifolds $X'$, $X^\dagger$ and $X^{\dagger\dagger}$ all admit symplectic structures. 

\end{proposition}

\begin{proof}[Proof of \Cref{prop:symplectic structure blowdowns}]
The manifold $X'$ can be obtained by blowing down $X$ along fibers of $\mathcal{S}_i$'s (as a $\PP^1$-bundle, for $i$ with $a_i=1$). 
Let $\mathcal{N}_i$ be the normal bundle of $\mathcal{S}_i$ in $X$. 
The Euler numbers of the restrictions of $\mathcal{N}_i$ on the $(-m)$-section $\mathcal{S}_i\cap F^{-1}(S_m)$ and a fiber $F$ of $\mathcal{S}_i$ are respectively equal to $0$ and $-a_i=-1$. 
Thus, the Chern class $c_1(\mathcal{N}_i)$ is equal to $-\PD([\mathcal{S}_i\cap F^{-1}(S_m)])$, whose self-intersection in $\mathcal{S}_i$ is equal to $-m$.
By \cite[Theorem 1.10]{LRZ2022sympblowdowndim6}, $X'$ admits a symplectic structure. 

The preimage $F^{-1}(S_m)$ is a surface bundle over $S_m$ having a section with self-intersection $0$. 
Thus, $F^{-1}(S_m)$ is diffeomorphic to $S_m\times \Sigma_g$, in particular it also admits an $S_m\cong \PP^1$-bundle structure. 
Furthermore, the restrictions of the normal bundle of $F^{-1}(S_m)$ in $X$ on the section above and a fiber as a $\Sigma_g$-bundle have the Euler number $-m=-1$ and $0$, respectively.  
Since $X^\dagger$ is obtained by blowing-down $X$ along fibers of $F^{-1}(S_m)$ as a $\PP^1$-bundle, one can obtain a symplectic structure of $X^\dagger$ by applying \cite[Theorem 1.10]{LRZ2022sympblowdowndim6}. 
Lastly, $X^{\dagger\dagger}$ can be obtained by blowing-down $X^\dagger$ along several copies of $\PP^2$, which admits a symplectic structure (as observed in the introduction of \cite{LRZ2022sympblowdowndim6}). 
\end{proof}

\section{Topological invariants of pencil-fibration structures from monodromies}\label{sec:top inv LbF}

In this section, we calculate topological invariants of the total space of a pencil-fibration structure from its monodromies. 
Throughout the section, let $F:X\to \Hir_m$ be a Lefschetz bifibration obtained in \Cref{thm:combinatorial construction LbF over Hir_m} (from $\theta$, $\mu_i$'s and $\epsilon_i$'s in the theorem). 

\begin{proposition}\label{prop:Euler LbF}

The Euler characteristic $\chi(X)$ is equal to $2(4-4g+d)-\ell$, where $\ell$ is the number of indices $i\in \{1,\ldots, d\}$ with $\epsilon_i=1$. 

\end{proposition}

\begin{proof}
The composition $\rho_m\circ F$ is a Lefschetz fibration whose regular fiber $W$ also admits a genus-$g$ Lefschetz fibration over $\PP^1$ with $d$ critical points. 
We can deduce from \cite{Kas1980handlebodyLF} that $X$ can be obtained by attaching $\ell$ $3$-handles to $D^2\times W$, and then closing the boundary by $D^2\times W$. 
We thus obtain $\chi(X) = 2\chi(W)-\ell = 2(4-4g+d)-\ell$.
\end{proof}

Let $T=F^{-1}(S_m)$. 
The restriction $F|_{T}:T\to S_m$ has a section $\widetilde{S}_m$ with self-intersection $0$ in $T$. 
Since $S_m$ is a section of $\rho_m$, $\widetilde{S}_m$ is also a section of the Lefschetz fibration $\rho_m\circ F$.
By this observation and the handle decomposition in the proof above, we obtain:

\begin{proposition}\label{prop:pi1 LbF}

The fundamental group of $X$ is isomorphic to that of a regular fiber $W$ of $\rho_m\circ F$. 

\end{proposition}

We can further prove the following proposition:

\begin{proposition}\label{prop:H2 LbF}

Let $L_1,\ldots,L_{\ell}\subset W$ be the vanishing cycles of the Lefschetz fibration $\rho_m\circ F$. 
There is a natural surjection $\left<[\widetilde{S}_m]\right>\oplus \left(H_2(W)/\left<[L_1],\ldots,[L_\ell]\right>\right) \twoheadrightarrow H_2(X;\Z)$. 
This surjection is isomorphism if $W$ is homeomorphic to a geometrically simply connected $4$-manifold (that is, a $4$-manifold admitting a handle decomposition without $1$-handles).

\end{proposition}

\begin{proof}[Proof of \Cref{prop:H2 LbF}]
Surjectivity follows from the aforementioned handle decomposition. 
Let $X_0$ be the union of $D^2\times W$ and $\ell$ $3$-handles in this handle decomposition. 
The manifold $X$ can be obtained by gluing $D^2\times W$ with $X_0$. 
Let $\Phi:\Pa D^2\times W\to \Pa X_0$ be the gluing map. 
By the assumption, there exists a geometrically simply connected $4$-manifold $W'$ and a homeomorphism $\psi:W'\to W$. 
Modifying $\psi$ by an isotopy if necessary, we can assume that $\psi$ preserves the $0$-handle. 
Let $\Phi'= \Phi\circ (\id_{\Pa D^2}\times \psi):\Pa D^2\times W'\to \Pa X_0$ and $X' = X_0\cup_{\Phi'} (D^2\times W')$. 
The map $\Psi=\id_{X_0}\cup (\id_{D^2}\times \psi):X' \to X$ is a homeomorphism.
Since $W'$ is geometrically simply connected, $D^2\times W'\subset X'$ can be regarded as a union of a $2$-cell and other cells with dimensions larger than $3$. 
Therefore, there exists an isomorphism $\left<[\widetilde{S}_m]\right>\oplus \left(H_2(W)/\left<[L_1],\ldots,[L_\ell]\right>\right)\to H_2(X';\Z)$ making the following diagram commute:
\[
\xymatrix{
\left<[\widetilde{S}_m]\right>\oplus \left(H_2(W)/\left<[L_1],\ldots,[L_\ell]\right>\right)\ar[d]^{\cong}\ar[r]^{\id} &\left<[\widetilde{S}_m]\right>\oplus \left(H_2(W)/\left<[L_1],\ldots,[L_\ell]\right>\right)\ar@{->>}[d]\\
H_2(X';\Z)\ar[r]^{\Psi_\ast}& H_2(X;\Z).
}
\]
Thus, the right vertical surjection in this diagram is an isomorphism. 
\end{proof}

In what follows, we identify the homology and cohomology groups of closed manifolds via the Poincar\'{e} duality. 
For closed manifolds $M^m, N^n$ and a continuous map $h:M\to N$, we put 
\[
h_! = \PD_N^{-1}\circ h_\ast \circ \PD_M :H^k(M;\Z)\to H^{k-m+n}(N;\Z),
\]
which is called the \textit{pushforward} (or the \textit{Umkehr homomorphism} in e.g.~\cite{Dyer1969cohomology}).
The following formula is used repeatedly in this manuscript: 
\begin{equation}\label{eq:projection formula}
h_!(\alpha)\cdot \beta = h_!(\alpha \cdot h^\ast \beta) \mbox{ for }\alpha \in H^k(M), \beta \in H^l (N). 
\end{equation}

By \Cref{thm:almost cpx str LbF}, we take almost complex structures $J,J',J''$ of $X,\Hir_m,\PP^1$, respectively, so that $F$ is $(J,J')$-holomorphic outside the disjoint union $K$ of closed balls each of which contains a point in $\mathcal{I}_F^-$. 
Using them, we can define the total Chern classes $c(X),c(\Hir_m),c(\PP^1)$, which are unit elements in the cohomology rings. 
The following proposition (together with \Cref{prop:H2 LbF}) uniquely characterize the first Chern class $c_1(X)$ up to torsions. 

\begin{proposition}\label{prop:value 1st Chern}

The Kronecker product $\left<c_1(X),[\widetilde{S}_m]\right>$ is equal to $2-m$. 
For $\alpha \in H_2(W)$, we denote by $[\alpha]$ the element in $H_2(X)$ corresponding to $\alpha$ via the surjection in \Cref{prop:H2 LbF}. 
The Kronecker product $\left<c_1(X),[\alpha]\right>$ is equal to $\left<c_1(W),\alpha\right>$. 

\end{proposition}

\begin{proof}
The Kronecker product $\left<c_1(X),[\widetilde{S}_m]\right>$ is equal to $\left<c_1(X)\cdot [\widetilde{S}_m],[X]\right>$. 
The normal bundle $\mathcal{N}_{\widetilde{S}_m}$ of ${\widetilde{S}_m}$ in $X$ admits a direct-sum decomposition of two complex bundles: one (the fiber direction of $F:T\to S_m$) is trivial and the other (the base direction of $F:X\to \Hir_m$) has degree $-m$. 
Hence, $c_1(X)\cdot [\widetilde{S}_m]$ is calculated as follows:
\begin{align*}
c_1(X)\cdot [\widetilde{S}_m]=&c_1(X)\cdot (i_{\widetilde{S}_m})_!(1) = (i_{\widetilde{S}_m})_!(i_{\widetilde{S}_m}^\ast(c_1(X))) \hspace{.5em}(\because \mbox{the formula }\eqref{eq:projection formula})\\
=&(i_{\widetilde{S}_m})_!(c_1({\widetilde{S}_m})+c_1(\mathcal{N}_{\widetilde{S}_m}))=(i_{\widetilde{S}_m})_!((2-m)[\ast])=(2-m)[\ast],
\end{align*}%
where $i_{\widetilde{S}_m}:\widetilde{S}_m\to X$ is the inclusion and $[\ast]\in H_0(X)\cong H^6(X)$ is the generator represented by a single point.
We can take a closed oriented surface $Q\subset W$ with $[Q]=\alpha$. 
Since the normal bundle $\mathcal{N}_W$ of $W$ in $X$ is trivial, one can show  $\left<c_1(X),[\alpha]\right>=\left<c_1(W),\alpha\right>$ in the same way as above. 
\end{proof}

We denote the $2i$-th degree part of $F^\ast c(\Hir_m)\cdot c(X)^{-1}$ by $s_i\in H^{2i}(X;\Z)$. 
We can explicitly describe the classes $s_1,s_2$ as follows:
{\allowdisplaybreaks
\begin{align*}
s_1 =& F^\ast c_1(\Hir_m) -c_1(X), \\
s_2 =& c_1(X)^2-c_2(X)-F^\ast c_1(\Hir_m)\cdot c_1(X)+ F^\ast c_2(\Hir_m). 
\end{align*}
}%
In the rest of this section, we identify $H_\ast(X;\Z)$ and $H^\ast(X;\Z)$ via Poincar\'{e} duality.

\begin{theorem}\label{thm:crit pt set Chern class}

Let $\mathcal{C}_F\subset \Crit(F)$ be the set of cusps of $F$. 
The classes $[\Crit(F)]$ and $[\mathcal{C}_F]$ are respectively equal to $s_1^2-s_2$ and $2(s_1^3-s_1s_2)$. 

\end{theorem}

\begin{proof}
Let $X'=X\setminus K$ and $F':X'\to \Hir_m$ be the restriction. 
One can define a section $dF':X'\to \Hom_{\C}(TX',F'^\ast T\Hir^2)$ by using $J$ and $J'$. 
Let $\Sigma_2=\{\sigma\in \Hom_\C(TX',F'^\ast T\PP^2)~|~ \dim_\C\Ker \sigma =2\}$.
In the next paragraph, we show that the section $dF$ is transverse to $\Sigma_2$. 
The equality $\PD([\Crit(F)])=s_1^2-s_2$ then follows from this claim, the result in \cite{Porteous1971SimpleSing}, and commutativity of the following diagram:
\[
\begin{CD}
H^0(\Crit(F);\Z) @>\iota_!>> H^4(X;\Z)\\
@V\cong VV @VV\cong V\\
H^0(\Crit(F');\Z) @>\iota'_!>> H^4(X';\Z),
\end{CD}
\] 
where the vertical isomorphisms are induced by the inclusions.

Around any cusp and its image, $J$ and $J'$ are pull-backs of the standard complex structures on $\C^3$ and $\C^2$ by complex charts making $F'$ the local model $(z_1,z_2,z_3)\mapsto (z_1,z_2^3+z_1z_2+z_3^2)$. 
In particular, $dF'$ is transverse to $\Sigma_2$ at a cusp. 
For $q_0\in \Crit(F')\setminus \mathcal{C}_{F'}$, we take a complex coordinate neighborhoods $(U,\varphi)$ and $(V,\psi)$ at $q_0$ and $F'(q_0)$, respectively, satisfying the following conditions:

\begin{enumerate}

\item 
$\varphi(q_0)=0$,

\item 
$\psi\circ F'\circ \varphi^{-1}(z_1,z_2,z_3)=(z_1,z_2z_3)$.

\end{enumerate}

\noindent
Let $(x_1,y_1,x_2,y_2,x_3,y_3),(X_1,Y_1,X_2,Y_2)$ be the real coordinates induced by 
$\varphi$ and $\psi)$, respectively. 
The critical point set $\Crit(F')$ is $J$-holomorphic, and thus $T_{q_0}\Crit(F') = \left<\Pa_{x_1},\Pa_{y_1}\right>$ is a $J$-complex subspace of $T_{q_0}X$.
We can take a sequence of regular points $\{q_i^j\}_{i\geq 1}$ ($j=2,3$) of $F'$ so that it converges to $q_0$ and $\Ker dF'_{q_i^j} = \left<\Pa_{x_j},\Pa_{y_j}\right>$. 
Since $\Ker dF'_q$ is $J$-complex for any $x\in X$, we can deduce that $\left<\Pa_{x_j},\Pa_{y_j}\right>$ is also $J$-complex at $q_0$. 
Hence, $\Pa_{x_1},\Pa_{x_2},\Pa_{x_3}$ is a basis of the $J$-complex vector space $T_{q_0}X$. 
Since $\Im dF'_{q_0}=\left<\Pa_{X_1},\Pa_{Y_1}\right>$ is $J'$-complex, $\Pa_{X_1},\Pa_{X_2}$ is a basis of the $J'$-complex vector space $T_{F'(q_0)}\Hir_m$.  
We take $\alpha_{ij},\beta_{ij}:V\to \R$ so that $(J'\Pa_{X_i})=\sum_{j=1}^2\alpha_{ij}(q')\Pa_{X_j}+\beta_{ij}(q')\Pa_{Y_j}$ at $q'\in V$. 
Let $\delta(y)=\det (\beta_{ij}(y))_{i,j}$, which is not equal to zero near $\psi(q_0)=0$ by the observation above. 
The following equality holds for $q\in U$ close to $q_0$ and $\{j,k\}=\{2,3\}$:
{\allowdisplaybreaks
\begin{align*}
dF'_{q}(\Pa_{x_j})=& x_k \Pa_{X_2}+y_k \Pa_{Y_2} \\
=&x_k \Pa_{X_2}+y_k \left(-\frac{\beta_{21}(F'(q))}{\delta(F'(q))}J'\Pa_{X_1}+\frac{\beta_{11}(F'(q))}{\delta(F'(q))}J'\Pa_{X_2} \right.\\
&\left.+\frac{\beta_{21}(F'(q))\alpha_{11}(F'(q))-\beta_{11}(F'(q))\alpha_{21}(F'(q))}{\delta(F'(q))}\Pa_{X_1}\right.\\
&\left.+\frac{\beta_{21}(F'(q))\alpha_{12}(F'(q))-\beta_{11}(F'(q))\alpha_{22}(F'(q))}{\delta(F'(q))}\Pa_{X_2}\right).
\end{align*}
}%
Thus, the representation matrix of $dF'_q$ with respect to the complex basis $\Pa_{x_1},\Pa_{x_2},\Pa_{x_3}$ and $\Pa_{X_1},\Pa_{X_2}$ is $\begin{pmatrix}
1 & C(q) \\
0 & D(q)
\end{pmatrix}$, where $C(q),D(q)$ are some $1\times 2$-matrices with complex entries, and $D_{1j}(q)$ ($j=1,2$) are explicitly described as follows:
{\allowdisplaybreaks
\begin{align*}
&D_{1j}(q) = x_k+y_k\frac{\beta_{21}(F'(q))\alpha_{12}(F'(q))-\beta_{11}(F'(q))\alpha_{22}(F'(q))}{\delta(F'(q))}+y_k\frac{\beta_{11}(F'(q))}{\delta(F'(q))}\sqrt{-1},
\end{align*}
}%
where $k=3$ and $2$ if $j=1$ and $2$, respectively. 
Since $\alpha_{21}(F'(q_0))=\beta_{21}(F'(q_0))=0$, $C(q_0)$ is the zero matrix. 
Thus, the section $dF'$ is transverse to $\Sigma_2$ at $q_0$ if and only if $D$ is a submersion at $q_0$. 
Regarding $D$ as a smooth map from $\C^3\cong \R^6$ to $\C^2\cong \R^4$ via $\varphi$, the Jacobi matrix of $D$ at the origin (which corresponds to $q_0$) is calculated as follows:
\[
\begin{pmatrix}
0&0&0&0&1&-\frac{\beta_{11}(F'(q_0))\alpha_{22}(F'(q_0))}{\delta(F'(q_0))}\\
0&0&0&0&0&\frac{\beta_{11}(F'(q_0))}{\delta(F'(q_0))}\\
0&0&1&-\frac{\beta_{11}(F'(q_0))\alpha_{22}(F'(q_0))}{\delta(F'(q_0))}&0&0\\
0&0&0&\frac{\beta_{11}(F'(q_0))}{\delta(F'(q_0))}&0&0
\end{pmatrix}
\]
The rank of this matrix is $4$ since $\beta_{11}(F'(q_0))\neq 0$. 
Hence, $dF'$ is transverse to $\Sigma_2$ at $q_0$. 

Since $\Ker dF_q$ is $J$-complex subspace of $T_qX$ for any $q\in X$ and $\dim_{\C}\Ker dF$ is constant (equal to $4$) on $\Crit(F)\subset X$, $\Ker dF$ and $\Coker dF$ are both complex vector bundles on $\Crit(F)$.
By the definition of the intrinsic derivative of $F$, it can be regarded as a section 
\[
d^2F:\Crit(F)\to\Hom(TX, \Hom_\C(\Ker dF,\Coker dF)).
\] 
For each $q\in X$, the restriction of $d^2F(q)\in \Hom(T_qX, \Hom_\C(\Ker dF_q,\Coker dF_q))$ on $\Ker dF_q$ is a symmetric bilinear form. 
Thus, we can further regard $d^2F$ as a section of the bundle $\Hom_{\C}(\Ker dF,(\Ker dF)^\ast\otimes_\C \Coker dF)$ on $\Crit(F)$. 
We define a submanifold $\Sigma_1'$ of this bundle as follows:
\[
\Sigma_1' = \{\sigma \in \Hom_{\C}(\Ker dF,(\Ker dF)^\ast\otimes_\C \Coker dF)~|~\dim_\C \Ker \sigma = 1\}.
\] 
By calculating the intrinsic derivative using the local coordinate description of $F$ around each point in $\Crit(F)$, one can show that the preimage of $\Sigma_1'$ by $d^2F$ is equal to $\mathcal{C}_{F}$ and $d^2F$ is transverse to $\Sigma_1'$ at any point in $\mathcal{C}_{F}$. 
Therefore, we can represent the class $\PD([\mathcal{C}_{F}])$ by the Thom polynomial for the cusp singularity, which is equal to $2(s_1^3-s_1s_2)$ (see e.g.~\cite{FK2006Thompolysecondorder}). 
\end{proof}

\begin{corollary}\label{cor:explicit description cusp [crit]}

Let $\Sigma,W\subset X$ be $2$- and $4$-dimensional fibers of $F$, respectively, $\widetilde{S}_m\subset T=F^{-1}(S_m)$ be a section of $F|_{T}$ with self-intersection $0$ in $T$, and $g$ be the genus of $F$. 

\begin{itemize}

\item 
$[\Crit(F)]\in H_2(X;\Z)\cong H^4(X;\Z)$ is equal to the following class:
\begin{equation}\label{eqn:class crit set}
2m[\Sigma]-(2+m)(i_W)_!(c_1(W))+4(g-1)[\widetilde{S}_m]+c_2(X),
\end{equation}
where $i_W:W\to X$ is the inclusion.

\item 
The number of cusps of $F$ is equal to 
\begin{equation}\label{eqn:number cusps}
24\left(m(g-1)+(2+m)\frac{\sigma(W)+\chi(W)}{4}-\frac{c_1c_2(X)}{12}\right),
\end{equation}
where $c_1c_2(X)$ is the Chern number.
In particular, the number of cusps of $F$ is divisible by $24$.

\end{itemize}

\end{corollary}
%
%
%

\begin{proof}
We put $\alpha_i = F^\ast c_i(\Hir_m)$ and $\beta_i=c_i(X)$. 
As $\dim_{\R} \Hir_m=4$, $\alpha_1^3$ and $\alpha_1\alpha_2$ are equal to $0$. 
Let $L\subset \Hir_m$ be a fiber of $\rho_m$, which is $J'$-holomorphic by Theorem~\ref{thm:almost cpx str LbF}. 
We can deduce from the adjunction equality that $c_1(\Hir_m)$ is equal to $(2+m)[L]+2[S_m]\in H_2(\Hir_m)$. 
We thus obtain $\alpha_1 = (2+m)[W] + 2[T]$. 
It is easy to check that $[T]^2 = -m[\Sigma]$, $[W]\cdot [T]=[\Sigma]$ and $[W]^2=0$.
Thus the following equality holds: 
\[
\alpha_1^2 = \left((2+m)[W] + 2[T]\right)^2 = 4(2+m)[\Sigma]-4m[\Sigma]=8[\Sigma]. 
\]
The Euler characteristic of $\Hir_m$ is $4$, in particular we obtain $\alpha_2 =4[\Sigma]$. 
Since $[A]=(i_A)_!(1)$ for a submanifold $A\subset X$, the following equality holds: 
\[
[W]\cdot\beta_1 =(i_{W})_!(1)\cdot \beta_1 = (i_W)_!(i_W^\ast c_1(X))= (i_W)_!(c_1(W)+c_1(\mathcal{N}_W)), 
\]
where $\mathcal{N}_W$ is the normal bundle of $W$ in $X$, which is trivial since $W$ is a regular fiber of $\rho_m\circ F$.
We thus obtain $[W]\cdot\beta_1 = (i_W)_!(c_1(W))$.
We can also calculate $[T]\cdot \beta_1$ as follows: 
{\allowdisplaybreaks
\begin{align*}
[T]\cdot\beta_1 &= (i_T)_!(c_1(T)+c_1(\mathcal{N}_T))\\
&= (i_T)_!\left((2-2g)[\widetilde{S}_m]+2[\Sigma]-m[\Sigma]\right)\\
& = (2-2g)[\widetilde{S}_m]+(2-m)[\Sigma],
\end{align*}
}%
where the second equality follows from the adjunction equality. 
We thus obtain:
{\allowdisplaybreaks
\begin{align*}
[\Crit(F)] =&s_1^2-s_2\\
=&(\alpha_1-\beta_1)^2 - (\beta_1^2-\beta_2-\alpha_1\beta_1+ \alpha_2)\\
=&\alpha_1^2-\alpha_1\beta_1+\beta_2-\alpha_2\\
=&4[\Sigma]-\left((2+m)[W]+2[T]\right)\beta_1+\beta_2\\
=&4[\Sigma]-(2+m)(i_W)_!(c_1(W))-2(2-2g)[\widetilde{S}_m]-2(2-m)[\Sigma]+\beta_2\\
=&2m[\Sigma]-(2+m)(i_W)_!(c_1(W))+4(g-1)[\widetilde{S}_m]+c_2(X).
\end{align*}
}%

In the same way as above, one can show the following equalities: 
{\allowdisplaybreaks
\begin{align*}
[\Sigma]\cdot\beta_1 =&(i_\Sigma)_!(c_1(\Sigma)+c_1(\mathcal{N}_\Sigma)) = (2-2g)[\ast],\\
[W]\cdot\beta_1^2 =& (i_W)_!((c_1(W)+c_1(\mathcal{N}_W))^2)=(3\sigma(W)+2\chi(W))[\ast], \\
[W]\cdot\beta_2 =&(i_W)_!(c_2(W)+c_1(W)c_1(\mathcal{N}_W))=\chi(W)[\ast],\\
[T]\cdot\beta_1^2 =& (i_T)_!((c_1(T)+c_1(\mathcal{N}_T))^2)=(i_T)_!((c_1(T)-m[\Sigma])^2)\\
=&(i_T)_!(c_1(T)^2-2m[\Sigma]c_1(T))=(3\sigma(T)+2\chi(T))[\ast]-2m(i_T)_!([\Sigma]c_1(T))\\
=&2(4-4g)[\ast]-2m(2-2g)[\ast]=(2-m)(4-4g)[\ast],\\
[T]\cdot\beta_2 =& (i_T)_!(c_2(T)+c_1(T)c_1(\mathcal{N}_T))=(i_T)_!(c_2(T)-m[\Sigma]c_1(T))\\
=&\chi(T)[\ast]-m(i_T)_!([\Sigma]c_1(T))=(4-4g)[\ast]-m(2-2g)[\ast]\\
=&(2-m)(2-2g)[\ast].
\end{align*}
}%
By \Cref{thm:crit pt set Chern class}, $[\mathcal{C}_F]$ is calculated as follows:
{\allowdisplaybreaks
\begin{align*}
[\mathcal{C}_F]=&2(s_1^3-s_1s_2) \\
=&2\left(\left(\alpha_1-\beta_1\right)^3-\left(\alpha_1-\beta_1\right)\left(\beta_1^2-\beta_2-\alpha_1\beta_1+ \alpha_2\right)\right)\\
=&2\left(-3\alpha_1^2\beta_1+3\alpha_1\beta_1^2-\beta_1^3-\left(2\alpha_1\beta_1^2-\alpha_1\beta_2-\alpha_1^2\beta_1-\beta_1^3+\beta_1\beta_2-\beta_1\alpha_2\right)\right)\\
=&2\left(-2\alpha_1^2\beta_1+\alpha_1\beta_1^2+\alpha_1\beta_2-\beta_1\beta_2+\beta_1\alpha_2\right)\\
=&2\left(-12[\Sigma]\beta_1+\alpha_1(\beta_1^2+\beta_2)-\beta_1\beta_2\right)\\
=&2\left(24(g-1)[\ast]+\left((2+m)[W] + 2[T]\right)(\beta_1^2+\beta_2)-\beta_1\beta_2\right)\\
=&2\left(24(g-1)[\ast]+3(2+m)(\sigma(W)+\chi(W))[\ast]-12(2-m)(g-1)[\ast]-\beta_1\beta_2\right)\\
=&24\left(m(g-1)[\ast]+(2+m)\frac{\sigma(W)+\chi(W)}{4}[\ast]-\frac{c_1(X)c_2(X)}{12}\right).
\end{align*}
}%
Hence the number of cusps of $F$ is equal to \eqref{eqn:number cusps}. 
Since $W$ has an almost complex structure, $\sigma(W)+\chi(W)$ is divisible by $4$ (\cite[\S1.4.]{GS19994-mfd}).
It follows from the $\hat{A}$-integrality theorem~(\cite[Theorem~26.1.1]{Hirzebruch1995topmethod}) that $c_1c_2(X)$ is divisible by $24$, and thus \eqref{eqn:number cusps} is also divisible by $24$.
\end{proof}

Lastly, the following propositions clarify how various blow-down procedures affect topological invariants. 

\begin{proposition}\label{prop:blowdown topinv 1}

Suppose that $a_{i_1}=\cdots a_{i_k}=1$, and let $F':X'\dasharrow \Hir_m$ be the Lefschetz pencil-fibration obtained by blowing down $F:X\to \Hir_m$ along ruled surfaces $\mathcal{S}_{i_1},\ldots, \mathcal{S}_{i_k}$. 
Let $\ell$ be the number of indices $i\in {1,\ldots,d}$ with $\epsilon_i=1$, $W'\subset X'$ be a regular fiber of $\rho_m\circ F'$, $\Sigma'\subset X'$ be the closure of a regular fiber of $F'$, $\widetilde{S}_m'$ be a section with self-intersection $0$, and $\mathcal{S}_{i_1}',\ldots, \mathcal{S}_{i_k}'\subset X'$ be the image of $\mathcal{S}_{i_1},\ldots, \mathcal{S}_{i_k}$ by the blowdown map. 

\begin{enumerate}

\item 
$\chi(X') = 2(4-4g+d)-\ell-2k$, 

\item 
$\pi_1(X') \cong \pi_1(W')$, 

\item 
There is a natural surjection $\left<[\widetilde{S}_m]\right>\oplus \left(H_2(W')/ \left<[L_1],\ldots, [L_\ell]\right>\right)\twoheadrightarrow H_2(X;\Z)$. 
This surjection is isomorphism if $W'$ is homeomorphic to a geometrically simply connected $4$-manifold. 

\item 
$[\Crit(F')]=2m[\Sigma']-(2+m)(i_{W'})_!(c_1(W'))+4(g-1)[\widetilde{S}_m']+c_2(X')+\sum_{a=1}^{k}[\mathcal{S}_{i_a}']$,

\item 
The number of cusps is equal to \eqref{eqn:number cusps} with $W$ and $X$ replaced with $W'$ and $X'$, respectively. 

\end{enumerate}

\end{proposition}

\begin{proof}
The manifold $X'$ can be obtained from $X$ by removing tubular neighborhoods of $\mathcal{S}_{i_1},\ldots, \mathcal{S}_{i_k}$ and then gluing those of $\mathcal{S}_{i_1}',\ldots, \mathcal{S}_{i_k}'$.
In particular, the Euler characteristic decreases by $k(\chi(\mathcal{S}_i)-\chi(\mathcal{S}_i'))=2k$. 
Thus, (1) follows from \Cref{prop:Euler LbF}.
(2) follows from the observation above and the Seifert-Van Kampen theorem. (Note that $W'$ is the blowdown of $X$ along the $k$ spheres $\mathcal{S}_{i_1}\cap W,\ldots, \mathcal{S}_{i_k}\cap W$.)
One can show (3) in the same way as the proof of \Cref{prop:H2 LbF} using the Lefschetz fibration $\rho_m\circ F'$ (admitting a section). 
Let $\pi:X\to X'$ be the blowdown mapping. 
Applying the formula in \cite[\S8]{GP2007Chernblowup} repeatedly, we obtain:
\[
c_2(X) = \pi^\ast \left(c_2(X')+\sum_{a=1}^{k}[\mathcal{S}_{i_a}']\right)-\sum_{a=1}^{k}\pi^\ast c_1(X')\cdot [\mathcal{S}_{i_a}].
\]
Hence, we can calculate $\pi_!c_2(X)$ as follows:
{\allowdisplaybreaks
\begin{align*}
&\pi_!c_2(X)\\
=&\pi_!\pi^\ast \left(c_2(X')+\sum_{a=1}^{k}[\mathcal{S}_{i_a}']\right)-\pi_!\left(\sum_{a=1}^{k}\pi^\ast c_1(X')\cdot [\mathcal{S}_{i_a}]\right)\\
=&c_2(X')+\sum_{a=1}^{k}[\mathcal{S}_{i_a}']-\sum_{a=1}^{k}c_1(X')\cdot \pi_![\mathcal{S}_{i_a}]\\
=&c_2(X')+\sum_{a=1}^{k}[\mathcal{S}_{i_a}']-\sum_{a=1}^{k}c_1(X')\cdot \pi_!(i_{\mathcal{S}_{i_a}})_!(1)\\
=&c_2(X')+\sum_{a=1}^{k}[\mathcal{S}_{i_a}']. & (\because \pi_!(i_{\mathcal{S}_{i_a}})_!=(\pi\circ i_{\mathcal{S}_{i_a}})_!=0)
\end{align*}
}%
By \Cref{cor:explicit description cusp [crit]}, we obtain (4) as follows:
{\allowdisplaybreaks
\begin{align*}
[\Crit(F')] =&\pi_! [\Crit(F)]\\
=&\pi_! \left(2m[\Sigma]-(2+m)(i_W)_!(c_1(W))+4(g-1)[\widetilde{S}_m]+c_2(X)\right)\\
=&2m[\Sigma']-(2+m)(\pi|_W)_!\left(\pi^\ast c_1(W')-\sum_{a=1}^{k}[\mathcal{S}_{i_a}\cap W]\right)\\
&+4(g-1)[\widetilde{S}_m']+c_2(X')+\sum_{a=1}^{k}[\mathcal{S}_{i_a}']\\
=&2m[\Sigma']-(2+m)(i_{W'})_!(c_1(W'))+4(g-1)[\widetilde{S}_m']+c_2(X')+\sum_{a=1}^{k}[\mathcal{S}_{i_a}'].
\end{align*}
}%
Since $W'$ can be obtained by blowing down $W$, $\sigma(W)+\chi(W)$ is equal to $\sigma(W')+\chi(W')$. 
Using the formulae in \cite[\S8]{GP2007Chernblowup}, one can show that the Chern number $c_1c_2$ of a $6$-manifold is invariant under blowdowns.
Hence we obtain (5). 
\end{proof}

\begin{proposition}\label{prop:blowdown topinv 2}

Suppose that $m=1$, and let $F^\dagger:X^\dagger\to \PP^2$ be the Lefschetz pencil-fibration obtained by blowing down $X$ along the ruled surface $F^{-1}(S_1)$. 
Let $\ell$ be the number of indices $i\in {1,\ldots,d}$ with $\epsilon_i=1$, $W^\dagger\subset X^\dagger$ be the closure of a regular fiber of $\rho\circ F^\dagger$, and $\Sigma^\dagger\subset X^\dagger$ be a regular fiber of $F^\dagger$.

\begin{enumerate}

\item 
$\chi(X^\dagger) = 2(3-3g+d)-\ell$,

\item 
$\pi_1(X^\dagger) \cong \pi_1(W^\dagger)$, 

\item 
There is a natural surjection $H_2(W^\dagger)/ \left<[L_1],\ldots, [L_\ell]\right>\twoheadrightarrow H_2(X^\dagger)$. 
This surjection is isomorphism if $W^\dagger$ is homeomorphic to a geometrically simply connected $4$-manifold.

\item 
$[\Crit(F^\dagger)]=3[\Sigma^\dagger]-3(i_{W^\dagger})_!(c_1(W^\dagger))+c_2(X^\dagger)$,

\item 
The number of cusps is equal to \eqref{eqn:number cusps} with $W$ and $X$ replaced with $W^\dagger$ and $X^\dagger$, respectively. 

\end{enumerate}

\end{proposition}

\begin{proof}
The manifold $X^\dagger$ can be obtained from $X$ by removing tubular neighborhood of $F^{-1}(S_1)\cong S_1\times \Sigma_g$ and then gluing that of $F^{-1}([0:0:1])\cong \Sigma_g$.
One can deduce (1) and (2) from this observation in the same way as the proof of the preceding proposition. 
(3) follows from the Meyer-Vietoris exact sequences for the decompositions of $X$ and $X^\dagger$ above. 
Let $\pi^\dagger:X\to X^\dagger$ be the blowdown mapping. 
Applying the formula in \cite[\S8]{GP2007Chernblowup}, we obtain:
\[
c_2(X) = (\pi^\dagger)^\ast \left(c_2(X^\dagger)+[\Sigma^\dagger]\right)-(\pi^\dagger)^\ast c_1(X^\dagger)\cdot [F^{-1}(S_1)].
\]
We can obtain the following equality in the same way as the proof of \Cref{prop:blowdown topinv 1}:
\[
(\pi^\dagger)_!c_2(X)=c_2(X^\dagger)+[\Sigma^\dagger].
\]
By \Cref{cor:explicit description cusp [crit]}, we obtain (4) as follows:
{\allowdisplaybreaks
\begin{align*}
[\Crit(F^\dagger)] =&(\pi^\dagger)_! [\Crit(F)]\\
=&(\pi^\dagger)_! \left(2\cdot 1[\Sigma]-(2+1)(i_W)_!(c_1(W))+4(g-1)[\widetilde{S}_1]+c_2(X)\right)\\
=&2[\Sigma^\dagger]-3(\pi^\dagger|_W)_!\left((\pi^\dagger)^\ast c_1(W^\dagger)\right)+c_2(X^\dagger)+[\Sigma^\dagger]\\
=&3[\Sigma^\dagger]-3(i_{W^\dagger})_!(c_1(W^\dagger))+c_2(X^\dagger).
\end{align*}
}%
Lastly, (5) follows from the facts that $W^\dagger$ is diffeomorphic to $W$ and $c_1c_2$ of a $6$-manifold is invariant under blowdowns. 
\end{proof}

\begin{proposition}\label{prop:blowdown topinv 3}

Suppose that $m=1$ and $a_{i_1}=\cdots = a_{i_k=1}$.
Let $F^{\dagger\dagger}:X^{\dagger\dagger}\dashrightarrow \PP^2$ be the Lefschetz bipencil obtained by blowing down $X^\dagger$ along $F^\dagger(\mathcal{S}_{i_1}),\ldots, F^\dagger(\mathcal{S}_{i_1})$.  
Let $\ell$ be the number of indices $i\in {1,\ldots,d}$ with $\epsilon_i=1$, $W^{\dagger\dagger}\subset X^{\dagger\dagger}$ be the closure of a regular fiber of $\rho\circ F^{\dagger\dagger}$, and $\Sigma^{\dagger\dagger}\subset X^{\dagger\dagger}$ be a regular fiber of $F^{\dagger\dagger}$.

\begin{enumerate}

\item 
$\chi(X^{\dagger\dagger}) = 2(3-3g+d)-\ell-2k$, 

\item 
$\pi_1(X^{\dagger\dagger}) \cong \pi_1(W^{\dagger\dagger})$, 

\item 
There is a natural surjection $H_2(W^{\dagger\dagger})/ \left<[L_1],\ldots, [L_\ell]\right>\twoheadrightarrow H_2(X^{\dagger\dagger})$. 
This surjection is isomorphism if $W^{\dagger}$ is homeomorphic to a geometrically simply connected $4$-manifold. 

\item 
$[\Crit(F^{\dagger\dagger})]=3[\Sigma^{\dagger\dagger}]-3(i_{W^{\dagger\dagger}})_!(c_1(W^{\dagger\dagger}))+c_2(X^{\dagger\dagger})$,

\item 
The number of cusps is equal to \eqref{eqn:number cusps} with $W$ and $X$ replaced with $W^{\dagger\dagger}$ and $X^{\dagger\dagger}$, respectively. 

\end{enumerate}

\end{proposition}

\noindent
One can show this proposition in the same way as before, using the formula $c_2(X^{\dagger\dagger})=(\pi^{\dagger\dagger})^\ast c_2(X^\dagger)$ (see \cite[\S8]{GP2007Chernblowup}). 
The details are left to the reader.


\begin{remark}\label{rem:other topinv from monodromies}

For an almost complex $6$-manifold $X$, we can consider the three (non-trivial) Chern numbers $c_3(X)$, $c_1c_2(X)$, and $c_1^3(X)$. 
Suppose that $X$ is the total space of a pencil-fibration structure. 
One can easily calculate $c_3(X) = \chi(X)$ from monodromies (cf.~\Cref{prop:Euler LbF}). 
Since the $4$-dimensional fiber $W$ admits a Lefschetz fibration, one can calculate $\chi(W)$ and $\sigma(W)$ from monodromies in theory (using Meyer's signature cocycle \cite{Meyer1973sigcocycle}, for example).
Thus we can obtain $c_1c_2(X)$ from monodromies of the pencil-fibration structure (\Cref{cor:explicit description cusp [crit]}). 
The author does not know how to calculate the last Chern number from monodromy invariants:

\begin{problem}

Calculate the Chern number $c_1^3(X)$ (or more generally, determine the triple cup product structure of $H^2(X)$) of the total space $X$ of a pencil-fibration structure from its monodromies. 

\end{problem} 

\end{remark}

\section{An example}\label{sec:example}

In this section, we give an example of a compatible pair of braid and fiber monodromies. 
It follows from \Cref{thm:combinatorial construction pencil-fibration} that a Lefschetz bipencil can be obtained from this pair. 
We also carry out the computation of topological invariants of the total space of the bipencil from its monodromies. 
Note that this Lefschetz bipencil is naturally conjectured to be isomorphic to one obtained by composing the degree-$2$ Veronese embedding $\PP^3 \hookrightarrow \PP^9$ and a generic projection $\PP^9 \dashrightarrow \PP^2$ (cf.~\Cref{rem:example is deg2 Veronese}).

Let $R\subset D^2$ be the set of twelve points described in \Cref{fig:vanishing path 1}. 
We take simple paths $\beta_0,\beta_1, \mu_i\subset D^2$ between two points in $R$ as shown in \Cref{fig:vanishing path 1}, and put
\[
\mu_{i\pm 2} = \tau_{\beta_0}^{\mp 1}(\mu_i)~(i=4,14),\hspace{.3em}\mu_{j\pm 2} = \tau_{\beta_1}^{\mp 1}(\mu_j)~(i=5,15),\hspace{.3em}\mu_{11} = \tau_{\beta_0}^{-1}\tau_{\beta_1}^{-1}(\mu_{10}).
\]
We put $\epsilon_i =3$ for $i=2,\ldots, 7,13,\ldots, 18$, $\epsilon_i=1$ for $i=10,11$, and $\epsilon_i=2$ for other $i\in \{1,\ldots, 28\}$. 
Let $\kappa:D^2\to D^2$ be the $\pi/2$-degree counterclockwise rotation, which preserves $R$ and fixes the center $r_0\in D^2$ described in \Cref{fig:vp_1_2}.
We put $\mu_{28+i} = \kappa (\mu_i)$ and $\epsilon_{28+i}=\epsilon_i$ for $i=1,\ldots, 28$. 
We take simple closed curves $c_1,\ldots, c_6\subset \Sigma_1^8$ as shown in \Cref{fig:vanishing cycle 1}. 
Let $\iota:\Sigma_1^8\to \Sigma_1^8$ be the hyperelliptic involution (i.e.~the $\pi$-degree rotation) along the axis described in \Cref{fig:vc_1_1}, which preserves $\Pa \Sigma_1^8$. 
We put $c_{6+i} = \iota(c_i)$ for $i=7,\ldots, 12$. 
We take simple paths $\gamma_1,\ldots, \gamma_{12}\subset D^2$ from $r_0$ to points in $R$ as shown in \Cref{fig:vp_1_8}, and define a homomorphism $\theta:\pi_1(D^2\setminus R,r_0)\to \mathcal{M}_\Pa(\Sigma_1^8)$ so that $\theta(\alpha_i)=t_{c_i}$ for the meridian loop $\alpha_i$ of $\gamma_i$ for $i=1,\ldots, 12$. 

\begin{figure}[htbp]
\subfigure[$\mu_1,\mu_4,\mu_5$ and $\beta_0,\beta_1$.]{\includegraphics{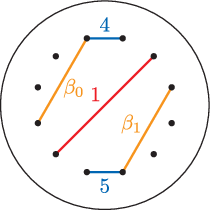}\label{fig:vp_1_1}}
\subfigure[$\mu_8,\mu_9$ and the center $r_0$ of $D^2$.]{\includegraphics{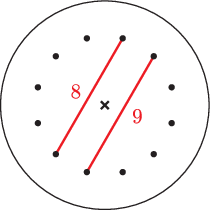}\label{fig:vp_1_2}}
\subfigure[$\mu_{10}$.]{\includegraphics{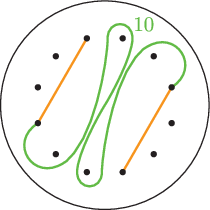}\label{fig:vp_1_3}}
\subfigure[$\mu_{12},\mu_{15}$ and $\mu_{16}$.]{\includegraphics{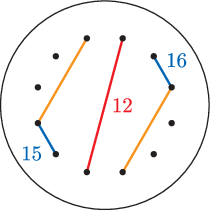}\label{fig:vp_1_4}}

\subfigure[$\mu_{19},\ldots, \mu_{22}$.]{\includegraphics{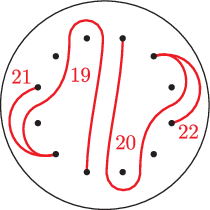}\label{fig:vp_1_5}}
\subfigure[$\mu_{23},\ldots,\mu_{26}$.]{\includegraphics{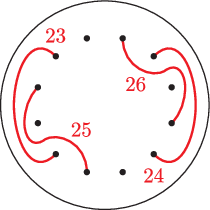}\label{fig:vp_1_6}}
\subfigure[$\mu_{27}$ and $\mu_{28}$.]{\includegraphics{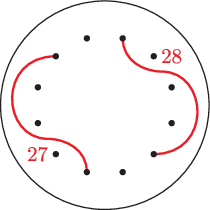}\label{fig:vp_1_7}}
\subfigure[Reference paths.]{\includegraphics{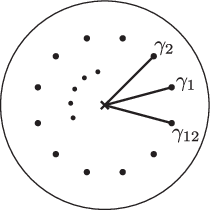}\label{fig:vp_1_8}}
\caption{Simple paths in $D^2$. 
Note that $\epsilon_i=1,2,3$ when $\mu_i$ is described in green, red, blue, respectively.}
\label{fig:vanishing path 1}
\end{figure}

\begin{figure}[htbp]
\subfigure[$c_1$ and the axis for $\iota$.]{\includegraphics{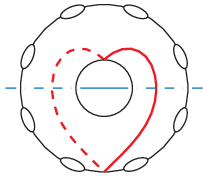}\label{fig:vc_1_1}}
\subfigure[$c_2$ (red) and $c_3$ (blue).]{\includegraphics{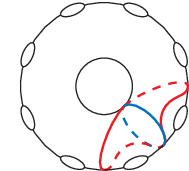}\label{fig:vc_1_2}}
\subfigure[$c_4$.]{\includegraphics{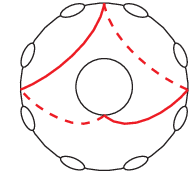}\label{fig:vc_1_3}}
\subfigure[$c_5$ (red) and $c_6$ (blue).]{\includegraphics{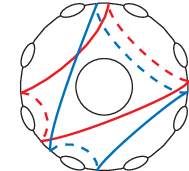}\label{fig:vc_1_4}}
\caption{Simple closed curves in $\Sigma_1^8$.}
\label{fig:vanishing cycle 1}
\end{figure}

\begin{theorem}\label{thm:relation for deg2 Veronese}

The product $\tau_{\mu_1}^{\epsilon_1}\cdots \tau_{\mu_{56}}^{\epsilon_{56}}$ is equal to $t_{\Pa}$ in $B_{12} = \mathcal{M}_{\Pa}(D^2,R)$. 
The product $t_{c_1}\cdots t_{c_{12}}\in \mathcal{M}_\Pa(\Sigma_1^8)$ is equal to $t_{\delta_1}\cdots t_{\delta_8}$. 
Furthermore, $\theta$ and the factorization $\tau_{\mu_1}^{\epsilon_1}\cdots \tau_{\mu_{56}}^{\epsilon_{56}}=t_\Pa$ are compatible.  

\end{theorem}

\begin{proof}[Proof of \Cref{thm:relation for deg2 Veronese}]

We first give several lemmas for calculating the product $\tau_{\mu_1}^{\epsilon_1}\cdots \tau_{\mu_{56}}^{\epsilon_{56}}$.

\begin{lemma}[{\cite[Lemma 2]{MTIV}}]\label{lem:relation for cusps}

Let $\sigma_1,\sigma_2\in B_3$ be the standard generators of the braid group $B_3$. 
The product $(\overline{\sigma_1}\sigma_2^3\sigma_1)\sigma_2(\sigma_1\sigma_2^3\overline{\sigma_1})$ is equal to $(\sigma_1\sigma_2)^6 \overline{\sigma_1}^3$. 

\end{lemma}

\begin{lemma}\label{lem:relation double quotient for 5pts}

Let $d_1',d'_2\subset \Sigma_0^4$ be the simple closed curves given in \Cref{fig:scc for doublequotient}, $\kappa'':\Sigma_0^4\to \Sigma_0^4$ be the $2\pi/3$-degree clockwise rotation around the center in \Cref{fig:scc for doublequotient 1}, and $d_{2k+i}'=(\kappa'')^{k}(d_i')$ for $i,k=1,2$. 
The following relation holds in $\mathcal{M}_\Pa(\Sigma_0^4)$:
\begin{equation}\label{eqn:relation double quotient for 5pts}
t_{d_1'}^2\cdots t_{d_6'}^2 = t_{\delta_1}^4\cdots t_{\delta_4}^4,
\end{equation}
where $\delta_1,\ldots, \delta_4\subset \Sigma_0^4$ are simple closed curves parallel to the four boundary components.

\begin{figure}[htbp]
\centering
\subfigure[$d_1'$ and the "center" of $\Sigma_0^4$.]{\includegraphics{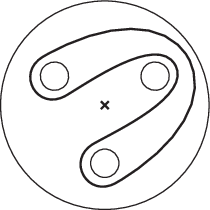}\label{fig:scc for doublequotient 1}}
\subfigure[$d_2'$.]{\includegraphics{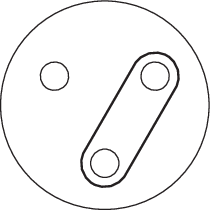}\label{fig:scc for doublequotient 2}}
\caption{Simple closed curves in $\Sigma_0^4$.}
\label{fig:scc for doublequotient}
\end{figure}

\end{lemma}

\begin{proof}[Proof of \Cref{lem:relation double quotient for 5pts}]
By the lantern relation (\cite[Proposition 5.1]{FarbMargalitBookMCG}), the following equalities hold: 
\[
t_{d_1'}t_{d_2'}t_{d_6'} =t_{d_2'}t_{d_3'}t_{d_4'}=t_{d_4'}t_{d_5'}t_{d_6'} = t_{d_2'}t_{d'_4}t_{d_6'}= t_{\delta_1}\cdots t_{\delta_4}.
\]
Using these relations, the product $t_{d_1'}^2\cdots t_{d_6'}^2$ is calculated as follows: 
{\allowdisplaybreaks
\begin{align*}
t_{d'_1}^2\cdots t_{d'_6}^2 =&(t_{\delta_1}\cdots t_{\delta_4})^5t_{d'_1}\overline{t_{d'_6}}\overline{t_{d'_4}}\overline{t_{d'_2}}\overline{t_{d'_6}}\overline{t_{d'_4}}t_{d'_6}\\
=&(t_{\delta_1}\cdots t_{\delta_4})^3t_{d'_1}t_{d'_2}t_{d'_6}=(t_{\delta_1}\cdots t_{\delta_4})^4.
\end{align*}
}%
This completes the proof of the lemma.
\end{proof}

Let $p':\Sigma_0^6\to \Sigma_0^4$ be the double covering corresponding to the element in $H^1(\Sigma_0^4;\Z/2\Z)$ represented by the cocycle sending $\delta_1$ to $1$ and $\delta_i$ to $0$ for $i=2,3$. 
It is easy to see that a representative of each factor in \eqref{eqn:relation double quotient for 5pts} admits a lift by $p'$ whose restriction on $\Pa \Sigma_0^6$ is the identity. 
Furthermore, an isotopy between compositions of the representatives for both sides of \eqref{eqn:relation double quotient for 5pts} is also lifted to an isotopy between their lifts. 
By this procedure, we obtain the following relation in $\mathcal{M}_\Pa(\Sigma_0^6)$.
\begin{equation}\label{eqn:relation single quotient for 5pts}
t_{d_1}^2t_{d_2}^2t_{d_3}t_{d_4}t_{d_5}^2t_{d_6}^2t_{d_7}=t_{d_0}^{-1}t_{\delta_0}^2t_{\delta_0'}^2t_{\delta_1}^4\cdots t_{\delta_4}^4,
\end{equation}
where $d_1,\ldots, d_7\subset \Sigma_0^6$ are simple closed curves given in \Cref{fig:scc for singlequotient 1}, $\delta_0, \delta_0'\subset \Sigma_0^6$ are those parallel to the outer and the central boundary components in \Cref{fig:scc for singlequotient 1}, and $\delta_1,\ldots, \delta_4\subset \Sigma_0^6$ are those parallel to the other boundary components. 

\begin{figure}[htbp]
\centering
\subfigure[$d_0,\ldots,d_3$]{\includegraphics{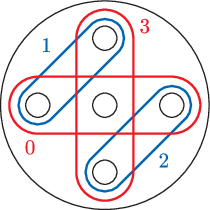}\label{fig:scc for singlequotient 1 1}}
\subfigure[$d_4$.]{\includegraphics{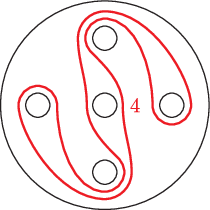}\label{fig:scc for singlequotient 1 2}}
\subfigure[$d_5$ and $d_6$.]{\includegraphics{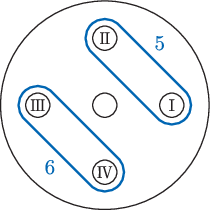}\label{fig:scc for singlequotient 1 3}}
\subfigure[$d_7$.]{\includegraphics{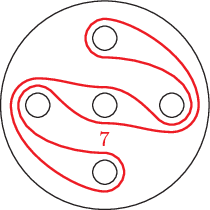}\label{fig:scc for singlequotient 1 4}}
\caption{Simple closed curves in $\Sigma_0^6$.
Each curve is drawn in red (resp.~blue) if the index for it in the relation \eqref{eqn:relation single quotient for 5pts} is $1$ (resp.~$2$).}
\label{fig:scc for singlequotient 1}
\end{figure}

Let $R'\subset \Sigma_0^2$ be the set of six points shown in \Cref{fig:scc for singlequotient 2}. 
By capping the central boundary of $\Sigma_0^6$ by a disk and then embedding the resulting planar surface into $\Sigma_0^2$ so that each boundary component of $\Sigma_0^6$ (labeled with I--IV, given in \Cref{fig:scc for singlequotient 1 3}) is arranged as described in \Cref{fig:scc for singlequotient 2 1}, we obtain the following relation in $\mathcal{M}_\Pa(\Sigma_0^2,R')$:
\begin{align}\nonumber
&(\tau_{\mu'_1}\tau_{\beta'})^6t_{e_2}^2\tau_{\mu'_3}^2t_{e_4}t_{e_5}^2(\tau_{\mu'_6}\tau_{\beta'})^6\tau_{\mu'_7}^2=t_{e_0}^{-1}t_{\widetilde{\delta_0}}^2t_{\widetilde{\delta_1}}^4\tau_{\beta'}^8,\\
\Leftrightarrow&\left((\tau_{\mu'_1}\tau_{\beta'})^6\tau_{\beta'}^{-3}\right)t_{e_2}^2\tau_{\mu'_3}^2\left(t_{e_4}\tau_{\beta'}^{-2}\right)t_{e_5}^2\left((\tau_{\mu'_6}\tau_{\beta'})^6\tau_{\beta'}^{-3}\right)\tau_{\mu'_7}^2=t_{e_0}^{-1}t_{\widetilde{\delta_0}}^2t_{\widetilde{\delta_1}}^4,
\end{align}
where $e_i,\mu_j',\beta'$ are described in \Cref{fig:scc for singlequotient 2}, $\widetilde{\delta_0}$ is the image of the outer boundary component of $\Sigma_0^6$ by the embedding $\Sigma_0^6\hookrightarrow \Sigma_0^2$, and $\widetilde{\delta_1}$ is that of the boundary component labeled with I.

\begin{figure}[htbp]
\centering
\subfigure[Configurations of components of $\Pa \Sigma_0^6$.]{\includegraphics{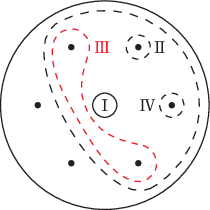}\label{fig:scc for singlequotient 2 1}}
\subfigure[$e_0,\mu_1',e_2$ and $\beta'$.]{\includegraphics{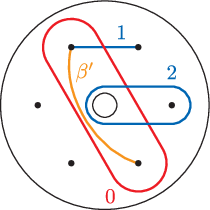}\label{fig:scc for singlequotient 2 2}}
\subfigure[$\mu_3'$.]{\includegraphics{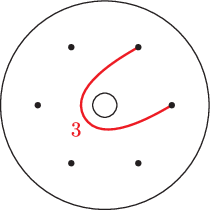}\label{fig:scc for singlequotient 2 3}}
\subfigure[$e_4$.]{\includegraphics{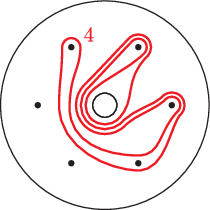}\label{fig:scc for singlequotient 2 4}}
\subfigure[$e_5$ and $\mu_6'$.]{\includegraphics{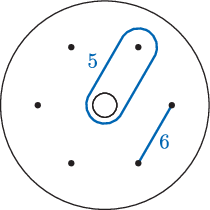}\label{fig:scc for singlequotient 2 5}}
\subfigure[$\mu_7'$.]{\includegraphics{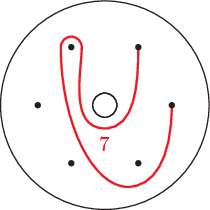}\label{fig:scc for singlequotient 2 6}}
\subfigure[$\mu_8'$ and $\mu_9'$.]{\includegraphics{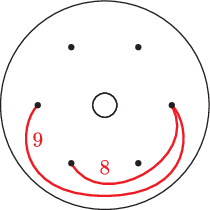}\label{fig:scc for singlequotient 2 7}}
\subfigure[$\mu_{10}'$ and $\mu_{11}'$.]{\includegraphics{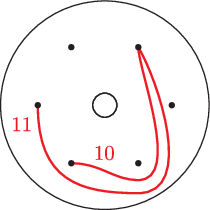}\label{fig:scc for singlequotient 2 8}}
\caption{Simple closed curves $e_i$ and simple paths $\mu_j'$ and $\beta'$ in $\Sigma_0^2$.}
\label{fig:scc for singlequotient 2}
\end{figure}

Let $\kappa':\Sigma_0^2\to \Sigma_0^2$ be the $\pi$-degree rotation, which preserves $R'$ and $\Pa \Sigma_0^2$ setwise. 
One can obtain the following relation by applying the lantern relation repeatedly:
\[
t_{e_0}^{-1}t_{\widetilde{\delta_0}}^2t_{\widetilde{\delta_1}}^4 \tau_{\mu_8'}^2\tau_{\mu_9'}^2\tau_{\mu_{10}'}^2\tau_{\mu_{11}'}^2 = t_{\delta} t_{\widetilde{\delta_1}}^4 t_{\widetilde{\delta_0}}t_{\kappa'(\widetilde{\delta_0})}^{-1},
\]
where $\delta\subset \Sigma_0^2$ is a simple closed curve parallel to the outer boundary component.  
Combining this relation with \eqref{eqn:relation single quotient for 5pts}, we obtain:
\begin{align*}
&\left((\tau_{\mu'_1}\tau_{\beta'})^6\tau_{\beta'}^{-3}\right)t_{e_2}^2\tau_{\mu'_3}^2\left(t_{e_4}\tau_{\beta'}^{-2}\right)t_{e_5}^2\left((\tau_{\mu'_6}\tau_{\beta'})^6\tau_{\beta'}^{-3}\right)\tau_{\mu'_7}^2\tau_{\mu_8'}^2\tau_{\mu_9'}^2\tau_{\mu_{10}'}^2\tau_{\mu_{11}'}^2\\
&\cdot \left((\tau_{\kappa'(\mu'_1)}\tau_{\kappa'(\beta')})^6\tau_{\kappa'(\beta')}^{-3}\right)t_{\kappa'(e_2)}^2\tau_{\kappa'(\mu'_3)}^2\left(t_{\kappa'(e_4)}\tau_{\kappa'(\beta')}^{-2}\right)t_{\kappa'(e_5)}^2\left((\tau_{\kappa'(\mu'_6)}\tau_{\kappa'(\beta')})^6\tau_{\kappa'(\beta')}^{-3}\right)\tau_{\kappa'(\mu'_7)}^2\\
&\cdot \tau_{\kappa'(\mu_8')}^2\tau_{\kappa'(\mu_9')}^2\tau_{\kappa'(\mu_{10}')}^2\tau_{\kappa'(\mu_{11}')}^2\\
=&t_{\delta}^2 t_{\widetilde{\delta_1}}^8.
\end{align*}
Let $p:\Sigma_0^2\to \Sigma_0^2$ be a double covering. 
Taking a lift of each factor of the relation above by $p$ as before, and capping the central boundary of $\Sigma_0^2$ by a disk, we eventually obtain the desired relation $\tau_{\mu_1}^{\epsilon_1}\cdots \tau_{\mu_{56}}^{\epsilon_{56}} = t_{\Pa}$. (One can show that the lifts of $t_{e_4}\tau_{\beta'}^{-2}$ becomes $\tau_{\mu_{10}}\tau_{\mu_{11}}$ after capping the boundary by the Alexander method (see \cite{FarbMargalitBookMCG} for this method) and those of $(\tau_{\mu'_1}\tau_{\beta'})^6\tau_{\beta'}^{-3}$ and $(\tau_{\mu'_6}\tau_{\beta'})^6\tau_{\beta'}^{-3}$ become $\tau_{\mu_2}^3\cdots \tau_{\mu_7}^3$ and $\tau_{\mu_{13}}^3\cdots \tau_{\mu_{18}}^3$, respectively, by using \Cref{lem:relation for cusps}.)

Before calculating the product $t_{c_1}\cdots t_{c_{12}}\in \mathcal{M}_\Pa(\Sigma_1^8)$, we introduce a root $\iota'$ of the involution $\iota:\Sigma_1^8\to \Sigma_1^8$, which clarifies symmetry of relevant simple closed curves in $\Sigma_1^8$. 
In order to define $\iota'$, we first move the boundary components of $\Sigma_1^8$ as described in \Cref{fig:surf for symmetry 1}, yielding the surface in \Cref{fig:surf for symmetry 2}.
We next decompose the surface into four rectangles given in \Cref{fig:surf for symmetry 3} by cutting it along the curves $a,\ldots, h\subset \Sigma_1^8$ in \Cref{fig:surf for symmetry 2}.
Let $\iota':\Sigma_1^8\to \Sigma_1^8$ be the diffeomorphism sending each rectangle in \Cref{fig:surf for symmetry 3} as the arrows indicate. 
\begin{figure}[htbp]
\centering
\subfigure[Trajectories of the components of $\Pa \Sigma_1^8$.]{\includegraphics{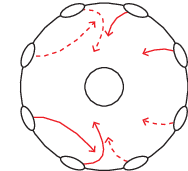}\label{fig:surf for symmetry 1}}
\subfigure["Perforated curves".]{\includegraphics{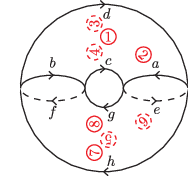}\label{fig:surf for symmetry 2}}
\subfigure[The diffeomorphism $\iota'$.]{\includegraphics{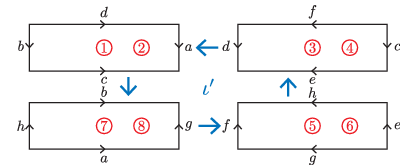}\label{fig:surf for symmetry 3}}
\caption{How to define $\iota'$.}
\label{fig:surf for symmetry}
\end{figure}
One can easily check that the square $(\iota')^2$ is equal to the involution $\iota$, and the image $\iota'(c_i)$ coincides with $c_{i+3}$ (taking indices modulo $12$).

The following relation holds in $\mathcal{M}_{\Pa}(\Sigma_1^8)$ (\cite{KorkmazOzbagci2008}):
\begin{equation}\label{eqn:four holed torus relation}
(t_{a_1}t_{a_3}t_bt_{a_2}t_{a_4}t_b)^2 = t_{\sigma_1}t_{\sigma_2}t_{\sigma_3}t_{\sigma_4},
\end{equation}
where $a_i,b,\sigma_j\subset \Sigma_1^8$ are simple closed curves shown in \Cref{fig:vc_2}.
\begin{figure}[htbp]
\centering
\subfigure[$a_1,\ldots, a_4$.]{\includegraphics{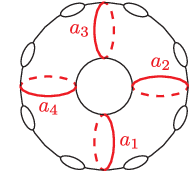}\label{fig:vc_2_1}}
\subfigure[$\sigma_1,\sigma_3$ and $b_1$.]{\includegraphics{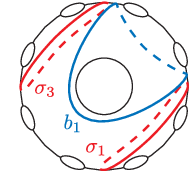}\label{fig:vc_2_2}}
\subfigure[$\sigma_2, \sigma_4$ and $b$.]{\includegraphics{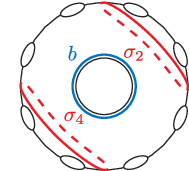}\label{fig:vc_2_3}}
\subfigure[$b_2$.]{\includegraphics{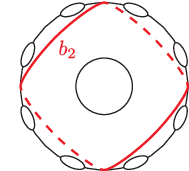}\label{fig:vc_2_4}}
\caption{Simple closed curves in $\Sigma_1^8$.}
\label{fig:vc_2}
\end{figure}
By direct calculations, one can easily check that the following equalities hold:
{\allowdisplaybreaks
\begin{align*}
t_b &= (t_{a_1}t_{a_3}t_bt_{a_2}t_{a_4})^{-1}t_{b_2}(t_{a_1}t_{a_3}t_bt_{a_2}t_{a_4}), \\
t_{a_3}&=(t_{b}t_{a_2})^{-1}t_{b_1}(t_{b}t_{a_2}), \\
t_{c_1}&=t_{a_1}t_bt_{a_1}^{-1}\mbox{ and }t_{c_4}=t_{b_1}t_{a_4}t_{b_1}^{-1}. 
\end{align*}
}%
Thus, the product $(t_{a_1}t_{a_3}t_bt_{a_2}t_{a_4}t_b)^2$ is transformed as follows: 
{\allowdisplaybreaks
\begin{align*}
(t_{a_1}t_{a_3}t_bt_{a_2}t_{a_4}t_b)^2 &= (t_{a_1}t_{a_3}t_bt_{a_2}t_{a_4}t_{b_2})^2\\
&=(t_{a_1}t_{a_3}t_bt_{a_2}t_{a_4}t_{b_2})(t_{\iota(a_1)}t_{\iota(a_3)}t_{\iota(b)}t_{\iota(a_2)}t_{\iota(a_4)}t_{\iota(b_2)})\\
&=(t_{a_1}t_bt_{a_2}t_{b_1}t_{a_4}t_{b_2})(t_{\iota(a_1)}t_{\iota(b)}t_{\iota(a_2)}t_{\iota(b_1)}t_{\iota(a_4)}t_{\iota(b_2)})\\
&=(t_{c_1}t_{a_1}t_{a_2}t_{c_4}t_{b_1}t_{b_2})(t_{\iota(c_1)}t_{\iota(a_1)}t_{\iota(a_2)}t_{\iota(c_4)}t_{\iota(b_1)}t_{\iota(b_2)}).
\end{align*}
}%
One can deduce the following equalities from the lantern relation: 
{\allowdisplaybreaks
\begin{align*}
t_{a_1}t_{a_2} &= t_{c_2}t_{c_3}t_{\sigma_1}^{-1}t_{\delta_1}t_{\delta_2}, \\
t_{b_1}t_{b_2} &= t_{c_5}t_{c_6}t_{\sigma_4}^{-1}t_{\delta_7}t_{\delta_8}. 
\end{align*}
}%
We thus obtain:
{\allowdisplaybreaks
\begin{align*}
(t_{a_1}t_{a_3}t_bt_{a_2}t_{a_4}t_b)^2 &= t_{c_1}\cdots t_{c_{12}}\cdot t_{\sigma_1}t_{\sigma_2}t_{\sigma_3}t_{\sigma_4}t_{\delta_1}^{-1}\cdots t_{\delta_8}^{-1}. 
\end{align*}
}%
Combining it with \eqref{eqn:four holed torus relation}, we eventually obtain the desired relation $t_{c_1}\cdots t_{c_{12}}= t_{\delta_1}\cdots t_{\delta_8}$. 

\begin{remark}\label{rem:example LP is P1P1}

One can indeed show that the factorization $t_{c_1}\cdots t_{c_{12}}$ of the boundary multi-twist is Hurwitz equivalent to that for the Lefschetz pencil $f_s:\PP^1\times \PP^1\dasharrow \PP^1$ given in \cite[Table 1]{HH2021genus1holLP}. 

\end{remark}

Lastly, we check the compatibility condition with $\theta$ for each $\mu_i$. 
By the symmetries induced by $\iota', \iota=(\iota')^2$ and $\kappa$, it is enough to check the condition for $\mu_i$ with 
$i=1,11$, and $i=2j$ ($j\leq 14$). 
The paths $\mu_1,\mu_{10},\mu_{11},\mu_{12}$ go through the center $r_0$.
In order to check the condition, we move the other paths by isotopy so that these go through the center $r_0$ as (implicitly) shown in \Cref{fig:vp_2}. 
We denote the resulting pairs of reference paths by $\beta_i^1, \beta_i^2$.
\begin{figure}[htbp]
\centering
\subfigure[]{\includegraphics{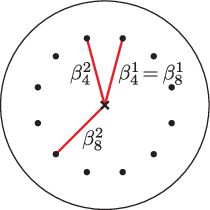}\label{fig:vp_2_1}}
\subfigure[]{\includegraphics{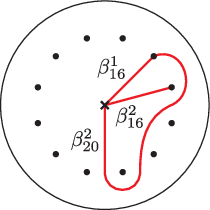}\label{fig:vp_2_2}}
\subfigure[]{\includegraphics{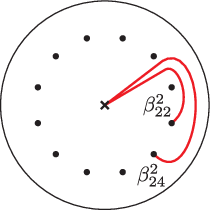}\label{fig:vp_2_3}}
\caption{Reference paths. We put $\beta_{20}^1=\beta_{26}^1=\beta_{28}^1:=\beta_4^1$, $\beta_{22}^1=\beta_{24}^1:=\beta_{16}^1$, $\beta_{26}^2 :=\beta_{22}^2$, and $\beta_{28}^2 :=\beta_{24}^2$. }
\label{fig:vp_2}
\end{figure}
One can easily check that the vanishing cycles associated with $\beta_i^1$ and $\beta_i^2$ are as given in \Cref{tab:vanishing cycles}, and up to isotopy, these intersect in one point transversely (resp.~these are disjoint, these coincide) if $\epsilon_i=3$ (resp.~$2$, $1$). 
This eventually completes the proof of \Cref{thm:relation for deg2 Veronese}.
\begin{table}[htbp]
\centering
\begin{tabular}{|r|l|l|}
\hline
$i$ & $\epsilon_i$ &  Vanishing cycles \\
\hline\hline
$1$ & $2$ & $c_2$ and $c_8$\\
\hline
$2$ & $3$ & $c_3$ and $t_{c_4}^{-1}(c_7)$\\
\hline
$4$ & $3$ & $c_3$ and $c_4$\\
\hline
$6$ & $3$ & $c_3$ and $c_7$\\
\hline
$8$ & $2$ & $c_3$ and $c_8$\\
\hline
$10$ & $1$ & $t_{c_3}^{-1}\left(t_{c_8}(c_7)\right)$ and $t_{c_9}^{-1}\left(t_{c_2}(c_1)\right)$\\
\hline
$11$ & $1$ & $t_{c_3}^{-1}\left(t_{c_8}t_{c_7}(c_4)\right)$ and $t_{c_9}^{-1}\left(t_{c_2}t_{c_1}(c_{10})\right)$\\
\hline
$12$ & $2$ & $c_3$ and $c_9$\\
\hline
$14$ & $3$ & $c_2$ and $c_{10}$\\
\hline
$16$ & $3$ & $c_2$ and $c_1$\\
\hline
$18$ & $3$ & $c_2$ and $t_{c_1}(c_{10})$\\
\hline
$20$ & $2$ & $c_3$ and $t_{c_{10}}^{-1}\left(t_{c_1}^{-1}(c_2)\right)$\\
\hline
$22$ & $2$ & $c_2$ and $t_{c_1}(c_{12})$\\
\hline
$24$ & $2$ & $c_2$ and $t_{c_1}\left(t_{c_{12}}(c_{11})\right)$\\
\hline
$26$ & $2$ & $c_3$ and $t_{c_1}(c_{12})$\\
\hline
$28$ & $2$ & $c_3$ and $t_{c_1}\left(t_{c_{12}}(c_{11})\right)$\\
\hline
\end{tabular}
\V{.3em}
\caption{Vanishing cycles associated with $\beta_i^1$ and $\beta_i^2$. }
\label{tab:vanishing cycles}
\end{table}
\end{proof}

By \Cref{thm:combinatorial construction pencil-fibration,thm:relation for deg2 Veronese}, we obtain a genus-$1$ Lefschetz bipencil $F:X\dashrightarrow \PP^2$.
Furthermore, one can easily check that for each path $\mu_i$ with $\epsilon_i=2$, the corresponding vanishing cycles $\beta_i^1$ and $\beta_i^2$ are not isotopic. 
Thus, by \Cref{prop:existence cohomology for symp form,prop:symplectic structure blowdowns}, $X$ admits a symplectic structure. 
In what follows, we calculate topological invariants of $X$ using the results in \Cref{sec:top inv LbF}.  
The Euler characteristic of $X$ is $2(3-3\cdot 1+12)-4-2\cdot 8=4$. 
Let $W$ be the closure of a regular fiber of $\rho\circ F:X\dasharrow \PP^1$, which admits a Lefschetz pencil corresponding to the factorization $t_{c_1}\cdots t_{c_{12}}=t_{\delta_1}\cdots t_{\delta_8}$. 
The handle decomposition of $W$ associated with the Lefschetz pencil has two $1$-handles, both of which can be canceled with $2$-handles. 
Thus, $W$ is geometrically simply connected and $\pi_1(X) = 1$.  

As mentioned in \Cref{rem:example LP is P1P1}, $W$ is diffeomorphic to $\PP^1\times \PP^1$.
In particular, $H_2(W)$ is isomorphic to $\Z\oplus \Z$. 
Furthermore, any Lagrangian sphere in $W$ represents $(1,-1)\in \Z\oplus \Z\cong H_2(W)$ since it has self-intersection $-2$ by the Lagrangian neighborhood theorem. 
Thus, $H_2(X)$ is isomorphic to $\Z\oplus \Z / \left<(1,-1)\right>\cong \Z$. 
In what follows, we give an alternate proof of this fact without relying on geometric properties of $W$ and using only monodromy information (in particular applicable to general pencil-fibration structures constructed by \Cref{thm:combinatorial construction LbF over Hir_m,thm:combinatorial construction pencil-fibration}). 

Let $x_i\in W$ ($i=1,\ldots, 8$) be the base point corresponding to the component of $\Pa \Sigma_1^8$ labeled with $i$ in \Cref{fig:surf for symmetry}, and $W'$ be the $4$-manifold obtained from $W$ by removing neighborhoods of $x_2,\ldots, x_8$, and then blowing up at $x_1$. 
The manifold $W'$ admits a Lefschetz fibration over $\PP^1$ with fiber $\Sigma_1^7$ and a $(-1)$-section. 
In particular, $W'$ can  be decomposed as follows:
\[
W' = (D^2\times \Sigma_1^7)\cup (h^2_1\sqcup\cdots \sqcup h^2_{12})\cup (D^2\times \Sigma_1^7), 
\]
where $h^2_i$ is the $2$-handle corresponding to the Lefschetz singularity with the vanishing cycle $c_i$. 
(Note that we regard $\Sigma_1^8$ as a subsurface of $\Sigma_1^7$ in the obvious way.)
Furthermore, the former (resp.~the latter) $D^2\times \Sigma_1^7$ can be regarded as one $0$-handle and eight $1$-handles (resp.~one $2$-handle and eight $3$-handles) so that the union of $0$- and $2$-handles is the $(-1)$-section. 
Let $W'_k\subset W'$ be the union of handles with indices less than or equal to $k$, $C_k= H_k(W_k',W'_{k-1})$, and $\Pa_k:C_k\to C_{k-1}$ be the boundary operator. 
It is easy to see that the groups $C_1$ and $C_3$ are both isomorphic to $H_1(\Sigma_1^7)$, while $C_2$ is isomorphic to $\Z^{12}\oplus \Z$.
One can show the following lemma in the same way as that in \cite[Appendix]{HamadaHayano}:

\begin{lemma}\label{lem:representing H_2(W) by vc}

Under the isomorphisms for $C_k$'s, the maps $\Pa_2:\Z^{12}\oplus \Z\to H_1(\Sigma_1^7)$ and $\Pa_3:H_1(\Sigma_1^7)\to \Z^{12}\oplus \Z$ are calculated as follows: 

\begin{itemize}

\item 
$\Pa_2(e_i):= [c_i]\in H_1(\Sigma_1^7)$ and $\Pa_2(e_0)=0$, where $e_1,\ldots, e_{12}\in \Z^{12}$ are the standard generators and $e_0$ is the generator of the $\Z$-summand.

\item 
For each $\alpha=\alpha_0\in H_1(\Sigma_1^7)$, we inductively take $\alpha_i\in H_1(\Sigma_1^7)$ and $m_i\in \Z$ ($i=1,\ldots, 12$) by $m_{i}=\left<\alpha_{i-1},[c_i]\right>$ and $\alpha_i = \alpha_{i-1} - m_i[c_i]\in H_1(\Sigma_1^7)$, where $\left<~,~\right>$ is the intersection pairing.
The image $\Pa_3(\alpha)$ is then equal to $\sum_{i=1}^{12} m_i e_i$. 

\end{itemize}

\end{lemma}

\noindent
We take oriented simple closed curves $a,b,\delta_1,\ldots,\delta_8\subset \Sigma_1^8$ as shown in \Cref{fig:generator homology}, and denote the homology classes represented by these curves by the same symbols $a,b,\delta_1,\ldots, \delta_8\in H_1(\Sigma_1^8)$. 
\begin{figure}[htbp]
\centering
\subfigure[]{\includegraphics{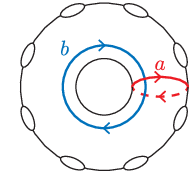}\label{fig:generator homology 1}}
\subfigure[]{\includegraphics{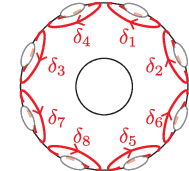}\label{fig:generator homology 2}}
\caption{Simple closed curves generating $H_1(\Sigma_1^8)$.}
\label{fig:generator homology}
\end{figure}
The following equalities hold: 
\[
\sum_{i=1}^{8}\delta_i=0,\hspace{.3em}\iota'_\ast(a)=-b,\hspace{.3em}\iota'_\ast(b)=a\mbox{, and } \iota'_\ast(\delta_i)=\delta_{i-2}\mbox{ (taking indices modulo $8$).}
\]  
Furthermore, under suitable orientations, $c_1,c_2$ and $c_3$ in \Cref{fig:vanishing cycle 1} represent the classes $a+b-\delta_5-\delta_6, a-\delta_5$ and $a-\delta_6$, respectively. 
One can deduce the following lemma from these observations (note that $\delta_1=0$ in $H_1(\Sigma_1^7)$):

\begin{lemma}\label{lem:kernel and image of boundary operators}

The kernel of $\Pa_2$ in \Cref{lem:representing H_2(W) by vc} has a basis 
\begin{align*}
&\left\{e_1+e_4+e_7+e_{10},e_2+e_3+e_5+e_6+e_8+e_9+e_{11}+e_{12},\right.\\
&\hspace{1em}\left.e_2+e_3+e_4+e_8+e_9+e_{10},e_7+e_8-e_9+e_{10}-e_{11}-e_{12}\right\}
\end{align*}
and the image of $\Pa_3$ in \Cref{lem:representing H_2(W) by vc} has a basis
\[
\left\{e_1+e_4+e_7+e_{10},e_2+e_3+e_5+e_6+e_8+e_9+e_{11}+e_{12}\right\}.
\]

\end{lemma}

\noindent
By \Cref{lem:representing H_2(W) by vc,lem:kernel and image of boundary operators}, one can conclude that the homology group $H_2(W)$ is isomorphic to $\Z\oplus \Z$, as expected. (Note that $W$ is diffeomorphic to $\PP^1\times \PP^1$. cf.~\Cref{rem:example LP is P1P1}.) 

In order to calculate the representatives in $\Ker \Pa_2\subset H_2(W'_2,W'_1)$ of the vanishing cycles $L_1,\ldots, L_4$ of the Lefschetz pencil $\rho\circ F:X\dashrightarrow \PP^1$, we first observe relation among reference paths for the Lefschetz pencil on $W'$, corresponding vanishing cycles, and elements in $H_2(W'_2,W'_1)$.  
For a given reference path $\gamma \subset D^2\subset \PP^1$ for the Lefschetz pencil on $W'$, one can obtain a $2$-disk $D_\gamma\subset W'_2$, called a \textit{Lefschetz thimble} in \cite{SeidelbookFukayaPicard}, which contains the Lefschetz critical point over the end of $\gamma$, and bounds a vanishing cycle associated with $\gamma$. 
It is easy to see that $D_\gamma$ represents an element in $H_2(W'_2,W'_1)$, the image $\Pa_2([D_\gamma])\in H_1(\Sigma_1^7)$ is represented by a vanishing cycle associated with $\gamma$, and $[D_{\gamma_i}]$ is equal to $e_i\in \Z^{12}\subset H_2(W'_2,W'_1)$ for the reference path $\gamma_i$ in \Cref{fig:vp_1_8}.
Furthermore, one can deduce the following lemma from relation between elementary transformations for reference paths and handle decompositions observed in \cite[Exercises 8.2.7.~(a)]{GS19994-mfd}: 

\begin{lemma}\label{lem:relation elem trans & homology}

Let $\gamma$ be a reference path from $r_0$ and $\widetilde{\gamma}$ be another path obtained by applying elementary transformation with $\gamma_i$, as shown in \Cref{fig:elem trans}. 
Then $[D_{\widetilde{\gamma}}]$ is equal to $[D_{\gamma}] - \left<\Pa_2([D_{\gamma}]),c_i\right>e_i$. 
\begin{figure}[htbp]
\centering
\subfigure[$\gamma$ and $\gamma_i$.]{\includegraphics{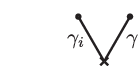}\label{fig:elem trans 1}}
\subfigure[$\widetilde{\gamma}$.]{\includegraphics{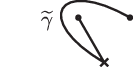}\label{fig:elem trans 2}}
\caption{Reference paths.}
\label{fig:elem trans}
\end{figure}

\end{lemma}

\noindent
The reference path $\beta_{10}^1$ is given in \Cref{fig:matching path 2}. 
This path can be obtained by applying the inverse of the elementary transformation with $\gamma_3$ to the reference path in \Cref{fig:matching path 1}, and the path in \Cref{fig:matching path 1} can also be obtained by applying the elementary transformation with $\gamma_8$ to $\gamma_7$. 
\begin{figure}[htbp]
\centering
\subfigure[]{\includegraphics{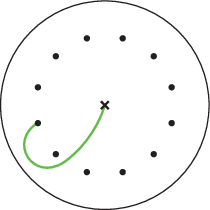}\label{fig:matching path 1}}
\subfigure[]{\includegraphics{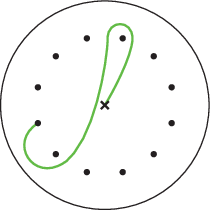}\label{fig:matching path 2}}
\caption{\subref{fig:matching path 1}~:~an intermediate reference path. \subref{fig:matching path 2}~:~the reference path $\beta_{10}^1$. }
\label{fig:matching path}
\end{figure}
We thus obtain:
\[
[D_{\beta_{10}^1}] = \left(e_7 - \left<c_7,c_8\right>e_8\right)+\left<c_7 - \left<c_7,c_8\right>c_8,c_3\right>e_3=e_7 - e_8-e_3.
\]
By the symmetry, we also obtain $[D_{\beta_{10}^2}] = e_1-e_2-e_9$.
Since $L_1$ is a matching cycle associated with $\gamma_{10}$, the class $[L_1]$ is represented by 
\[
[D_{\beta_{10}^1}]+[D_{\beta_{10}^2}] = e_1-e_2-e_3+e_7-e_8-e_9. 
\]
(Note that we take orientations of the thimbles so that $\Pa_2([D_{\beta_{10}^1}]+[D_{\beta_{10}^2}])=0$.)
We can also obtain the following equality in a similar manner:
\[
[D_{\beta_{11}^1}] = e_4 + 2 e_7 -e_8-e_3. 
\]
Thus, $[L_2]$ is represented by 
\[
2e_1-e_2-e_3+e_4+2e_7-e_8-e_9+e_{10} = (e_1-e_2-e_3+e_7-e_8-e_9)+(e_1+e_4+e_7+e_{10}).
\] 
Since $e_1+e_4+e_7+e_{10}$ is contained in $\Im \Pa_3$, $[L_1]$ is equal to $[L_2]$ in $H_2(W')$. 
By the symmetry, we obtain the following equalities:
{\allowdisplaybreaks
\begin{align*}
[L_3]=&[e_4-e_5-e_6+e_{10}-e_{11}-e_{12}],\\
[L_4]=&[2e_4-e_5-e_6+e_7+2e_{10}-e_{11}-e_{12}+e_{1}].
\end{align*}
}%
Both of the elements are equal to $[L_1]$ in $H_2(W')$. 
Hence, $H_2(X)$ is isomorphic to 
\[
(\Ker \Pa_2/\Im \Pa_3)/\left<e_2+e_3+e_4+e_8+e_9+e_{10}\right>\cong \Z
\]
and generated by $A_0:=[[e_7+e_8-e_9+e_{10}-e_{11}-e_{12}]]$. 

We next calculate the value $\left<c_1(X),A_0\right>$. 
Let $W''\subset W'$ be the complement of a tubular neighborhood of a regular fiber in $W'$, which admits a Lefschetz fibration over the disk. 
It is easy to see that the handle decomposition we considered above induces the isomorphism $H_2(W'')\cong \Ker \Pa_2$. 
We denote by $\alpha_0''\in H_2(W'')$ the element corresponding to $e_7+e_8-e_9+e_{10}-e_{11}-e_{12}\in \Ker \Pa_2$ via this isomorphism. 
The universal covering $\R^2\to \Sigma_1$ induces a trivialization of $T\Sigma_1$, and that of $T\Sigma_1^7$ via the inclusion $\Sigma_1^7\hookrightarrow \Sigma_1$. 
It is easy to see that the rotation number of any essential simple closed curve in $\Sigma_1$ with respect to this trivialization is equal to $0$. 
Thus, by \cite[Proposition 12]{APZ2013LFoverdisc}, the value $\left<c_1(W''),\alpha_0''\right>$ is also equal to $0$.

Let $\widetilde{F}:\widetilde{X}\to \Hir_1$ be the Lefschetz bifibration obtained by blowing up $F:X\dashrightarrow \PP^2$ at the eight base points, and then along a regular fiber on $[0:0:1]\in \PP^2$, and $\widetilde{W}$ be a $4$-dimensional fiber of $\widetilde{F}$. 
We can then regard $W'$ as a subset of $\widetilde{W}$. 
Let $i:W''\to \widetilde{W}$ be the inclusion, $\alpha_0 = i_\ast \alpha_0''\in H_2(\widetilde{W})$, and $\widetilde{A}_0\in H_2(\widetilde{X})$ be the homology class corresponding to ${\alpha}_0$ via the isomorphism in \Cref{prop:H2 LbF} (i.e.~$\widetilde{A}_0 = j_\ast \alpha_0$ for the inclusion $j:\widetilde{W}\to \widetilde{X}$). 
By \Cref{prop:value 1st Chern}, the value $\left<c_1(\widetilde{X}),\widetilde{A}_0\right>$ is equal to $\left<c_1(\widetilde{W}),\alpha_0\right>$, which is further calculated as follows:
\[
\left<c_1(\widetilde{W}),\alpha_0\right> = \left<i^\ast c_1(\widetilde{W}),\alpha_0''\right> = \left<c_1(W''),\alpha_0''\right>=0. 
\] 
On the other hand, the following equality holds by \cite[\S8]{GP2007Chernblowup}:
\[
c_1(\widetilde{X}) = \pi^\ast c_1(X) -2 \sum_{i=1}^{8}[\mathcal{S}_i] - 2 [\widetilde{F}^{-1}(S_1)], 
\]
where $\pi:\widetilde{X}\to X$ is the blow-down map. 
Since $\Pa_2(\alpha_0'') = c_7+c_8-c_9+c_{10}-c_{11}-c_{12}$ is zero in $H_1(\Sigma_1^7)$, and $-2\delta_1$ in $H_1(\Sigma_1^8)$, One can take a properly embedded surface in $T\subset D^2\times \Sigma_1^7\cong W_1'\subset W''$ with the following properties:

\begin{itemize}

\item 
$T$ bounds the union of the attaching circles of the $2$-handles $h_7,\ldots, h_{12}$ corresponding to $e_7,\ldots, e_{12}$,

\item 
the union of $T$ and the cores of the handles $h_7,\ldots, h_{12}$ consist of the closed oriented surface $\widetilde{T}$ representing $\alpha_0''$, 

\item 
the algebraic intersection of $\widetilde{T}$ with the $(-1)$-section is $2$. 

\end{itemize} 

\noindent
Since $\widetilde{T}$ is contained in $W''$, it is disjoint from a regular fiber of $\widetilde{T}$. 
We can eventually calculate $\left<c_1(X),A_0\right>$ as follows: 
{\allowdisplaybreaks
\begin{align*}
&\left<c_1(X),A_0\right>\\
=&\left<\pi^\ast c_1(X),\widetilde{A}_0\right>\\
=&\left<c_1(\widetilde{X}),\widetilde{A}_0\right> -2 \sum_{i=1}^{8}[\mathcal{S}_i]\cdot [\widetilde{T}] - 2 [\widetilde{F}^{-1}(S_1)]\cdot [\widetilde{T}]\\
=&0 -2 \cdot 2 - 2 \cdot 0 = -4.
\end{align*}
}%
In particular, $c_1(X)$ is four times a generator of $H^2(X)$. 

\begin{remark}\label{rem:example is deg2 Veronese}

The paths $\mu_i$ and the simple closed curves $c_i$ are obtained by adapting the methods from Moishezon-Teicher's theory for projective surfaces (see \cite{MTII,MTIII,MTIV}) to the degree-$2$ Veronese embedding $\PP^3 \hookrightarrow \PP^9$. 
Specifically, this involves taking a sequence of degenerations, which transform the embedded $\PP^3 \subset \PP^9$ into a union of linear $\PP^3$'s, followed by analyzing the changes in (braid and fiber) monodromies induced by each degeneration, referred to as \textit{regeneration rules} in \cite{MTIV}.
While several lemmas necessary for this procedure, particularly those concerning the determination of regeneration rules, remain unproved, the following conjecture arises naturally from this approach:

\begin{conjecture}

The Lefschetz bipencil constructed in this section is isomorphic to the composition of the degree-$2$ Veronese embedding $\PP^3 \hookrightarrow \PP^9$ and a generic projection $\PP^9 \dashrightarrow \PP^2$.

\end{conjecture}

\noindent The proofs of the necessary lemmas and the above conjecture will be provided in a forthcoming paper.

\end{remark}

\begin{remark}\label{rem:example null-hom irre comp}

By \Cref{thm:relation for deg2 Veronese}, we obtain the following factorization in $B_{12}\cong \mathcal{M}_\Pa(D^2,\R)$: 
\begin{equation}\label{eqn:example null-hom irre comp}
\tau_{\mu_1}^{\epsilon_1}\cdots \tau_{\mu_{56}}^{\epsilon_{56}}\cdot \tau_{\mu_1}\cdot \tau_{\mu_1} \cdot\tau_{\mu_2}^{\epsilon_2}\cdots \tau_{\mu_{56}}^{\epsilon_{56}} = t_{\Pa}^2. 
\end{equation}
Let $C:\mathcal{M}_\Pa(\Sigma_1^8)\to \mathcal{M}_\Pa(\Sigma_1^6)$ be the surjective homomorphism obtained by capping the two boundary components of $\Sigma_1^8$ labeled with $1$ and $5$ in \Cref{fig:surf for symmetry 2} by the disk.  
We regard $\Sigma_1^8$ as a subsurface of $\Sigma_1^6$ in the obvious way. 
Since $c_2$ and $c_8$ are isotopic in $\Sigma_1^6$, the homomorphism $C\circ \theta$ is compatible with the factorization \eqref{eqn:example null-hom irre comp} (cf.~\Cref{tab:vanishing cycles}). 
By \Cref{thm:combinatorial construction LbF over Hir_m}, we can obtain a Lefschetz bifibration $F:X\to \Hir_2$ over $\rho_2$ with $\# B_F = 4\times 2 + 2 =10$.  
This Lefschetz bifibration has a spherical irreducible component $\Sigma$ in a fiber corresponding to the first factor $\tau_{\mu_1}^{\epsilon_1=2}$ in \eqref{eqn:example null-hom irre comp}.
On the other hand, one can easily check that $\Sigma$ is isotopic in $X$ to the vanishing cycle of the Lefschetz fibration $\rho_2\circ F$ associated with the critical point corresponding to the middle factor $\tau_{\mu_1}$ in \eqref{eqn:example null-hom irre comp}. 
In particular, $\Sigma$ is null-homologous and thus, one cannot take a symplectic structure of $X$ taming $J$ in \Cref{thm:almost cpx str LbF} for $F$. (Note that each irreducible component of $F$ is $J$-holomorphic.)

\end{remark}

%


\begin{thebibliography}{10}

\bibitem{AM2020genus2LF}
Anar Akhmedov and Naoyuki Monden.
\newblock Genus 2 {L}efschetz fibrations with {$b^+_2=1$} and {$c_1^2=1,2$}.
\newblock {\em Kyoto J. Math.}, 60(4):1419--1451, 2020.

\bibitem{APZ2013LFoverdisc}
Nikos Apostolakis, Riccardo Piergallini, and Daniele Zuddas.
\newblock Lefschetz fibrations over the disc.
\newblock {\em Proc. Lond. Math. Soc. (3)}, 107(2):340--390, 2013.

\bibitem{Auroux2000symp4mfdasbranchedcover}
Denis Auroux.
\newblock Symplectic 4-manifolds as branched coverings of {$\bold C\bold P^2$}.
\newblock {\em Invent. Math.}, 139(3):551--602, 2000.

\bibitem{AurouxSympplecticMapsInvariant}
Denis Auroux.
\newblock Symplectic maps to projective spaces and symplectic invariants.
\newblock {\em Turkish J. Math.}, 25(1):1--42, 2001.

\bibitem{AurouxKatzarkovbranchedcoversymp4mfd}
Denis Auroux and Ludmil Katzarkov.
\newblock Branched coverings of {${\bf C}{\rm P}^2$} and invariants of
  symplectic 4-manifolds.
\newblock {\em Invent. Math.}, 142(3):631--673, 2000.

\bibitem{BaykurHayanoHuriwitzMultisection}
R.~Inan\c{c} Baykur and Kenta Hayano.
\newblock Hurwitz equivalence for {L}efschetz fibrations and their
  multisections.
\newblock In {\em Real and complex singularities}, volume 675 of {\em Contemp.
  Math.}, pages 1--24. Amer. Math. Soc., Providence, RI, 2016.

\bibitem{BH2016multisection}
R.~\.Inan\c{c} Baykur and Kenta Hayano.
\newblock Multisections of {L}efschetz fibrations and topology of symplectic
  4-manifolds.
\newblock {\em Geom. Topol.}, 20(4):2335--2395, 2016.

\bibitem{Donaldson1999LPonsympmfd}
S.~K. Donaldson.
\newblock Lefschetz pencils on symplectic manifolds.
\newblock {\em J. Differential Geom.}, 53(2):205--236, 1999.

\bibitem{duPW1991rightsym}
Andrew du~Plessis and Leslie~Charles Wilson.
\newblock Right-symmetry of mappings.
\newblock In {\em Singularity theory and its applications, {P}art {I}
  ({C}oventry, 1988/1989)}, volume 1462 of {\em Lecture Notes in Math.}, pages
  258--275. Springer, Berlin, 1991.

\bibitem{Dyer1969cohomology}
Eldon Dyer.
\newblock {\em Cohomology theories}.
\newblock Mathematics Lecture Note Series. W. A. Benjamin, Inc., New
  York-Amsterdam, 1969.

\bibitem{ES1970homotopytypediffeogrp}
C.~J. Earle and A.~Schatz.
\newblock Teichm\"{u}ller theory for surfaces with boundary.
\newblock {\em J. Differential Geometry}, 4:169--185, 1970.

\bibitem{EE1969homotopytypediffeogrp}
Clifford~J. Earle and James Eells.
\newblock A fibre bundle description of {T}eichm\"{u}ller theory.
\newblock {\em J. Differential Geometry}, 3:19--43, 1969.

\bibitem{EMV2011substitution}
Hisaaki Endo, Thomas~E. Mark, and Jeremy Van Horn-Morris.
\newblock Monodromy substitutions and rational blowdowns.
\newblock {\em J. Topol.}, 4(1):227--253, 2011.

\bibitem{FarbMargalitBookMCG}
Benson Farb and Dan Margalit.
\newblock {\em A primer on mapping class groups}, volume~49 of {\em Princeton
  Mathematical Series}.
\newblock Princeton University Press, Princeton, NJ, 2012.

\bibitem{FK2006Thompolysecondorder}
L\'aszl\'o{}~M. Feh\'er and Bal\'azs K\"om\"uves.
\newblock On second order {T}hom-{B}oardman singularities.
\newblock {\em Fund. Math.}, 191(3):249--264, 2006.

\bibitem{FduP2011str_multigerm_eq_preprint}
Aasa Feragen and Andrew du~Plessis.
\newblock The structure of groups of multigerm equivalences.
\newblock {\em arXiv e-prints}, 2011.
\newblock arXiv:1110.1981.

\bibitem{GP2007Chernblowup}
Hansj\"org Geiges and Federica Pasquotto.
\newblock A formula for the {C}hern classes of symplectic blow-ups.
\newblock {\em J. Lond. Math. Soc. (2)}, 76(2):313--330, 2007.

\bibitem{GompfTowardTopcharasympmfd}
Robert~E. Gompf.
\newblock Toward a topological characterization of symplectic manifolds.
\newblock {\em J. Symplectic Geom.}, 2(2):177--206, 2004.

\bibitem{Gompf2005localholyieldsymp}
Robert~E. Gompf.
\newblock Locally holomorphic maps yield symplectic structures.
\newblock {\em Comm. Anal. Geom.}, 13(3):511--525, 2005.

\bibitem{GS19994-mfd}
Robert~E. Gompf and Andr\'as~I. Stipsicz.
\newblock {\em {$4$}-manifolds and {K}irby calculus}, volume~20 of {\em
  Graduate Studies in Mathematics}.
\newblock American Mathematical Society, Providence, RI, 1999.

\bibitem{HamadaHayano}
Noriyuki Hamada and Kenta Hayano.
\newblock Topology of holomorphic {L}efschetz pencils on the four-torus.
\newblock {\em Algebr. Geom. Topol.}, 18(3):1515--1572, 2018.

\bibitem{HH2021genus1holLP}
Noriyuki Hamada and Kenta Hayano.
\newblock Classification of genus-1 holomorphic {L}efschetz pencils.
\newblock {\em Turkish J. Math.}, 45(3):1079--1119, 2021.

\bibitem{Hirzebruch1995topmethod}
Friedrich Hirzebruch.
\newblock {\em Topological methods in algebraic geometry}.
\newblock Classics in Mathematics. Springer-Verlag, Berlin, english edition,
  1995.
\newblock Translated from the German and Appendix One by R. L. E.
  Schwarzenberger, Appendix Two by A. Borel, Reprint of the 1978 edition.

\bibitem{Ivanov1998MCG}
Nikolai~V. Ivanov.
\newblock Mapping class groups.
\newblock In {\em Handbook of geometric topology}, pages 523--633.
  North-Holland, Amsterdam, 2002.

\bibitem{Kas1980handlebodyLF}
A.~Kas.
\newblock On the handlebody decomposition associated to a {L}efschetz
  fibration.
\newblock {\em Pacific J. Math.}, 89(1):89--104, 1980.

\bibitem{KorkmazOzbagci2008}
Mustafa Korkmaz and Burak Ozbagci.
\newblock On sections of elliptic fibrations.
\newblock {\em Michigan Math. J.}, 56(1):77--87, 2008.

\bibitem{LRZ2022sympblowdowndim6}
Tian-Jun Li, Yong~Bin Ruan, and Wei~Yi Zhang.
\newblock Symplectic blowing down in dimension six.
\newblock {\em Acta Math. Sin. (Engl. Ser.)}, 38(10):1831--1855, 2022.

\bibitem{Matsumoto1996LFgenus2}
Yukio Matsumoto.
\newblock Lefschetz fibrations of genus two---a topological approach.
\newblock In {\em Topology and {T}eichm\"{u}ller spaces ({K}atinkulta, 1995)},
  pages 123--148. World Sci. Publ., River Edge, NJ, 1996.

\bibitem{Meyer1973sigcocycle}
Werner Meyer.
\newblock Die {S}ignatur von {F}l\"achenb\"undeln.
\newblock {\em Math. Ann.}, 201:239--264, 1973.

\bibitem{MTII}
Boris Moishezon and Mina Teicher.
\newblock Braid group technique in complex geometry. {II}. {F}rom arrangements
  of lines and conics to cuspidal curves.
\newblock In {\em Algebraic geometry ({C}hicago, {IL}, 1989)}, volume 1479 of
  {\em Lecture Notes in Math.}, pages 131--180. Springer, Berlin, 1991.

\bibitem{MTIII}
Boris Moishezon and Mina Teicher.
\newblock Braid group techniques in complex geometry. {III}. {P}rojective
  degeneration of {$V_3$}.
\newblock In {\em Classification of algebraic varieties ({L}'{A}quila, 1992)},
  volume 162 of {\em Contemp. Math.}, pages 313--332. Amer. Math. Soc.,
  Providence, RI, 1994.

\bibitem{MTIV}
Boris Moishezon and Mina Teicher.
\newblock Braid group techniques in complex geometry. {IV}. {B}raid monodromy
  of the branch curve {$S_3$} of {$V_3\to\bold C{\rm P}^2$} and application to
  {$\pi_1(\bold C{\rm P}^2-S_3,*)$}.
\newblock In {\em Classification of algebraic varieties ({L}'{A}quila, 1992)},
  volume 162 of {\em Contemp. Math.}, pages 333--358. Amer. Math. Soc.,
  Providence, RI, 1994.

\bibitem{Porteous1971SimpleSing}
I.~R. Porteous.
\newblock Simple singularities of maps.
\newblock In {\em Proceedings of {L}iverpool {S}ingularities---{S}ymposium, {I}
  (1969/70)}, volume Vol. 192 of {\em Lecture Notes in Mathematics}, pages pp
  286--307. Springer, Berlin, 1971.

\bibitem{SeidelbookFukayaPicard}
Paul Seidel.
\newblock {\em Fukaya categories and {P}icard-{L}efschetz theory}.
\newblock Zurich Lectures in Advanced Mathematics. European Mathematical
  Society (EMS), Z\"{u}rich, 2008.

\bibitem{Torricelli2020generalizedlantern_preprint}
Brunella~Charlotte Torricelli.
\newblock Model projective twists and generalised lantern relations.
\newblock {\em arXiv e-prints}, 2020.
\newblock arXiv:2008.02758.

\bibitem{Wall1980secondnotesym}
C.~T.~C. Wall.
\newblock A second note on symmetry of singularities.
\newblock {\em Bull. London Math. Soc.}, 12(5):347--354, 1980.

\end{thebibliography}
\end{document}